\documentclass[reqno]{amsart}
\usepackage[left=1in,right=1in,top=1in,bottom=1in]{geometry}
\usepackage{microtype}
\usepackage{hyperref}

\usepackage{amsmath}             
\usepackage{amsfonts}             
\usepackage{amsthm}               
\usepackage{amsbsy}
\usepackage{amssymb}
\usepackage{amsthm}
\usepackage[nobysame]{amsrefs}
\usepackage{mathrsfs}

\usepackage{enumerate}

\usepackage{calligra}
\usepackage{stmaryrd}
\usepackage{relsize}

\usepackage{tikz}
\usetikzlibrary{shapes,snakes,calc,arrows}
\usetikzlibrary{decorations.pathreplacing,shapes.geometric}
\usetikzlibrary{calc,positioning}
\usetikzlibrary{matrix,decorations.pathreplacing,calc,positioning,calligraphy}
\tikzset{Box/.style={very thick, rounded corners}}
\tikzset{marked/.style={star, star point height = .75mm, star points =5, fill=black,minimum size=2mm, inner sep=0mm} }
\tikzset{verythickline/.style = {line width=7pt}}
\tikzset{thickline/.style = {line width=5pt}}
\tikzset{medthick/.style = {line width=3pt}}
\tikzset{med/.style = {line width=2pt}}
\tikzset{count/.style = {fill=white,circle,draw,thin, inner sep=2pt}}
\tikzset{rcount/.style = {fill=white,rectangle,draw,thin,inner sep=2pt, rounded corners}}
\tikzset{cpr/.style = {draw,fill=white,rectangle,thin, rounded corners}}

\theoremstyle{plain}
	\newtheorem{theorem*}{Theorem}
	\newtheorem{theorem}{Theorem}[section]

	\newtheorem{proposition}[theorem]{Proposition}
	\newtheorem{lemma}[theorem]{Lemma}
	
	\newtheorem{corollary}[theorem]{Corollary}

\theoremstyle{definition}	\newtheorem{definition}[theorem]{Definition}
	\newtheorem{notation}[theorem]{Notation}
	\newtheorem{example}[theorem]{Example}

\theoremstyle{remark}
	\newtheorem{remark}[theorem]{Remark}

\DeclareMathOperator{\iden}{I}

\DeclareMathOperator{\NC}{NC}
\DeclareMathOperator{\BNC}{BNC}
\DeclareMathOperator{\tr}{tr}
\DeclareMathOperator{\alg}{alg}

\DeclareMathOperator{\Krew}{K}

\newcommand{\C}{\mathbb{C}}

\newcommand{\N}{\mathbb{N}}
\newcommand{\Z}{\mathbb{Z}}

\newcommand{\sx}{s_{\chi}}
\newcommand{\sxh}{s_{\widehat{\chi}}}
\newcommand{\sxv}{s_{{\chi}|_V}}

\newcommand{\xv}{\chi|_V}

\newcommand{\xh}{\widehat{\chi}}
\newcommand{\xhv}{\xh|_V}
\newcommand{\oxh}{\widehat{0_{\chi}}}

\renewcommand{\phi}{\varphi} 

\newcounter{my_enumerate_counter}
\newcommand{\pushcounter}{\setcounter{my_enumerate_counter}{\value{enumi}}}
\newcommand{\popcounter}{\setcounter{enumi}{\value{my_enumerate_counter}}}

\title{On Bi-R-Diagonal Pairs of Operators}

\author{Georgios Katsimpas}
\address{School of Mathematical Sciences, Harbin Engineering University, 145 Nantong Street, Nangang District, Harbin, P. R. China}
\email{gkats@hrbeu.edu.cn}

\begin{document}

\begin{abstract}
We study the properties of the analogue of R-diagonal operators in the setting
of bi-free probability. Products of bi-R-diagonal pairs of operators that are $*$-bi-free are studied and powers of such pairs are found to also be bi-R-diagonal. It
is moreover shown that the joint $*$-distribution of a bi-R-diagonal pair of operators
remains invariant under the multiplication by a $*$-bi-free bi-Haar unitary pair and
equivalent characterizations of  bi-R-diagonal pairs are developed.
\end{abstract}

\maketitle

\section{Introduction}\label{sec:intro}
In the theory of free probability, an R-diagonal operator is an element of a non-commutative $*$-probability space $(A,\varphi)$ whose $*$-distribution coincides with the $*$-distribution of a  product of the form $u\cdot p$, where the sets $\{u,u^*\}$ and $\{p,p^*\}$ are freely independent and $u$ is a Haar unitary, i.e. $u$ is a unitary and $\varphi(u^n)=0$, for all $n\in\Z\setminus\{0\}$. It is due to this free factorization property that the class of R-diagonal operators constitutes a particularly well-behaved class of non-normal operators.  From a combinatorial point of view, R-diagonal elements are characterized by having all of their free $*$-cumulants that are either of odd order, or have entries that are not alternating in $*$-terms and non-$*$-terms equal to zero. This combinatorial approach has proved to be extremely fruitful in the development of the theory of R-diagonal operators (see \cite{nicaspeicher} for an exposition of the combinatorics of free probability).

In \cite{approach}, R-diagonal operators were found to satisfy a ``free absorption'' property, namely that for any elements $a,b$ in some non-commutative $*$-probability space such that $a$ is R-diagonal and $a$ is $*$-free from $b$, the element $ab$ is also R-diagonal.  In \cite{haagerup}, Brown's spectral distribution measure was computed for R-diagonal operators in finite von Neumann algebras, while in \cite{larsenpowers}, powers of R-diagonal operators were shown to be R-diagonal and their determining sequences were computed (see also \cite[Theorem 15.22] {nicaspeicher} for a proof making use of combinatorial arguments).

In \cite{nicaspeichersh}, a number of equivalent characterizations of  R-diagonality were formulated, including conditions on $*$-moments, free cumulants and the freeness of certain self-adjoint matrices from the scalar matrices, with amalgamation over the diagonal scalar matrices, while in \cite{dykema} similar results were obtained on $B$-valued R-diagonal elements in the operator-valued setting.
Distributions of R-diagonal operators have found applications in the non-microstate approach to free entropy, answering questions regarding the minimization of the free Fisher information in the tracial framework (see \cite{fisher}). 

Bi-free probability theory originated in \cite{voiculescu} as an extension of the free setting and involves the simultaneous study of left and right actions of algebras on reduced free product spaces.  The corresponding notion of bi-free independence found its  combinatorial characterization in \cite{CNS2015-2} (see also \cite{CNS2015-1} for the development of the combinatorics of bi-free probability in the operator-valued setting). This paper is devoted to the study of the analogue of R-diagonal operators in the bi-free setting, namely bi-R-diagonal pairs of operators and, to this end,  the combinatorial approach originally proposed in \cite[Section 4]{S2016-2} shall be adopted, which makes use of the bi-free cumulant functions. For the study of products and powers of bi-R-diagonal pairs, similar arguments are used as to those corresponding to the results in the free case, but more care is required due to the dealing with the lattice of bi-non-crossing partitions and the $\chi$-order. Since products of pairs of operators are considered pointwise (i.e. left operators are multiplied by left operators and right operators are multiplied by right operators), caution ought to be exercised when it comes to the order in which the multiplication takes place and, for the most general cases, it is necessary that the order of the multiplication of right operators is reversed (see Theorem \ref{prodbir-bifree}). However, this is found not to play a role in the case when both  pairs in question are bi-R-diagonal and $*$-bi-free (Proposition \ref{prodbir-bir}). These results imply that bi-R-diagonal pairs of operators satisfy a  corresponding ``bi-free absorption'' property and indicate that such pairs of operators exist in abundance.  Powers of bi-R-diagonal pairs  (with equal exponents on the left and right operators) are shown to remain bi-R-diagonal (Theorem \ref{birpowers}) and various nuances in the case of distinct exponents are discussed and examined (Proposition \ref{prop:bi-R-diagpowerscounterexample}, Remark \ref{remark: every power is bi-R-diag}). Joint distributions of pairs of R-diagonal operators are also studied in the context of free products and tensor products, and it is shown that classically independent R-diagonal operators give rise to bi-R-diagonal pairs, which comes in contrast to the case of freely independent R-diagonal operators (Example \ref{example:freeR-diagnotbi-R-diag}  and Proposition \ref{prop:tensorproductofR-diag}). These results highlight the differences between R-diagonal operators in the free and bi-free settings and add to the richness of the structure of bi-R-diagonal pairs.

The absence of characterizations of bi-free phenomena with conditions on moments is an unfortunate theme in the theory of bi-free probability (see, however, \cite{C2019} for an equivalent formulation of bi-free independence in terms of alternating moments). In particular, a characterization of the condition of bi-R-diagonality in terms of $*$-moments was unable to be obtained. In the setting of free probability, one of the most salient features of the $*$-distribution of an R-diagonal operator is that it remains invariant after the multiplication by a freely independent Haar unitary, a result obtained with the use of freeness in terms of its characterization via moments (see \cite[Theorem 1.2]{nicaspeichersh} and \cite[Theorem 15.10]{nicaspeicher}). Bi-Haar unitary pairs of operators constitute the bi-free analogue of Haar unitaries and their joint $*$-distribution is modelled by the left and right regular representations of groups on Hilbert spaces. Theorem \ref{invdistr} is the generalization of the aforementioned fact to the bi-free setting and displays the invariance of the joint $*$-distribution of any bi-R-diagonal pair of operators under the multiplication of a $*$-bi-free bi-Haar unitary pair. The proof follows the combinatorial approach instead, using the bi-free cumulant functions and hence a new proof follows for the free case as well. In the spirit of \cite[Theorem 1.2]{nicaspeichersh},  \cite[Theorem 3.1]{dykema} and by combining results from \cite{S2016-2}, we obtain Theorem \ref{together}, displaying equivalent formulations of the condition of bi-R-diagonality.

The paper is organized as follows: In Section \ref{sctn:preliminaries} we list all the necessary preliminary
notions on bi-free probability theory and fix the appropriate notation. Here, the notion
of a bi-R-diagonal pair of operators is defined and a number of auxiliary results that will be used
in subsequent parts of this manuscript will be stated and proved. In Section \ref{sctn:products} we investigate the behaviour of bi-R-diagonal pairs under the taking of sums, products and
arbitrary powers, and we study joint distributions of pairs formed by freely and classically independent R-diagonal operators,  while Section \ref{sctn:distributions} is devoted to showing that the joint $*$-distributions of
bi-R-diagonal pairs remain invariant under the multiplication by $*$-bi-free, bi-Haar unitary
pairs. Most non-trivial combinatorial arguments are accompanied by representative diagrams for the convenience of the reader and  examples and counter-examples are provided whenever possible for the sake of completeness and clarity.

\section{Preliminaries and Notation}\label{sctn:preliminaries}
In this section we will develop the common preliminaries, fix the appropriate notation and state a number of lemmas to be used later in this paper.

Our main framework will be that of a \emph{non-commutative probability space}, i.e. a pair $(A,\varphi)$ where $A$ is a complex, unital algebra and $\varphi:A\rightarrow\C$ is a unital, linear map. If $A$ is a unital $*$-algebra and $\varphi:A\rightarrow\C$ is a additionally assumed to be positive, i.e. $\varphi$ satisfies that $\varphi (a^*a)\geq 0$ for all $a\in A$, then the pair $(A,\varphi)$ will be called a \emph{non-commutative $*$-probability space}. In addition, if $A$ is a unital $\mathrm{C}^*$-algebra and $\varphi:A\rightarrow\C$ is a state on $A$, then the pair $(A,\varphi)$ will be called a \emph{$\mathrm{C}^*$-probability space}.

 If $(A,\varphi)$ is a non-commutative probability space, for any set $S\subseteq A$ we will denote by $\alg(S)$ the subalgebra of $A$ generated by the set $S$. If $a_1,\ldots,a_n$ are elements of  $(A,\varphi)$, then their \emph{joint distribution} is given by the linear functional
\[
\mu:\C\langle X_1,\ldots,X_n\rangle\rightarrow\C
\]
defined as
\[
\mu (P) = \varphi(P(a_1,\ldots,a_n)),\quad (P\in \C\langle X_1,\ldots,X_n\rangle)
\]
where $\C\langle X_1,\ldots,X_n\rangle$ denotes the unital algebra of polynomials in $n$ non-commuting indeterminates $X_1,\ldots,X_n$.

If $(A,\varphi)$ is a non-commutative $*$-probability space and $S\subseteq A$ then we define 
$S^*=\{a^*:a\in S\}$. If $a_1,\ldots,a_n\in A$, then:
\begin{enumerate}[(a)]
\item  their \emph{joint  $*$-distribution} is given by the joint distribution of the family 
\[
\{a_1,\ldots,a_n,a_1^*,\ldots,a_n^*\},
\]
\item the family of their \emph{joint $*$-moments} is given by the action of their joint $*$-distribution $\mu $ on the monomials in $\C\langle X_1,\ldots,X_n\rangle$ and is therefore determined by the collection 
\[
\{\varphi(c_1\cdots c_k) : k\geq 1, c_i\in\{a_1,\ldots,a_n,a_1^*,\ldots,a_n^*\}  \text{ for all }1\leq i\leq k\}.
\]
\end{enumerate}
It is clear that for $a_1,\ldots,a_n,b_1,\ldots,b_n\in A$, in order to verify equality of joint $*$-distributions of the families $\{a_1,\ldots,a_n\}$ and $\{b_1,\ldots,b_n\}$, it suffices to prove that all of their joint $*$-moments coincide. For $a_1,\ldots,a_n\in A$ and $\emptyset\neq V=\{j_1<j_2<\ldots<j_s\}\subseteq\{1,\ldots,n\}$, the restriction of the sequence $(a_1,\ldots,a_n)$ to the set $V$ is given by
\[
(a_1,\ldots,a_n)|_V = (a_{j_1},a_{j_2},\ldots,a_{j_s}).
\]
In this case, we define
\[
\varphi((a_1,\ldots,a_n)|_V) = \varphi(a_{j_1}\cdot a_{j_2}\cdot\ldots\cdot a_{j_s}).
\]
Also, if $\pi$ is a partition of the set $\{1,\ldots,n\}$, then we use the following notation:
\[
\varphi_{\pi}(a_1,\ldots,a_n) = \prod_{V\in\pi}\varphi((a_1,\ldots,a_n)|_V).
\]
\subsection{The Lattice of Bi-Non-Crossing Partitions}\label{sctn:bnc}
Familiarity with the collection of non-crossing partitions $\NC(n)$, multiplicative functions on $\NC(n)$ and free cumulants is assumed  (see \cite{nicaspeicher} for an exposition of the combinatorics of free probability).

For $n\in\N$, we will be using maps $\chi\in {\{l,r\}}^n$ to distinguish between left and right operators in a sequence of $n$ operators. Any such  map gives rise to a permutation $\sx$ on $\{1,\ldots,n\}$ as follows:  if ${\chi}^{-1}(\{l\}) = \{i_1<\ldots <i_p\}$ and ${\chi}^{-1}(\{r\}) = \{j_1<\ldots <j_{n-p}\}$, then we define
\[   
\sx(k) = 
     \begin{cases}
       i_k, &\text{if } k\leq p\\
       j_{n+1-k}, &\text{if } k>p. 
     \end{cases}
\]
From a combinatorial standpoint, the only differences between free and bi-free probability arise from dealing with $\sx$.  The permutation $\sx$ naturally induces a total order on $\{1,\ldots,n\}$ (which we will henceforth be referring to as the \textit{$\chi $-order}) as follows:
\[
i \prec_\chi j \iff \sx^{-1}(i)<  \sx^{-1}(j).
\]
Instead of reading $\{1,\ldots,n\}$ in the traditional order, this corresponds to first reading the elements of $\{1,\ldots,n\}$ labelled ``$l$'' in increasing order, followed by reading the elements labelled ``$r$'' in decreasing order. Note that if $V$ is any non-empty subset of $\{1,\ldots,n\}$, the map $\xv$ naturally gives rise to a map $\sxv$, which should be thought of as a permutation on $\{1,\ldots,|V|\}$. A set $I\subseteq\{1,\ldots,n\}$ is called a \emph{$\chi$-interval} if $I$ is an interval with respect to the $\chi$-order. Note that any $\chi$-interval is of the form $\{\sx(i),\sx(i+1),\ldots,\sx(i+j)\}$ for $1\leq i\leq n$ and  $0\leq j\leq n-i$.

Before we discuss the lattice of bi-non-crossing partitions, we fix some notation regarding general partitions. For $n\in\N$, the collection of all partitions on $\{1,\ldots,n\}$ is denoted by $\mathcal{P}(n)$, while the collection of non-crossing partitions on $\{1,\ldots,n\}$ is denoted by $\NC(n)$. The elements of any $\pi\in \mathcal{P}(n)$ are called the \textit{blocks} of $\pi$ and for $1\leq i,j\leq n$, we write $i {\sim}_{\pi} j$ to mean that $i$ and $j$ belong to the same block of $\pi$, whereas $i \nsim_\pi j$ indicates that $i$ and $j$ belong to different blocks of $\pi$. For $\pi,\sigma\in\mathcal{P}(n)$, we write $\pi\leq\sigma$ if every block of $\pi$ is contained in a block of $\sigma$. This defines the partial order of refinement on $\mathcal{P}(n)$. The maximal element of $\mathcal{P}(n)$ with respect to this partial order is the partition consisting of one block (denoted by $1_n$), while the minimal element is the partition consisting of $n$ blocks (denoted by $0_n$). This partial order induces a lattice structure on $\mathcal{P}(n)$, hence for $\pi,\sigma\in\mathcal{P}(n)$, the \emph{join} $\pi\vee \sigma$ (i.e. the minimum element of the non-empty set $\{\rho\in\mathcal{P}(n) : \rho\geq\pi,\sigma\}$) of $\pi$ and $\sigma$ is well defined.

\begin{definition}\label{BNC}
Let $n\in\N$ and $ \chi\in {\{l,r\}}^n$. A partition $\tau\in\mathcal{P}(n)$ is called \textit{bi-non-crossing with respect to $\chi$} if the partition $\sx^{-1}\cdot\tau$ (i.e. the partition obtained by applying the permutation $\sx^{-1}$ to each entry of every block of $\tau$) is non-crossing. Equivalently, $\tau$ is bi-non-crossing with respect to $\chi$ if whenever $V,W$ are blocks of $\tau$ and $v_1,v_2\in V , w_1,w_2\in W$ are such that
\[
v_1 \prec_\chi w_1 \prec_\chi v_2 \prec_\chi w_2,
\]
then we necessarily have that $V=W$.  The collection of bi-non-crossing partitions with respect to $\chi$ is denoted by $\BNC(\chi)$. It is clear that
\[
\BNC(\chi) = \{\tau\in\mathcal{P}(n) : \sx^{-1}\cdot\tau\in \NC(n)\} = \{\sx\cdot\pi : \pi\in \NC(n)\}.
\]
\end{definition}
We will be referring to a partition $\tau$ simply as bi-non-crossing whenever it is clear from the context  which map $\chi$ is used. Note that in the special case when the map $\chi$ is constant, one ends up with the collection of all non-crossing partitions on $\{1,\ldots,n\}$.

\begin{example}
If $\chi\in {\{l,r\}}^6$ is such that $\chi^{-1}(\{l\}) = \{1,2,3,6\}$ and $\chi^{-1}(\{r\}) = \{4,5\}$, then $(\sx(1),\ldots,\sx(6)) = (1,2,3,6,5,4)$ and the partition given by 
\[
\tau = \{\{1,4\}, \{2,5\}, \{3,6\}\}
\]
is bi-non-crossing with respect to $\chi$, even though $\tau\notin \NC(6)$. This may also be seen via the following diagrams:

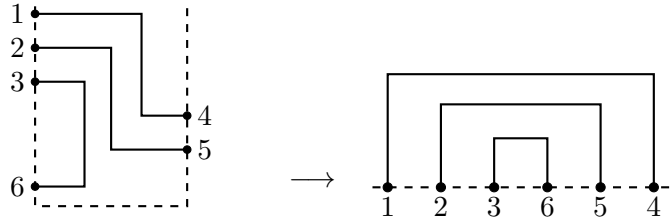
\begin{equation*}
\begin{tikzpicture}[baseline]
			\draw[thick,dashed] (-.5,2.4) -- (-.5,-.25) -- (1.5,-.25) -- (1.5,2.4);
\node[left] at (-.5,2.3) {1};
			\draw[black,fill=black] (-.5,2.3) circle (0.05);
\node[left] at (-.5,1.85) {2};
			\draw[black,fill=black] (-.5,1.85) circle (0.05);
\node[left] at (-.5,1.4) {3};
			\draw[black,fill=black] (-.5,1.4) circle (0.05);
\node[right] at (1.5,0.95) {4};
			\draw[black,fill=black] (1.5,0.95) circle (0.05);
\node[right] at (1.5,0.5) {5};
			\draw[black,fill=black] (1.5,0.5) circle (0.05);
\node[left] at (-.5,0.01) {6};
			\draw[black,fill=black] (-.5,0.01) circle (0.05);
			\draw[thick, black] (-.5,2.3) -- (0.9,2.3) -- (0.9,0.95) -- (1.5,0.95);
			\draw[thick, black] (-.5,1.85) -- (0.5,1.85) -- (0.5,0.5) -- (1.5,0.5);
			\draw[thick, black] (-.5,1.4) -- (0.15,1.4) -- (0.15, 0.01) -- (-.5,0.01);
	\end{tikzpicture}
	\qquad\longrightarrow	\qquad
\begin{tikzpicture}[baseline]
			\draw[thick,dashed] (5,0) -- (9,0);
\filldraw (5.2,0) circle (0.06) node[anchor=north] {1};
\filldraw (5.9,0) circle (0.06) node[anchor=north] {2};
\filldraw (6.6,0) circle (0.06) node[anchor=north] {3};
\filldraw (7.3,0) circle (0.06) node[anchor=north] {6};
\filldraw (8,0) circle (0.06) node[anchor=north] {5};
\filldraw (8.7,0) circle (0.06) node[anchor=north] {4};	
\draw[thick, black] (5.2,0) -- (5.2,1.5) -- (8.7,1.5) -- (8.7,0);	
\draw[thick, black] (5.9,0) -- (5.9,1.1) -- (8,1.1) -- (8,0);	
\draw[thick, black] (6.6,0) -- (6.6,0.65) -- (7.3,0.65) -- (7.3,0);	
			\end{tikzpicture}
			\qquad
				\end{equation*}

\end{example}
The set of bi-non-crossing partitions with respect to a map $\chi\in {\{l,r\}}^n$ inherits a lattice structure from $\mathcal{P}(n)$ via the partial order of refinement (although the join operation in $\BNC(\chi)$ need not coincide with the restriction of the join operation in $\mathcal{P}(n))$. The minimal and maximal elements of $\BNC(\chi)$ will be denoted  by $0_{\chi}$ and $1_{\chi}$ respectively (with $0_{\chi}=\sx(0_n)=0_n$ and $1_{\chi}=\sx(1_n)=1_n$).  For $\emptyset\neq V\subseteq\{1,\ldots,n\}$, we denote by ${\min}_< V$ and ${\min}_{{\prec}_{\chi}}V$ the minimum  element of $V$ with respect to the natural order and the ${\chi}$-order of $\{1,\ldots,n\}$ respectively. Similar notation will be used for such maximum elements.

\begin{definition}\label{defbncmobius}
The \emph{bi-non-crossing M{\"o}bius function} is the map
\[
\mu_{\BNC}: \bigcup_{n\in\N} \bigcup_{\chi\in {\{l,r\}}^n} \BNC(\chi)\times\BNC(\chi)\rightarrow\C
\]
defined recursively by
\[
\sum_{\substack{\rho\in\BNC(\chi)\\
\tau\leq\rho\leq\lambda\\}}{\mu}_{\BNC}(\tau,\rho) = \sum_{\substack{\rho\in\BNC(\chi)\\
\tau\leq\rho\leq \lambda\\}}{\mu}_{\BNC}(\rho,\lambda) =
     \begin{cases}
       1, &\text{if }  \tau=\lambda\\
        0, &\text{if }  \tau <\lambda,\\
     \end{cases}
\]
whenever $\tau\leq\lambda$, while taking the zero value otherwise.
\end{definition}
The connection between the bi-non-crossing M{\"o}bius function and the M{\"o}bius function on the lattice of non-crossing partitions $\mu_{\NC}$ is given by the formula
\[
{\mu}_{\BNC}(\tau,\lambda) = {\mu}_{\NC}(\sx^{-1}\cdot\tau, \sx^{-1}\cdot\lambda),
\]
for all $\tau\leq\lambda\in\BNC(\chi)$ and hence $\mu_{\BNC}$ inherits many of the multiplicative properties of $\mu_{\NC}$ (see \cite[Section 3]{CNS2015-2}).

The \emph{Catalan numbers} ${\{C_n\}}_{n\in\N}$ form a sequence of positive integers frequently used in the field of combinatorics; it is well known that the  $n$-th Catalan number $C_n$ equals the number of non-crossing partitions  on a set of $n$ elements and, as a result, also equals the number of bi-non-crossing partitions with respect to any map $\chi\in {\{l,r\}}^n$ (see \cite[Proposition 9.4]{nicaspeicher}). This sequence will come up when we make reference to the joint $*$-distribution of bi-Haar unitary pairs of operators (Corollary \ref{bihaarcum}). We state the following lemma tying the values of the bi-non-crossing  M{\"o}bius function with the Catalan numbers.

\begin{lemma}\label{lemmamobius}
Let $n\in\N$ and $\chi\in {\{l,r\}}^n$. Then for all $\tau\in\BNC(\chi)$ we have 
\[
\mu_{\BNC}(0_\chi,\tau) = \prod_{V\in\tau}{(-1)}^{|V|-1}\cdot C_{|V|-1}.
\]
In particular,
\[
\mu_{\BNC}(0_\chi, 1_\chi) = {(-1)}^{n-1}\cdot C_{n-1},
\]
where $C_n$ denotes the $n$-th Catalan number.
\end{lemma}

Due to the connection between $\mu_{\BNC}$ and $\mu_{\NC}$, the proof of the aforementioned  lemma is based on facts regarding the behaviour of multiplicative functions on $\NC(n)$. More specifically, it relies on the canonical factorization of intervals in the lattice of non-crossing partitions and on the multiplicative properties of the M{\"o}bius function $\mu_{\NC}$ (see \cite[Theorem 9.29, Proposition 10.14 and 10.15]{nicaspeicher}).

The \emph{Kreweras complementation map} $\Krew_{\NC}:\NC(n)\rightarrow\NC(n)$  defined in \cite{kreweras} is an important example of a lattice anti-isomorphism. For its descripition, we introduce new symbols $\overline{1},\overline{2},\ldots,\overline{n}$ and consider them interlaced with $1,2,\ldots,n$ in the following manner:
\[
1 \overline{1} 2 \overline{2}\ldots n\overline{n}.
\]
For $\pi\in\NC(n)$, its Kreweras complement $\Krew_{\NC}(\pi)\in\NC(\{\overline{1},\overline{2},\ldots,\overline{n}\})\cong\NC(n)$ is defined to be the largest non-crossing partition having the property 
\[
\pi\cup \Krew_{\NC}(\pi)\in\NC(\{1, \overline{1}, 2 ,\overline{2}\ldots n,\overline{n}\}).
\]
The complementation map found its generalization for the lattice of bi-non-crossing partitions in  \cite[Section 5]{CNS2015-2}. Specifically, for any $n\in\N,\chi\in {\{l,r\}}^n$ and $\tau\in\BNC(\chi)$, the Kreweras complement of $\tau$ in $\BNC(\chi)$, denoted by $\Krew_{\BNC}(\tau)$, is defined as
\[
\Krew_{\BNC}(\tau) = \sx\cdot \Krew_{\NC}(\sx^{-1}\cdot\tau),
\]
i.e. is given by applying the permutation $\sx$ to the Kreweras complement of $\sx^{-1}\cdot\tau$ in $\NC(n)$. Note that in the special case when $\chi\in {\{l,r\}}^n$ gives the constant value ``$l$'', one obtains $\Krew_{\NC}$. In the following lemma, we list properties of $\Krew_{\BNC}$ that are relevant to our purposes. 

\begin{lemma}\label{kncproperties}
Let $n\in\N$ and $\chi\in {\{l,r\}}^n$. Then:
\begin{enumerate}[(i)]
\item $\Krew_{\BNC} : \BNC(\chi)\rightarrow\BNC(\chi)$ is a bijection,
\item $\Krew_{\BNC}(0_{\chi})=1_{\chi}$ and $\Krew_{\BNC}(1_{\chi})=0_{\chi}$,
\item For all $\tau,\lambda\in\BNC(\chi)$ we have 
\[
\tau\leq\lambda\iff \Krew_{\BNC}(\lambda)\leq \Krew_{\BNC}(\tau)\iff \Krew_{\BNC}^{-1}(\lambda) \leq \Krew_{\BNC}^{-1}(\tau).
\]
\end{enumerate}
\end{lemma}
All of these properties are easily verified by the definition of $\Krew_{\BNC}$ and by the corresponding properties which hold for $\Krew_{\NC}$. We shall now state a combinatorial lemma, which may be of independent interest and involves the following cancellation property for the lattice of bi-non-crossing partitions. A special case of the next result will play a key role in the proof of Lemma \ref{lemmainv1}.

\begin{lemma}\label{techlemmaBNC}
Let $n\in\N, \chi\in {\{l,r\}}^n$ and consider a family ${\{d_{\tau}\}}_{\tau\in\BNC(\chi)}$ of  indeterminates indexed by the bi-non-crossing partitions $\BNC(\chi)$. Then, the following holds:
\[
\sum_{\tau\in\BNC(\chi)}
\left(   {\mu}_{\BNC}(0_{\chi},\tau)\cdot  \sum_{\substack{\lambda\in \BNC(\chi)\\
\lambda\leq \Krew_{\BNC}(\tau)\\}}d_{\lambda}  \right) = d_{1_{\chi}}.
\]
\end{lemma}

\begin{proof}
Re-arranging the left hand-side of the above expression yields
\[
\sum_{\tau\in\BNC(\chi)}
\left(   {\mu}_{\BNC}(0_{\chi},\tau)\cdot  \sum_{\substack{\lambda\in\BNC(\chi)\\
\lambda\leq \Krew_{\BNC}(\tau)\\}}d_{\lambda}  \right)  = \sum_{\lambda\in\BNC(\chi)}
\left( d_{\lambda}  \cdot  \sum_{\substack{\tau\in\BNC(\chi)\\
\lambda\leq \Krew_{\BNC}(\tau)\\}}{\mu}_{\BNC}(0_{\chi},\tau)  \right). 
\]
Thus to prove the conclusion of the lemma, it suffices to show that for all $\lambda\in\BNC(\chi)$, we have 
\[
\sum_{\substack{\tau\in\BNC(\chi)\\
\lambda\leq \Krew_{\BNC}(\tau)\\}}{\mu}_{\BNC}(0_{\chi},\tau) = 
     \begin{cases}
       1, &\text{ if } \lambda=1_{\chi}\\
        0, &\text{ if } \lambda < 1_{\chi}.\\
     \end{cases}
\]
Fix $\lambda\in\BNC(\chi)$ and let $\lambda'\in\BNC(\chi)$ be such that $\lambda = \Krew_{\BNC}(\lambda')$. Observe that since
\[
\lambda\leq \Krew_{\BNC}(\tau)\iff \Krew_{\BNC}(\lambda')\leq \Krew_{\BNC}(\tau)\iff \tau\leq\lambda',
\]
we have that
\[
\{\tau\in\BNC(\chi) : \lambda\leq \Krew_{\BNC}(\tau)\} = \{\tau\in\BNC(\chi) : \tau\leq\lambda'\}.
\]
Elementary properties of the M{\"o}bius function  on the lattice of bi-non-crossing partitions  imply that
\[
\sum_{\substack{\tau\in\BNC(\chi)\\
\lambda\leq \Krew_{\BNC}(\tau)\\}}{\mu}_{\BNC}(0_{\chi},\tau) = \sum_{\substack{\tau\in\BNC(\chi)\\
0_{\chi}\leq\tau\leq\lambda'\\}}{\mu}_{\BNC}(0_{\chi},\tau) = 
     \begin{cases}
       1, &\text{ if } 0_{\chi}=\lambda'\\
        0, &\text{ if }  0_{\chi} < \lambda'.\\
     \end{cases}
\]
Then, an application of Lemma \ref{kncproperties} shows
\[
0_{\chi}=\lambda'\iff \Krew_{\BNC}^{-1}(1_{\chi})=\Krew_{\BNC}^{-1}(\lambda)\iff \lambda=1_{\chi}
\]
and
\[
0_{\chi} <\lambda'\iff \Krew_{\BNC}^{-1}(1_{\chi}) < \Krew_{\BNC}^{-1}(\lambda)\iff \lambda < 1_{\chi}.
\]
This completes the proof.
\end{proof}
Of course, when the map $\chi\in {\{l,r\}}^n$ is constant with  constant value ``$l$'', one obtains the analogous result for the lattice of non-crossing partitions. It is this case that will be used in the proof of Lemma \ref{lemmainv1}.

\subsection{Bi-Free Independence and Bi-Free Cumulants}\label{sctn:bifreec}
We begin by recalling the notion of bi-free independence for pairs of faces in a non-commutative probability space, originally developed in \cite{voiculescu}.

\begin{definition}\label{defbifree}
Let $(A,\varphi)$ be a non-commutative probability space. 
\begin{enumerate}[(i)]
\item A \emph{pair of faces} in $(A,\varphi)$ consists of a pair $(C,D)$ of unital subalgebras of $A$,
\item a family ${\{(C_k,D_k)\}}_{k\in K}$ of pairs of faces in $(A,\varphi)$ is said to be \emph{bi-freely independent} (or simply \emph{bi-free}) if there exists a family of vector spaces with specified vector states ${\{(\mathcal{X}_k,\overset{\circ}{\mathcal{X}}_k,\xi_k)\}}_{k\in K}$ and unital homomorphisms
\[
l_k:C_k\rightarrow \mathcal{L}(\mathcal{X}_k)\text{ and }r_k:D_k\rightarrow \mathcal{L}(\mathcal{X}_k),
\]
(where $\mathcal{L}(\mathcal{X}_k)$ denotes the space of all linear maps on $\mathcal{X}_k$) such that the joint distribution of the family ${\{(C_k,D_k)\}}_{k\in K}$ with respect to $\varphi$ coincides with the joint distribution of the family 
\[
{\{((\lambda_k\circ l_k)(C_k),(\rho_k\circ r_k) (D_k))\}}_{k\in K}
\]
with respect to the vacuum state on $\mathcal{L}(*_{k\in K}\mathcal{X}_k)$, where $\lambda_k,\rho_k$ denote the left and right regular representations onto $\mathcal{X}_k\subseteq *_{k\in K}\mathcal{X}_k$ respectively,
\item if $S_k$ and $V_k$ are subsets of $A$ for all $k\in K$, then the family ${\{(S_k,V_k)\}}_{k\in K}$ will be said to be \emph{bi-free} if the family of pairs of faces 
\[
{\{(\alg(1_A\cup S_k), \alg(1_A \cup V_k))\}}_{k\in K}
\]
is bi-free,
\item if $(A,\varphi)$ is a non-commutative $*$-probability space and  $S_k$ and $V_k$ are subsets of $A$ for all $k\in K$, then the family ${\{(S_k,V_k)\}}_{k\in K}$ will be said to be \emph{$*$-bi-free} if the family  
\[
{\{(S_k\cup S_k^*,V_k\cup V_k^*)\}}_{k\in K}
\]
is bi-free.
\end{enumerate}
\end{definition}

In the following remark we list some salient properties of bi-free pairs following their definition.
\begin{remark}\label{remark:propertiesofbi-freepairs}
 \begin{enumerate}[(i)]
     \item If the family of pairs of faces ${\{(C_k,D_k)\}}_{k\in K}$ in the non-commutative probability space $(A,\varphi)$ is bi-free, then both the families ${(C_k)}_{k\in K}$ and ${(D_k)}_{k\in K}$ are freely independent. Moreover, for all $k\neq k'\in K$, the algebras $C_k$ and $D_{k'}$ are classically independent. In addition, if $C$ and $D$ are classically independent unital subalgebras of $(A,\varphi)$, then the pairs of faces $(C,\C1_A)$ and $(\C1_A,D)$ are bi-free (\cite[Proposition 2.15, 2.16]{voiculescu}).
     \item If  
     ${\{(C_k,D_k)\}}_{k\in K}$ is a family of pairs of unital algebras 
      (with each pair contained in some non-commutative probability space $(A_k,\varphi_k)$),
      then we can always embed these pairs into a larger non-commutative probability space in which they will be bi-free and all of their joint distributions will remain intact (see \cite[Corollary 2.10]{voiculescu} and also \cite[Theorem 6.4.1]{CNS2015-1}). For $k\in K$, suppose that the unital, linear functional $\mu_k:C_k*D_k\rightarrow\C$ defined on the algebraic free product (with amalgamation over $\C$) of $C_k$ and $D_k$ represents the joint distribution of the algebras $C_k,D_k$, so that $\mu_k=\varphi_k\circ a_k$, where $\alpha_k:C_k*D_k\rightarrow A_k$ is the (unique) unital homomorphism such that the restrictions $\alpha_k|_{C_k}$ and $\alpha_k|_{D_k}$ are the inclusion maps. Then there exists a (necessarily unique) unital, linear functional $\mu:*_{k\in K}(C_k*D_k)\rightarrow\C$ such that the family $\{(C_k,D_k)\}_{k\in K}$ is bi-free in the non-commutative probability space $\left(*_{k\in K}(C_k*D_k),\mu\right)$ and such that $\mu|_{C_k*D_k}=\mu_k$ for all $k\in K$. In addition, if each $(A_k,\varphi_k)$ is a $\mathrm{C}^*$-probability space and each $C_k,D_k$ is a unital $\mathrm{C}^*$-subalgebra, then $\mu$ extends by continuity to a state on the full $\mathrm{C}^*$-algebraic free product $*_{k\in K}(C_k*D_k)$ (\cite[Section 3.2]{voiculescu}).
\end{enumerate}   
\end{remark}
The bi-free cumulant function is the main combinatorial tool used in bi-free probability theory and its definition is given below.
\begin{definition}\label{defbnccum}
Let $(A,\varphi)$ be a non-commutative $*$-probability space. The \emph{bi-free cumulant function} is the map
\[
\kappa :\bigcup_{n\in\N}\bigcup_{\chi\in {\{l,r\}}^n}\BNC(\chi)\times A^n\rightarrow \C
\]
defined by
\[
\kappa_{\chi,\tau}(a_1,\ldots,a_n):=\kappa(\tau,a_1,\ldots,a_n) = \mathlarger{\sum}_{\substack{\lambda\in \BNC(\chi)\\
\lambda\leq\tau\\}}\varphi_{\lambda}(a_1,\ldots,a_n)\mu_{\BNC}(\lambda,\tau),
\]
for each $n\in\N,\chi\in {\{l,r\}}^n, \tau\in\BNC(\chi)$ and $a_1,\ldots,a_n\in A$.
\end{definition}
The previous formula is called the \emph{moment-cumulant formula} and an application of M{\"o}bius inversion yields that we must also have that
\[
\varphi(a_1\cdots a_n) = \sum_{\tau\in\BNC(\chi)}\kappa_{\chi,\tau}(a_1,\ldots,a_n).
\]
It is clear that for $n\in\N,\chi\in {\{l,r\}}^n$ and $\tau\in\BNC(\chi)$, the bi-free cumulant map
$\kappa_{\chi,\tau}:A^n\rightarrow\C$
is multilinear. In the special  case when $\tau=1_{\chi}$, we will denote $\kappa_{\chi, 1_{\chi}}$ simply by $\kappa_{\chi}$. Multiplicative properties of the bi-free cumulant function yield that
\[
\kappa_{\chi,\tau}(a_1,\ldots,a_n) = \prod_{V\in\tau}\kappa_{\xv}((a_1,\ldots,a_n)|_V),
\]
for all $n\in\N,\chi\in {\{l,r\}}^n, \tau\in\BNC(\chi)$ and $a_1,\ldots,a_n\in A$ (see \cite{CNS2015-2} for proofs  and discussions on all the aforementioned properties). Note that the result of reading the sequence $(a_1,\ldots,a_n)|_V$ with the indices in the induced $\xv$-order coincides with first reading the sequence $(a_1,\ldots,a_n)$ with the indices in the $\chi$-order and then restricting the resulting sequence to $\sx^{-1}(V)$. For a concrete example, let $a_1, a_2, a_3, a_4, a_5\in A$, let  $V=\{2, 3, 4\}$ and let $\chi\in\{l,r\}^5$ be such that $\chi^{-1}(\{l\})=\{1,4\}$. Then ${(a_1,\ldots, a_5)}|_V = (a_2, a_3, a_4)$ and the result of reading this sequence  in the induced ${\chi}|_V$-order is $(a_4, a_3, a_2)$ (since we are first listing the left entries  in increasing order, followed by the right entries in decreasing order). Also, since  $(a_{s_\chi(1)},\ldots, a_{s_\chi (5)}) = (a_1, a_4, a_5, a_3, a_2)$  and $\sx^{-1}(V) = \{2, 4, 5\}$ (which corresponds to the fact  that we will only be keeping the second, fourth and fifth terms of the aforementioned induced sequence), the coincidence of the two sequences follows.

For $a_1,\ldots,a_n\in A$, we will often make reference to the \textit{bi-free cumulants of the tuple $(a_1,\ldots,a_n)$},   which simply signify the collection
\[
\{\kappa_\chi (c_1,\ldots,c_k) : k\in\N, \chi\in\{l,r\}^k, c_1,\ldots,c_k\in\{a_1,\ldots,a_n\}\},
\]
of all bi-free cumulants with entries in the tuple $(a_1,\ldots, a_n)$. We will also make reference to the \textit{bi-free $*$-cumulants of the tuple $(a_1,\ldots,a_n)$} when talking about the bi-free cumulants of the tuple $(a_1,a_1^*,\ldots, a_n, a_n^*)$. Observe that the moment-cumulant formula implies that for elements $X,Y,Z,W\in A$  the joint $*$-distribution of the pair $(X,Y)$ coincides with the joint $*$-distribution of $(Z,W)$ if and only if all bi-free $*$-cumulants of the pair $(X,Y)$ coincide with the corresponding bi-free $*$-cumulants of the pair $(Z,W)$.

The following theorem displays the combinatorial characterization of bi-free independence.
\begin{theorem}[\cite{CNS2015-2}, Theorem 4.3.1]\label{equivbifree}
Let $(A,\varphi)$ be a non-commutative $*$-probability space and let ${\{(C_k,D_k)\}}_{k\in K}$ be family of pairs of faces in $A$. The following are equivalent:
\begin{enumerate}
[(i)]
\item the family  ${\{(C_k,D_k)\}}_{k\in K}$ is bi-free,
\item for all $n\in\N,\chi\in {\{l,r\}}^n, a_1,\ldots,a_n\in A$ and non-constant map $\epsilon:\{1,\ldots,n\}\rightarrow K$ such that
\[
a_i\in \begin{cases}
       C_{\epsilon(i)}, &\text{if }  \chi (i)=l\\
        D_{\epsilon(i)}, &\text{if }  \chi (i) = r\\   
     \end{cases}\quad (i=1,\ldots,n),
\]
we have that
\[
\kappa_{\chi}(a_1,\ldots,a_n)=0.
\]
\end{enumerate}
\end{theorem}

Given that the majority of the  results of this paper involve computations of bi-free cumulants having products of operators as entries,  it is important to be able to express these in terms of cumulants whose entries are obtained by isolating each of the original product terms. To this end we will 
recall from \cite[Section 9]{CNS2015-1} the necessary combinatorial notions.
\begin{definition}\label{notationonly}
Let $n, m \in\N$, $\chi\in\{l,r\}^n$ and fix natural numbers
\[
k(0)=0 < k(1) < \ldots < k(n)=m.
\]
We define $\xh\in\{l,r\}^m$ via $\xh (q) = \chi(p_q)$,
where $p_q$ is the unique element of $\{1,\ldots,n\}$ such that $k(p_q - 1) < q \leq k(p_q)$. With this notation, we define an embedding \begin{align*}
    \BNC (\chi) &\rightarrow \BNC(\xh)\\
    \pi & \mapsto \hat{\pi}
\end{align*}
that is obtained by replacing $i\in\{1,\ldots,n\}$ by the block $\{k(i-1)+1,\ldots, k(i)\}$. Note that $\xh$ is constant on each block $\{k(i-1)+1,\ldots, k(i)\}$ and its value is equal to $\chi(i)$.
\end{definition}
It is easy to see that the map $\pi\mapsto\hat{\pi}$ is an injective order embedding, $1_{\xh} = \widehat{1_\chi}$ and that $\widehat{0_\chi}$ is the partition with blocks
\[
\left\{\{k(i-1)+1,\ldots, k(i)\} : 1\leq i \leq n\right\}.
\]
Moreover, the image of $\BNC(\chi)$ under this map is
\[
\widehat{\BNC(\chi)} = \left[\widehat{0_{\chi}},\widehat{1_{\chi}}\right] = \left[\widehat{0_{\chi}},1_{\xh}\right]\subseteq \BNC( \xh),
\]
and, since this map preserves the order, we obtain that $\mu_{\BNC}(\sigma, \pi) = \mu_{\BNC} (\hat{\sigma}, \hat{\pi})$. The main idea behind introducing the notions  in Definition \ref{notationonly} becomes clear when considering the following theorem, which we will be using frequently  throughout this paper.
\begin{theorem}[Scalar case of \cite{CNS2015-1}, Theorem 9.1.5]\label{openup} Let $(A,\varphi)$ be a non-commutative probability space, $n<m\in\N, \chi\in {\{l,r\}}^n$ and integers
\[
k(0)=0<k(1)<\ldots<k(n)=m.
\]
Also, let $a_1,\ldots,a_m\in A$. Then with $\xh\in {\{l,r\}}^m$  as in Definition \ref{notationonly} we have 
\[
\kappa_{\chi}\left(a_1\cdots a_{k(1)},a_{k(1)+1}\cdots a_{k(2)},\ldots,a_{k(n-1)+1}\cdots a_{k(n)}\right) = \sum_{\substack{\tau\in \BNC(\xh)\\
\tau\vee\oxh = 1_{\xh}\\}}
        \kappa_{\xh,\tau}(a_1,\ldots,a_m).
\]
\end{theorem}
Note that the theorem above deals with bi-free cumulants whose entries are product terms, which in general are of different lengths. Given that the vast majority of the results in the next sections deal with cumulants all of whose entries are product terms of the same length and in order to reduce clutter, we find it convenient to fix the appropriate notation in order to fit our specific needs. Unless otherwise explicitly stated (namely only in Lemma \ref{lemma:bi-R-diagpowerscounterexample} and Proposition \ref{prop:bi-R-diagpowerscounterexample}) the notation given below will be adopted and referenced accordingly.
\begin{notation}\label{notation: products}
Let $n,p\in\N$ and let $\chi\in\{l,r\}^n$. 
\begin{enumerate}[(i)]
    \item We will use $\xh$ to denote the map in $\{l,r\}^{np}$ given in Definition \ref{notationonly} for the case when $k(i)=ip$ for all $i=1,\ldots,n$. For each $1\leq i\leq n$ this is given by
    \[
    \xh\left((i-1)p+1\right)=\ldots=\xh(ip)=l,\text{ if }\chi(i)=l,
    \]
    and
    \[
    \xh\left((i-1)p+1\right)=\ldots=\xh(ip)=r,\text{ if }\chi(i)=r.
    \]
    \item We will similarly denote by $\oxh$ the bi-non-crossing partition obtained in Definition \ref{notationonly} for the case when $k(i)=ip$ for all $i=1,\ldots,n$. This partition is specifically given as
    \[
    \oxh=\left\{\left\{(i-1)p+1,\ldots,ip\right\}:1\leq i\leq p\right\}\in\BNC(\xh).
    \]
    \item If  $(A,\varphi)$ is a non-commutative probability space and $a_1,\ldots,a_n\in A$ are such that each $a_i$ is a product of a $p$-tuple of operators in $A$, i.e.
    \[
    a_i=X_1^{(i)}\cdots X_p^{(i)},\text{ for all }i=1,\ldots,n,
    \]
    we denote by $c_1,\ldots,c_{np}$ the operators in $A$ given such that the sequence $(c_1,\ldots,c_{np})$ corresponds to isolating all product terms appearing in the sequence $(a_1,\ldots,a_n)$. 
    For $1\leq i\leq n$ and $1\leq j\leq p$, these are given by  $c_{(i-1)p+j}=X_j^{(i)}$.
\end{enumerate}
\end{notation}
If we think of $i\in\{1,\ldots,n\}$ as indexing an operator (that is either a left operator if $\chi(i)=l$, or a right operator if $\chi(i)=r$) which is a product of $p$-many terms, then each of those $p$ terms is marked accordingly  by $\xh$ as either left or right and all of these $p$  product terms are grouped together via the indices in a block of $\oxh$. Since the map $\xh\in\{l,r\}^{np}$ is constant on each of the blocks of $\oxh$, it follows that 
    \[    \oxh=\sxh\cdot\oxh=\left\{\left\{\sxh\left((i-1)p+1\right),\ldots,\sxh(ip)\right\}:1\leq i\leq n\right\}.
    \]
    Thus  the blocks of $\oxh$ retain the same structure regardless of  whether we look at the interval $\{1,\ldots,np\}$ with the original order or with the $\xh$-order. Additionally, if $a_1,\ldots,a_n,c_1,\ldots,c_{np}\in (A,\varphi)$ are as in Notation \ref{notation: products}, then by Theorem \ref{openup} we obtain
    \[
\kappa_\chi(a_1,\ldots,a_n)=\sum_{\substack{\tau\in\BNC(\xh)\\ \tau\vee\oxh=1_{\xh}}}\kappa_{\xh,\tau}\left(c_1,\ldots,c_{np}\right)=\sum_{\substack{\tau\in\BNC(\xh)\\ \tau\vee\oxh=1_{\xh}}}\prod_{V\in\tau}\kappa_{\xh|_V}\left((c_1,\ldots,c_{np})|_V\right),
    \]
where the last equality follows by the multiplicative properties of bi-free cumulants. Note that every product term appearing in the bi-free cumulant on the left-hand side of the equation above has the same length (namely $p$). In the majority of the results in the next sections, this is the situation that we will be concerned with, as we will be  dealing with bi-free $*$-cumulants that have entries in pairs  of the form $(X^p,Y^p)$ or $(XZ,WY)$ for operators $X,Y,Z,W$ in some non-commutative $*$-probability  $(A,\varphi)$. For a concrete example in the case $p=2$, if $\chi=\{r,l,l,r\}$, then
\[
\kappa_\chi\left(Y^*W^*, XZ, Z^*X^*, WY\right)=\sum_{\substack{\tau\in\BNC(\xh)\\ \tau\vee\oxh=1_{\xh}}}\kappa_{\xh,\tau}\left(Y^*, W^*, X, Z, Z^*, X^*, W, Y\right),
\]
where $\xh\in\{l,r\}^8$ with $\xh^{-1}(\{l\})=\{3,4,5,6\}$ and $\oxh=\left\{\{1,2\}, \{3,4\}, \{5,6\}, \{7,8\}\right\}$. In all cases, it will be clear from the context (i.e. from the length of the product terms that appear in a given bi-free cumulant) what is the value of $p$ that appears in the statement of Notation \ref{notation: products}. In the next remark we demonstrate some of the  specific instances where  Notation \ref{notation: products} will appear in the next sections.  
\begin{remark}
    Let $(A,\varphi)$ be a non-commutative probability space, let $X,Y,Z,W\in A$, let $n,p\in\N$ and $\chi\in\{l,r\}^n$.
\begin{enumerate}[(i)]
\item If $a_1,\ldots,a_n\in A$ are such that
\[
a_i\in\begin{cases}
    \left\{X^p, (X^*)^p\right\},&\text{if }\chi(i)=l\\
    \left\{Y^p, (Y^*)^p\right\},&\text{if }\chi(i)=r
\end{cases}\quad (i=1,\ldots,n),
\]
then the elements $c_1,\ldots,c_{np}$ as in Notation \ref{notation: products} are for $1\leq i\leq n$ and $1\leq j\leq n$ given by
\[
c_{(i-1)p+j}=\begin{cases}
    X,&\text{if }a_i=X^p\\
    X^*,&\text{if }a_i=(X^*)^p\\
    Y,&\text{if }a_i=Y^p\\
    Y^*,&\text{if }a_i=(Y^*)^p
\end{cases}.
\]
Such notation appears in Theorem \ref{birpowers}. As a concrete example, if $p=3$  then according to the sequence $\left((Y^*)^3,X^3,Y^3, (X^*)^3\right)$ we have that 
\[
(c_1,\ldots,c_{12})=(Y^*,Y^*,Y^*,X,X,X,Y,Y,Y, X^*,X^*,X^*).
\]
Also in this example the map $\xh\in\{l,r\}^{12}$ and the partition $\oxh$ from Notation \ref{notation: products} are given by   $\xh^{-1}(\{l\})=\{4,5,6,10,11,12\}$ and 
\[
\oxh=\{\{1,2,3\},\{4,5,6\},\{7,8,9\},\{10,11,12\}\}.
\] 
\item Suppose that $a_1,\ldots,a_n\in A$ are such that 
\[
a_i\in\begin{cases}
    \left\{XZ, Z^* X^*\right\},&\text{if }\chi(i)=l\\
    \left\{WY, Y^* W^*\right\},&\text{if }\chi(i)=r
\end{cases}\quad (i=1,\ldots,n).
\]
In this case the $a_i$'s should be thought of as entries of a bi-free $*$-cumulant of the pair $(XZ,WY)$, so that here $p=2$ (observe that the order of multiplication is reversed for the right operators). Then the elements $c_1,\ldots,c_{2n}\in A$ as in Notation \ref{notation: products} are given by
\[
c_{2i-1}=\begin{cases}
    X,&\text{if }a_i=XZ\\
    Z^*,&\text{if }a_i=Z^* X^*\\
    W,&\text{if }a_i=WY\\
    Y^*,&\text{if }a_i=Y^*W^*
\end{cases}\text{ and } c_{2i}=\begin{cases}
    Z,&\text{if }a_i=XZ\\
    X^*,&\text{if }a_i=Z^*X^*\\
    Y,&\text{if }a_i=WY\\
    W^*,&\text{if }a_i=Y^*W^*
\end{cases}\quad (i=1,\ldots,n).
\]
Such notation appears in Theorem \ref{prodbir-bifree}, Proposition \ref{prod-bi-even} and Lemma \ref{lemmainv1}. We mention in this case that the associated map $\xh\in\{l,r\}^{2n}$ as in Notation \ref{notation: products} will be such that $\xh(2i-1)=\xh(2i)=\chi(i)$ for all $i=1,\ldots,n$, while  $\oxh=\left\{\left\{2i-1,2i\right\}:1\leq i\leq n\right\}$. Observe for $1\leq i\leq n$ that if $a_{\sx(i)}=XZ$, then
\[
c_{\sxh(2i-1)}=X\text{ and }c_{\sxh(2i)}=Z,
\]
with a similar situation occurring if $a_{\sx(i)}=Z^*X^*$, since this corresponds to a left operator. On the other hand, if $a_{\sx(i)}=WY$, then
\[
c_{\sxh(2i-1)}=Y\text{ and }c_{\sxh(2i)}=W,
\]
with a similar situation occurring  when $a_{\sx(i)}=Y^*W^*$, since this corresponds to a right operator. Note in the latter case that the right operators must appear reversed in the $\xh$-order.
\end{enumerate}
\end{remark}
Note that (with the conventions of Notation \ref{notation: products}) if $\tau\in\BNC(\xh)$ satisfies $\tau\vee\oxh=1_{\xh}$, then the only $\xh$-interval that is simultaneously a union of blocks of $\tau$ and a union of  blocks of $\oxh$ is all of $\{1,\ldots,np\}$. This has certain implications on the blocks of such a partition $\tau$ which we record in the next lemma. The properties listed below may be thought of as  ``reductions'' that allow one to identify the partitions that contribute to a given sum of bi-free cumulants, in combination with Theorem \ref{openup}.
\begin{lemma}\label{lemma: no block coupling}
Let $n,p\in\N$, $\chi\in\{l,r\}^n$ and let $\xh\in\{l,r\}^{np}$ and $\oxh\in\BNC(\xh)$ be as in Notation \ref{notation: products}. Also, consider $\tau\in\BNC(\xh)$ with $\tau\vee\oxh=1_{\xh}$ and let $V$ be a block of $\tau$.
\begin{enumerate}[(i)]
    \item If there exist  $1\leq  i\leq j\leq n$ such that
    \[
    \sxh\left((i-1)p+1\right)=\min_{\prec_{\xh}}V\text{ and }\sxh(jp)=\max_{\prec_{\xh}}V,
    \]
    then necessarily $i=1$ and $j=n$.  In particular, whenever the $\xh$-interval formed between the minimum and maximum elements (in the $\xh$-order) of any block of $\tau$ is a union of blocks of $\oxh$,  then this $\xh$-interval necessarily equals all of $\{1,\ldots,np\}$,
    \item for all $1\leq i < j\leq n$, it cannot be the case that the indices $\sxh(ip)$ and $\sxh(jp+1)$ are consecutive indices of $V$ in the $\xh$-order (i.e. it cannot be the case that  for all $k\in V\setminus\left\{\sxh(ip),\sxh(jp+1)\right\}$ we have that either $k\prec_{\xh}\hspace{2pt}\sxh(ip)$ or $\sxh(jp+1)\prec_{\xh} \hspace{2pt}k$). In particular,  the $\xh$-interval between consecutive indices of any block of $\tau$ (in the $\xh$-order) cannot be a union of blocks of $\oxh$.
\end{enumerate}
\end{lemma}
\begin{proof}
For clause $(i)$, arguing by way of contradiction, suppose there exist $1\leq i\leq j\leq n$ with either $1<i$ or $j<n$ such that
\[
\sxh\left((i-1)p+1\right)=\min_{\prec_{\xh}}V\text{ and }\sxh(jp)=\max_{\prec_{\xh}}V.
\]
Note this implies that at least one of the sets
\[
\left\{\sxh(1),\ldots,\sxh\left((i-1)p\right)\right\}\text{ or }\left\{\sxh(jp+1),\ldots,\sxh(np)\right\},
\]
is non-empty. This situation is depicted in the following diagram, where the indices of the interval $\{1,\ldots,np\}$ are in the $\xh$-order.   \begin{equation*}
\begin{tikzpicture}[baseline]
			\draw[thick,dashed] (-2,0) -- (12.9,0);
\filldraw (-1.8,-0.3) circle (0.02) node[anchor=north] {};
           \filldraw (-1.3,-0.3) circle (0.02) node[anchor=north] {};
            \filldraw (-0.8,-0.3) circle (0.02) node[anchor=north] {};
             \filldraw (0.8,0) circle (0.06) node[anchor=north] {$\sxh\left((i-1)p+1\right)$};
               \filldraw (2.3,-0.3) circle (0.02) node[anchor=north] {};
               \filldraw (2.8,-0.3) circle (0.02) node[anchor=north] {};
               \filldraw (3.3,-0.3) circle (0.02) node[anchor=north] {};
               \filldraw (4.1,0) circle (0.06) node[anchor=north] {$\sxh\left(ip\right)$};
               \filldraw (4.9,-0.3) circle (0.02) node[anchor=north] {};
               \filldraw (5.4,-0.3) circle (0.02) node[anchor=north] {};
               \filldraw (5.9,-0.3) circle (0.02) node[anchor=north] {};
               \filldraw (7.5,0) circle (0.06) node[anchor=north] {$\sxh\left((j-1)p+1\right)$};
               \filldraw (9.1,-0.3) circle (0.02) node[anchor=north] {};
               \filldraw (9.6,-0.3) circle (0.02) node[anchor=north] {};
               \filldraw (10.1,-0.3) circle (0.02) node[anchor=north] {};
               \filldraw (10.9,0) circle (0.06) node[anchor=north] {$\sxh\left(jp\right)$};
               \filldraw (11.7,-0.3) circle (0.02) node[anchor=north] {};
               \filldraw (12.2,-0.3) circle (0.02) node[anchor=north] {};
               \filldraw (12.7,-0.3) circle (0.02) node[anchor=north] {};
            \draw[-latex,line width=1pt] (0.8,-1.3)--(0.8,-0.5);
  \node[draw] at (0.8,-1.6) {${\min}_{{\prec}_{\xh}}V$};
   \draw[-latex,line width=1pt] (10.9,-1.3)--(10.9,-0.5);
  \node[draw] at (10.9,-1.6) {${\max}_{{\prec}_{\xh}}V$};
\draw[thick, black] (0.8,-2.4) -- (0.8,-2);
\draw[thick, black] (10.9,-2.4) -- (10.9,-2);
 \draw[-latex,line width=1pt] (2.8,-2.2)--(0.8,-2.2);
  \draw[-latex,line width=1pt] (8.9,-2.2)--(10.9,-2.2);
  \draw[thick,dashed, line width=1pt] (2.8,-2.2) -- (5.5,-2.2);
  \draw[thick,dashed, line width=1pt] (8.9,-2.2) -- (6.2,-2.2);
 \node[draw] at (5.85,-2.2) {$V$};
 \draw [thick, black,decorate,decoration={brace,amplitude=10pt,mirror},xshift=0.4pt,yshift=-0.4pt](0.75,-2.5) -- (10.93,-2.5) node[black,midway,yshift=-0.6cm] {\footnotesize union of blocks of $\tau$ and union of blocks of $\oxh$};
\draw[thick, black] (0.8,0) -- (0.8 ,1) -- (10.9,1) -- (10.9,0);	
			\end{tikzpicture}
\end{equation*}
Define 
\[
\tilde{V}=\left\{\sxh(k):(i-1)p+1\leq k\leq jp\right\},
\]
and first note that $\tilde{V}$ is clearly a union of blocks of $\oxh$. If $V'\in\tau$ is another block of $\tau$ such that $V'\cap\tilde{V}\neq\emptyset$, then we must necessarily have that $V'\subseteq\tilde{V}$, since the partition $\sxh^{-1}\cdot\tau$ is non-crossing. As a result, $\tilde{V}$ must be written as a union of blocks of $\tau$. If we define the partition
\[
\lambda = \left\{\tilde{V}, {\left(\tilde{V}\right)}{\vphantom{\tilde{V}}}^\mathsf{c}\right\},
\]
then it is immediate that $\tau,\oxh\leq \lambda\lneq 1_{\xh}$ (note that both blocks of $\lambda$ are non-empty) and also $\lambda$ is bi-non-crossing, since
\[
\sxh^{-1}\cdot\lambda=\left\{\left\{(i-1)p+1,\ldots,jp\right\}, \left\{1,\ldots,(i-1)p\right\}\cup\left\{jp+1,\ldots,np\right\}\right\}\in\NC(np).
\]
This shows that the relation $\tau\vee\oxh=1_{\xh}$ cannot be satisfied. For clause $(ii)$ we similarly assume that $\sxh(ip)$ and $\sxh(jp+1)$ are consecutive indices of $V$ in the $\xh$-order, for some $1\leq i<j\leq n$. This situation  is depicted in the next diagram, where again the indices of the interval $\{1,\ldots,np\}$ are in the $\xh$-order.  
\begin{equation*}
\begin{tikzpicture}[baseline]
			\draw[thick,dashed] (2,0) -- (13,0);
       \filldraw (2.2,-0.3) circle (0.02) node[anchor=north] {};
        \filldraw (2.7,-0.3) circle (0.02) node[anchor=north] {};
        \filldraw (3.2,-0.3) circle (0.02) node[anchor=north] {};
        \filldraw (4,0) circle (0.06) node[anchor=north] {$\sxh(ip)$};
        \filldraw (5.6,0) circle (0.06) node[anchor=north] {$\sxh(ip+1)$};
         \filldraw (6.8,-0.3) circle (0.02) node[anchor=north] {};
          \filldraw (7.4,-0.3) circle (0.02) node[anchor=north] {};
           \filldraw (7.9,-0.3) circle (0.02) node[anchor=north] {};
           \filldraw (8.9,0) circle (0.06) node[anchor=north] {$\sxh(jp)$};
            \filldraw (10.5,0) circle (0.06) node[anchor=north] {$\sxh(jp+1)$};
        \filldraw (11.7,-0.3) circle (0.02) node[anchor=north] {};
         \filldraw (12.2,-0.3) circle (0.02) node[anchor=north] {};
          \filldraw (12.7,-0.3) circle (0.02) node[anchor=north] {};

\draw [thick, black,decorate,decoration={brace,amplitude=10pt},xshift=0.4pt,yshift=-0.4pt](5.6,0.2) -- (8.9,0.2) node[black,above=32pt, midway,yshift=-0.8cm,align=left] {\footnotesize union of blocks of $\tau$ and  \\ \footnotesize union of blocks of $\oxh$};

\draw[thick, black] (4,0) -- (4,1.6) -- (10.5,1.6) -- (10.5,0);
\draw[thick, black] (10.5,1.6) -- (11,1.6);
\draw[thick,dashed] (11,1.6) -- (11.9,1.6);
\draw[thick, black] (4,1.6) -- (3.5,1.6);
\draw[thick,dashed] (3.5,1.6) -- (2.6,1.6);
	
			\end{tikzpicture}
\end{equation*}
We define 
\[
\tilde{V}=\left\{\sxh(k):ip+1\leq k\leq jp\right\},
\]
and as before we note that $\tilde{V}$ is both  a union of blocks of $\oxh$ and a union of blocks of $\tau$. Thus, for the partition 
\[
\lambda = \left\{\tilde{V}, {\left(\tilde{V}\right)}{\vphantom{\tilde{V}}}^\mathsf{c}\right\},
\]
it follows that   $\tau,\oxh\leq \lambda\lneq 1_{\xh}$ (note that both blocks of $\lambda$ are non-empty, since  $i<j$). Also $\lambda$ is bi-non-crossing, since
\[
\sxh^{-1}\cdot\lambda=\left\{\left\{ip+1,\ldots,jp\right\}, \left\{1,\ldots,ip-1\right\}\cup\left\{jp+1,\ldots,np\right\}\right\}\in\NC(np).
\]
Therefore, the relation $\tau\vee\oxh=1_{\xh}$ once again cannot be satisfied.
\end{proof}
The previous lemma, along with Proposition \ref{prop: reduction properties}, comprise the central combinatorial observations necessary for a large part of the remaining sections.
Observe  in the case $p=2$ that clause $(i)$ of the previous lemma translates to the fact that the minimum and maximum indices (in the $\xh$-order) of any block $V$ of a partition $\tau\in\BNC(\xh)$ such that $\tau\vee\oxh=1_{\xh}$ cannot be of the form $\sxh(k)$ and $\sxh(m)$ where $k$ is odd and $m$ is even, unless $k=1$ and $m=2n$. Similarly,  clause $(ii)$ implies that if $\sxh(k)$ and $\sxh(m)$ are consecutive indices of $V$ in the $\xh$-order, then it cannot be the case that $k$ is even and $m$ is odd, unless $m=k+1$.

As a follow-up  we state the next result, which allows us to characterize (in the case $p=2$) the partitions $\tau$ all of whose blocks have even cardinality and such that $\tau\vee\oxh=1_{\xh}$. This result will be invoked frequently in the next sections in the computation of bi-free cumulants such that the $*$-moments of odd order of all of their entries are equal to zero (see Propositions \ref{prodbir-bir}, \ref{prod-bi-even} and Lemma \ref{lemmainv1}).
\begin{proposition}\label{evenblocks}
Let  $n\in\N$ and let   $\chi\in\{l,r\}^{2n}$ and $\oxh$ be given as in Notation  \ref{notation: products}, so that
\[
\xh(2i-1)=\xh(2i)=\chi(i)\text{ for all }i=1,\ldots,n\text{ and }\oxh=\left\{\left\{2i-1,2i\right\}:i=1,\ldots,n\right\}\in\BNC(\xh).
\]
For a bi-non-crossing partition $\tau\in \BNC(\xh)$ the following are equivalent:
\begin{enumerate}[(i)]
\item $\tau\vee\oxh = 1_{\xh}$ and every block of $\tau$ contains an even number of elements,
\item  $\sxh(1)\sim_\tau \sxh(2n)$  and  $\sxh(2i) \sim_\tau\sxh(2i+1)$ for every $i=1,\ldots,n-1 .       $
\end{enumerate}
\end{proposition}
\begin{proof}
If $\tau\in\BNC(\xh)$ satisfies property $(ii)$ above, then clearly every block of $\tau$ contains an even number of elements. Let $\lambda\in\BNC(\xh)$ be such that $\tau,\oxh\leq\lambda$. For $1\leq i\leq n-1$, we have
\[
\sxh(2i-1)\sim_{\oxh}\sxh(2i)\text{ and }\sxh(2i)\sim_\tau\sxh(2i+1),
\]
therefore as $\tau,\oxh\leq\lambda$ we obtain
\[
\sxh(2i-1)\sim_\lambda\sxh(2i)\sim_\lambda\sxh(2i+1),
\]
which, along with the fact that $\sxh(1)\sim_\lambda\sxh(2n)$, shows that $\sxh(k)\sim_\lambda \sxh(k+1)$ for all $1\leq k\leq 2n-1$. Thus $\lambda=1_{\xh}$ and this proves the implication $(ii)\implies (i)$.

For the converse, let $\tau\in \BNC(\chi)$ be such that $\tau\vee\oxh=1_{\chi}$ and every block of $\tau$ contains an even number of elements. We start with the following observation: let $V\in\tau$ and let $1\leq q < q'\leq 2n$ be such that $\sxh(q)=\min_{\prec_{\xh}}V$ and $\sxh(q')=\max_{\prec_{\xh}} V$ (equivalently, $q=\min_<\sxh^{-1}(V)$ and $q'=\max_<\sxh^{-1}(V)$). We claim that it must be the case that one of $q, q'$ is  an even number and the other is  odd. Indeed, arguing by way of contradiction, suppose that both $q$ and $q'$ are odd, so that there exist $1\leq m < k\leq n$ such that $q=2m+1$ and $q'=2k+1$. This situation is depicted in the following diagram.
\begin{equation*}
\begin{tikzpicture}[baseline]
			\draw[thick,dashed] (2,0) -- (13,0);
   \filldraw (2.1,0) circle (0.06) node[anchor=north] {$\sxh(1)$};
     \filldraw (3.1,0) circle (0.06) node[anchor=north] {$\sxh(2)$};
      
        \filldraw (3.7,-0.3) circle (0.02) node[anchor=north] {};
        \filldraw (4.3,-0.3) circle (0.02) node[anchor=north] {};
        \filldraw (4.9,-0.3) circle (0.02) node[anchor=north] {};

        \filldraw (6.1,0) circle (0.06) node[anchor=north] {$\sxh(2m+1)$};
         \filldraw (7.2,-0.3) circle (0.02) node[anchor=north] {};
         \filldraw (7.8,-0.3) circle (0.02) node[anchor=north] {};
         \filldraw (8.4,-0.3) circle (0.02) node[anchor=north] {};
         \filldraw (9,-0.3) circle (0.02) node[anchor=north] {};
          \filldraw (10.1,0) circle (0.06) node[anchor=north] {$\sxh(2k+1)$};
                \filldraw (11.3,-0.3) circle (0.02) node[anchor=north] {};
                   \filldraw (11.9,-0.3) circle (0.02) node[anchor=north] {};
                   \filldraw (12.9,0) circle (0.06) node[anchor=north] {$\sxh(2n)$};

\draw [thick, black,decorate,decoration={brace,amplitude=10pt},xshift=0.4pt,yshift=-0.4pt](6.15,0.2) -- (10.05,0.2) node[black,above=32pt, midway,yshift=-0.6cm,align=left] {\footnotesize union of blocks of $\tau$ \\ \footnotesize of odd cardinality};

 \draw[-latex,line width=1pt] (6.1,-1.3)--(6.1,-0.5);
  \node[draw] at (6.1,-1.6) {${\min}_{\prec_{\xh}}V$};
   \draw[-latex,line width=1pt] (10.1,-1.3)--(10.1,-0.5);
  \node[draw] at (10.1,-1.6) {${\max}_{\prec_{\xh}}V$};
\draw[thick, black] (6.1,-2) -- (6.1,-2.4);
\draw[thick, black] (10.1,-2) -- (10.1,-2.4);
 \draw[-latex,line width=1pt] (6.8,-2.2)--(6.1,-2.2);
  \draw[-latex,line width=1pt] (9.2,-2.2)--(10.1,-2.2);
  \draw[thick,dashed, line width=1pt] (6.8,-2.2) -- (7.8,-2.2);
  \draw[thick,dashed, line width=1pt] (8.4,-2.2) -- (9.2,-2.2);
 \node[draw] at (8.1,-2.2) {$V$};             
\draw[thick, black] (6.1,0) -- (6.1,1.8) -- (10.1,1.8) -- (10.1,0);	
	
			\end{tikzpicture}
\end{equation*}
The $\xh$-interval  $\{\sxh(i) : 2m+1\leq i\leq 2k+1\}$ must be written as a union of blocks of $\tau$ (since $\tau$ is bi-non-crossing) and has clearly odd cardinality. Thus $\tau$ must contain at least one block with an odd number of elements, contradicting our initial assumption. In a similar way we show that  $q, q'$ cannot both be even, hence proving the aforementioned claim.

If $V\in\tau$ is such that $\sxh(1)\in V$, then clearly $\sxh(1)=\min_{\prec_{\xh}} V$, thus if  $\sxh(q)=\max_{\prec_{\xh}}V$, then $q$ must be even and by Lemma \ref{lemma: no block coupling} we necessarily have $q=2n$, which shows that $\sxh(1)\sim_\tau\sxh(2n)$. Next,  let $i\in\{1,\ldots,n-1\}$ and $V\in\tau$ such that $\sxh(2i)\in V$. Observe that it cannot be that case that $\sxh(2i)=\max_{\prec_{\xh}}V$, since then there would exist an odd index $q<2i$ with $\sxh(q)=\min_{\prec_{\xh}}V$, which would contradict the conclusion of Lemma \ref{lemma: no block coupling}. As as result, there exists an index in $V$ that is greater than $\sxh(2i)$ (in the $\xh$-order), i.e. there exists  $q>2i$  such that $\sxh(q)\in V$. We may evidently assume that $\sxh(2i)$ and $\sxh(q)$ are consecutive indices of $V$ in the $\xh$-order. If the $\xh$-interval $\left\{\sxh(j): 2i+1\leq j\leq q-1\right\}$ between the consecutive indices $\sxh(i)$ and $\sxh(q)$ of $V$ is non-empty, then  it has to be written as a union of blocks of $\tau$ (as $\tau$ is bi-non-crossing) and hence has even cardinality. Therefore $q$ must be odd and we may write $q=2m+1$ for some $m\geq i$. However, if $m>i$, then the aforementioned $\xh$-interval is non-empty and is written both as a union of blocks of $\tau$ and $\oxh$, as depicted in the next diagram (with the indices in the $\xh$-order).
\begin{equation*}
\begin{tikzpicture}[baseline]
			\draw[thick,dashed] (2,0) -- (13,0);
   \filldraw (2.1,0) circle (0.06) node[anchor=north] {$\sxh(1)$};
 
       \filldraw (2.7,-0.3) circle (0.02) node[anchor=north] {};
        \filldraw (3.1,-0.3) circle (0.02) node[anchor=north] {};
        \filldraw (3.5,-0.3) circle (0.02) node[anchor=north] {};

        \filldraw (4.2,0) circle (0.06) node[anchor=north] {$\sxh(2i)$};
         \filldraw (5.7,0) circle (0.06) node[anchor=north] {$\sxh(2i+1)$};
         \filldraw (6.8,-0.3) circle (0.02) node[anchor=north] {};
         \filldraw (7.2,-0.3) circle (0.02) node[anchor=north] {};
         \filldraw (7.6,-0.3) circle (0.02) node[anchor=north] {};
         \filldraw (8.5,0) circle (0.06) node[anchor=north] {$\sxh(2m)$};
         \filldraw (10.2,0) circle (0.06) node[anchor=north] {$\sxh(2m+1)$};
         \filldraw (11.3,-0.3) circle (0.02) node[anchor=north] {};
          \filldraw (11.7,-0.3) circle (0.02) node[anchor=north] {};
          \filldraw (12.1,-0.3) circle (0.02) node[anchor=north] {};
          
                   \filldraw (12.9,0) circle (0.06) node[anchor=north] {$\sxh(2n)$};

\draw [thick, black,decorate,decoration={brace,amplitude=10pt},xshift=0.4pt,yshift=-0.4pt](5.7,0.2) -- (8.5,0.2) node[black,above=32pt, midway,yshift=-0.8cm,align=left] {\footnotesize union of blocks of $\tau$ and  \\ \footnotesize union of blocks of $\oxh$};

\draw[thick, black] (4.2,0) -- (4.2,1.6) -- (10.2,1.6) -- (10.2,0);
\draw[thick, black] (10.2,1.6) -- (10.7,1.6);
\draw[thick,dashed] (10.7,1.6) -- (11.6,1.6);
\draw[thick, black] (4.2,1.6) -- (3.7,1.6);
\draw[thick,dashed] (3.7,1.6) -- (2.8,1.6);
	
			\end{tikzpicture}
\end{equation*}
This again contradicts Lemma \ref{lemma: no block coupling}, therefore we have $m=i$, so that $\sxh(2i)\sim_\tau\sxh(2i+1)$, as desired.
\end{proof}

\subsection{Bi-R-Diagonal Pairs of Operators}\label{sctn:bihaar}
In the context of free probability, R-diagonal operators are characterized by having all of their free $*$-cumulants that are either of odd order, or have entries that are not alternating in $*$-terms and non-$*$-terms equal to zero (see \cite[Lecture 15]{nicaspeicher} for an exposition of the theory of R-diagonal operators in free probability). Specifically, they are defined as follows:
\begin{definition}\label{def:R-diagonal}
Let $(A,\varphi)$ be a non-commutative $*$-probability space. An element $a\in A$ is called \textit{R-diagonal} if for all $n\in\N$ and $a_1,\ldots, a_n\in\{a,a^*\}$ we have that $\kappa_n(a_1,\ldots,a_n)=0$ if $n$ is odd, or if $n$ is even and the sequence $(a_1,\ldots, a_n)$ is \textit{not} of the form
\[
(a,a^*,\ldots,a,a^*)\text{ or }(a^*,a,\ldots,a^*,a),
\]
i.e. if the sequence $(a_1,\ldots,a_n)$ is not alternating in $*$-terms and non-$*$-terms.
\end{definition}
With the previous preliminaries on bi-free cumulants in hand and by adopting the combinatorial approach in the bi-free setting, we will now give the definition of bi-R-diagonal pairs of operators, which will be the central focus of this paper. This definition was first proposed as the correct bi-free generalization of R-diagonal elements in \cite[Section 4]{S2016-2}, but was only used to yield examples of R-cyclic pairs of matrices (see Proposition \ref{rcyclic}).
\begin{definition}\label{birdiagdef}
Let $(A,\varphi)$ be a non-commutative $*$-probability space and $X,Y\in A$. We say that the pair  $(X,Y)$ is \textit{bi-R-diagonal} if for every $n\in\N$, $\chi\in \{l,r\}^n$ and $a_1,\ldots,a_n\in A$ such that 
\[
a_i\in \begin{cases}
       \{X,X^*\}, &\text{if }  \chi (i)=l\\
        \{Y,Y^*\}, &\text{if }  \chi (i) = r\\   
     \end{cases}\quad (i=1,\ldots,n),
\]
we have that:
\begin{enumerate}[(i)]
\item $\kappa_\chi(a_1,\ldots,a_n)=0$, if n is odd
\item ${
\kappa}
_\chi (a_1,\ldots,a_n)=0$, if n is even and the sequence  $(a_{\sx (1)},\ldots, a_{\sx (n)})$ is \emph{not} in one of the following forms:
\begin{enumerate}[(a)]
\item $(Z,Z^*,\ldots,Z,Z^*)$, with $Z\in\{X,X^*,Y,Y^*\},$
\item $(X,X^*,\ldots,X,X^*,Y,Y^*,\ldots,Y,Y^*)$,
\item $(X^*,X,\ldots,X^*,X,Y^*,Y,\ldots,Y^*,Y)$,
\item $(X,X^*,\ldots,X,X^*,X,Y^*,Y,\ldots,Y^*,Y,Y^*)$,
\item $(X^*,X,\ldots,X^*,X,X^*,Y,Y^*,\ldots,Y,Y^*,Y)$,
\end{enumerate}
i.e. whenever the sequence $(a_1,\ldots,a_n)$ is not alternating in $*$-terms and non $*$-terms when read with the indices in the $\chi$-order, with any number of $X$-terms followed by any number of $Y$-terms.
\end{enumerate}
\end{definition}
It is clear from the definition that if the map $\chi$ is constant, then bi-free cumulants reduce to free cumulants and all free cumulants with entries in either $\{X,X^*\}$ (if the map $\chi$  yields the constant value ``$l$'') or $\{Y,Y^*\}$ (if the map $\chi$ yields the constant value ``$r$'') that are of odd order or are not \emph{$*$-alternating} (i.e. are not  alternating in $*$-terms and non-$*$-terms) are equal to zero. In particular, if $(X,Y)$ is a bi-R-diagonal pair, then both $X$ and $Y$ are  R-diagonal operators. Also, it is immediate from the moment-cumulant formula that all joint $*$-moments of odd order of a bi-R-diagonal pair are equal to zero, i.e. if the pair $(X,Y)$ is bi-R-diagonal, then for all $k\in\N$ and $a_1,\ldots,a_{2k+1}\in\{X,X^*,Y,Y^*\}$, it follows that
\[
\varphi(a_1\cdots a_{2k+1})=0.
\]
If the pair $(X,Y)$ is bi-R-diagonal, it is evident from the definition above that if a bi-free cumulant of the form $\kappa_{\chi,\tau}(c_1,\ldots,c_n)$ with entries in the set $\{X,X^*,Y,Y^*\}$ is non-zero, then the partition $\tau\in\BNC(\chi)$ and all of its blocks must necessarily satisfy certain salient properties regarding alternating conditions and their cardinalities.  These properties may be thought of as  ``reductions'' that allow one to distinguish between partitions yielding zero and non-zero cumulant contributions. These follow directly by the multiplicative properties of the bi-free cumulant function and are summarized in the next proposition, which will be frequently invoked in the subsequent sections, especially in connection to the computation of bi-free cumulants with products as arguments.    Clause $(ii)$ of the next proposition refers to cumulants involving bi-free pairs and will also be used regularly in the next sections. \begin{proposition}\label{prop: reduction properties}
Let $(A,\varphi)$ be a non-commutative $*$-probability space, let $X,Y,Z,W\in A$, $n\in\N,\chi\in\{l,r\}^n$ and $\tau\in\BNC(\chi)$.\begin{enumerate}[(i)]
    \item If the pair $(X,Y)$ is bi-R-diagonal, then  whenever $c_1,\ldots,c_n\in A$ are such that
    \[
    c_i\in\begin{cases}
        \{X, X^*\},&\text{if }\chi(i)=l\\
        \{Y, Y^*\},&\text{if }\chi(i)=r
    \end{cases}\quad (i=1,\ldots,n),
    \]
    and such that the bi-free cumulant $\kappa_{\chi,\tau}(c_1,\ldots,c_n)$ is non-vanishing, we have:
    \begin{enumerate}
        \item for all $V\in\tau$, the sequence $(c_1,\ldots,c_n)|_V$ is $*$-alternating when read with the  indices in the induced $\chi|_V$-order,
        \item every block of $\tau$ contains indices corresponding to an equal number of $*$-terms and non-$*$-terms, i.e. for all $V\in\tau$, the cardinalities of the sets
        \[
        \{i\in V: c_i\in\{X,Y\}\}\text{ and }\{i\in V: c_i\in\{X^*,Y^*\}\},
        \]
        are equal,
         \item every block of $\tau$ contains an even number of elements.
    \end{enumerate}
    \item If the pairs $(X,Y)$ and $(Z,W)$ are $*$-bi-free, then whenever $c_1,\ldots,c_n\in A$ are such that
    \[
    c_i\in\begin{cases}
        \{X, X^*, Z, Z^*\},&\text{if }\chi(i)=l\\
        \{Y, Y^*, W, W^*\},&\text{if }\chi(i)=r
    \end{cases}\quad (i=1,\ldots,n),
    \]
    and such that the bi-free cumulant $\kappa_{\chi,\tau}(c_1,\ldots,c_n)$ is non-vanishing, then every block of $\tau$ must contain indices corresponding to elements from exactly one of the pairs $(X,Y)$ or $(Z,W)$, i.e. for all $V\in\tau$ and for all $i\in V$ we have that either $c_i\in\{X, X^*, Y, Y^*\}$ or $c_i\in\{Z, Z^*, W, W^*\}$.
\end{enumerate}    
\end{proposition}
\begin{proof}
For clause $(i)$ we note that
\[
\kappa_{\chi,\tau}(c_1,\ldots,c_n)=\prod_{V\in\tau}\kappa_{\chi|_V}\left((c_1,\ldots,c_n)|_V\right),
\]
therefore if $\kappa_{\chi,\tau}(c_1,\ldots,c_n)\neq 0$, then for $V\in\tau$ the sequence $(c_1,\ldots,c_n)|_V$ must be of even length and be $*$-alternating when read with the indices in the induced $\chi|_V$-order, since the pair $(X,Y)$ is bi-R-diagonal.  In particular, this clearly implies that every block $V$ of $\tau$ must contain an even number of elements and that the indices in $V$ corresponding to elements in the set $\{X,Y\}$ are in bijection with the indices in $V$ corresponding to elements in the set $\{X^*,Y^*\}$. Clause $(ii)$ follows directly by the characterization of bi-freeness in terms of the vanishing of mixed bi-free cumulants (Theorem \ref{equivbifree}).    
\end{proof}

Towards providing canonical examples of bi-R-diagonal pairs of operators, in analogy to the case of free probability and free Haar unitaries, we will define the notion of a bi-Haar unitary pair of operators and compute its bi-free $*$-cumulants. Bi-Haar unitary pairs  will act as both the prototypical examples and building blocks of bi-R-diagonal pairs (see Theorem \ref{invdistr}). First, we recall the definition of a free Haar unitary.

\begin{definition}\label{freehaardef}
Let $(B,\psi)$ be a non-commutative $*$-probability space. A unitary $v\in B$ is called a \textit{Haar unitary} if for all $n\in\Z$ we have that:
\[   
\psi(v^n) = 
     \begin{cases}
       1, &\text{if } \ n=0,\\
       0, &\text{otherwise.} \\
     \end{cases}
\]
\end{definition}
The free $*$-cumulants of a Haar unitary are computed as follows:
\begin{proposition}[\cite{nicaspeicher}, Proposition 15.1]\label{freehaarcum}
If $v\in (B,\psi)$ is a  Haar unitary, then for every $n\in\N$ the non-vanishing free $*$-cumulants of $v$ are given by:
\[
\kappa_{2n}(v,v^*,\ldots,v,v^*) = \kappa_{2n}(v^*,v,\ldots,v^*,v) = (-1)^{n-1}\cdot C_{n-1},
\]
where $C_n$ denotes the $n$-th Catalan number. All other free cumulants with entries in the set $\{v, v^*\}$ vanish.
\end{proposition}
The bi-free generalization of the notion of a Haar unitary was first proposed in \cite[Definition 10.1.2]{CNS2015-1} in the operator-valued  setting.

\begin{definition}\label{bihaardef}
Let $(A,\varphi)$ be a non-commutative $*$-probability space and $u_l, u_r$  be unitaries in $A$. The pair $(u_l, u_r)$ is called a \textit{bi-Haar unitary pair} if the following hold:
\begin{enumerate}[(i)]
\item the algebras $\alg(\{u_l, {u_l}^*\})$ and $\alg(\{u_r, {u_r}^*\})$ commute,
\item for all $n, m\in\Z$ we have that
\[   
\varphi(u_l^n\cdot u_r^m) = 
     \begin{cases}
       1, &\text{if } \ n+m=0,\\
       0, &   \text{otherwise.} \\
     \end{cases}
\]
\end{enumerate}
\end{definition}
In particular, if the pair $(u_l, u_r)$ is a bi-Haar unitary, then both $u_l$ and $u_r$ are free Haar unitaries.
\begin{example}\label{bihaarexample}
Let $G$ be a group with identity $e$ that contains an element of infinite order (i.e. there exists $g_0\in G$ such that $g_0^n\neq e$ for all $n\in\Z\setminus\{0\})$. If $\lambda:G\rightarrow \mathcal{B}(\ell_2(G))$ and $\rho :G\rightarrow \mathcal{B}(\ell_2(G))$ denote the left and right regular representations of $G$ respectively, then
it is straightforward to verify that the pair $(\lambda(g_0), \rho(g_0^{-1}))$ is a bi-Haar unitary pair, with respect to the vector state corresponding to the identity element of $G$. In particular, if $u$ denotes the bilateral shift on $\ell_2(\Z)$, then the pair $(u,u)$ is a bi-Haar unitary.\qed
\end{example} 
The joint distributions of bi-Haar unitary pairs found applications in \cite{CNS2015-1} where, in analogy to the setting of free probability, it was shown that conjugation  by  bi-Haar unitary pairs results in moving  pairs of algebras into bi-free position, while maintaining their joint distributions. Concretely, in the scalar setting the authors obtained the following result:

\begin{theorem}[Scalar case of \cite{CNS2015-1}, Theorem 10.1.3]\label{conjugation}
Let $(A,\varphi)$ be a non-commutative $*$-probability space and let $(u_l, u_r)$ be a bi-Haar unitary pair in $A$. If $(C,D)$ is a pair of subalgebras of $A$ that is bi-free from the pair $(\alg(\{u_l, u_l^*\}), \alg(\{u_r, u_r^*\}))$, then  the pairs of algebras $(u_l^* C u_l, u_r^* D u_r)$ and $(C, D)$ are bi-freely independent and, moreover, the joint distibution of the pair $(u_l^* C u_l, u_r^* D u_r)$ coincides with the joint distribution of $(C, D)$.

\end{theorem}

The commutation assumption on the left and right operators of a bi-Haar unitary pair allows one to reduce the computation of its bi-free cumulants to computing free cumulants of a free Haar unitary. In particular, we have the following:

\begin{proposition}\label{bihaarfreehaar}
Let $(A,\varphi) , (B,\psi)$ be non-commutative $*$-probability spaces and let $(u_l,u_r)$ be a bi-Haar unitary pair in $A$ and $v\in B$  a  Haar unitary. For $n\in\N$ and $\chi\in {\{l,r\}}^n$, let $a_1,\ldots,a_n\in A$ such that  
\[   
a_i\in 
     \begin{cases}
       \{u_l, {u_l}^*\}, &\text{if }  \chi(i) = l\\
       \{u_r, {u_r}^*\}, &   \text{if }   \chi(i) = r \\
     \end{cases}\quad (i=1,\ldots,n),
\]
and define $b_1,\ldots,b_n\in B$ by
\[   
b_i = 
     \begin{cases}
       v, &\text{if } \ a_{\sx(i)}\in \{u_l, u_r\} \\
       v^*, &   \text{if } \ \   a_{\sx(i)} \in \{{u_l}^*, {u_r}^*\} \\
     \end{cases}\quad (i=1,\ldots,n).
\]
Then, we have that
\[
\kappa_\chi (a_1,\ldots,a_n) = \kappa_n (b_1,\ldots,b_n),
\]
with the quantity on the left-hand side of the equation being a bi-free cumulant and the one on the right-hand side being a free cumulant.
\end{proposition}
\begin{proof}
For $n, m\in \Z,$ the following relation between the joint $*$-moments of $(u_l, u_r)$ and the $*$-moments of $v$ is immediate by Definitions \ref{freehaardef} and \ref{bihaardef}
\[
\varphi( {u_l}^n\cdot {u_r}^m) = \psi (v^{n+m}).
\]
Note that since the algebras $\alg(\{u_l,u_l^*\})$ and $\alg(\{u_r,u_r^*\})$ commute, every joint $*$-moment of the pair $(u_l,u_r)$ factorizes in a moment that has a form similar to the left hand-side of the previous expression. The moment-cumulant formulas yield that
\[
\kappa_{\chi}(a_1,\ldots,a_n)= \sum_{\tau\in\BNC(\chi)}\varphi_{\tau}(a_1,\ldots,a_n)\mu_{\BNC}(\tau,1_{\chi}),
\]
and
\[
\kappa_n(b_1,\ldots,b_n) = \sum_{\pi\in\NC(n)}\psi_{\pi}(b_1,\ldots,b_n)\mu_{\NC}(\pi,1_n).
\]
The main observation needed to lead us to the conclusion of the proof is that for all $\tau\in\BNC(\chi)$ and for all $V\in\tau$, we have that
\[
\varphi((a_1,\ldots,a_n)|_V) = \psi\left((b_1,\ldots,b_n)|_{\sx^{-1}(V)}\right).
\]
Indeed, let $\tau\in\BNC(\chi)$ and $V\in\tau$. Define the sets
\[
I_1=\{i\in V : a_i=u_l\}, I_2=\{i\in V : a_i=u_l^*\}, I_3=\{i\in V : a_i=u_r\}\text{ and } I_4=\{i\in V : a_i=u_r^*\},
\]
and let $n_i\in\N$ to be  equal to their respective cardinalities   ($i=1,2,3,4$). Then, by the definition of $b_1,\ldots,b_n$, we have 
\[
\varphi((a_1,\ldots,a_n)|_V)=\varphi\left(u_l^{n_1-n_2}u_r^{n_3-n_4}\right)=\psi\left(v^{n_1+n_3-n_2-n_4}\right)=\psi\left((b_1,\ldots,b_n)|_{\sx^{-1}(V)}\right).
\]
Hence, this implies that for any bi-non-crossing partition $\tau\in\BNC(\chi)$ we obtain
\[
\varphi_{\tau}(a_1,\ldots,a_n)\mu_{\BNC}(\tau,1_{\chi})= \varphi_{\tau}(a_1,\ldots,a_n)\mu_{\NC}(\sx^{-1}\cdot\tau,1_n)=\psi_{\sx^{-1}\cdot\tau}(b_1,\ldots,b_n)\mu_{\NC}(\sx^{-1}\cdot\tau,1_n).
\]
This completes the proof.
\end{proof}

A combination of  Propositions \ref{freehaarcum} and \ref{bihaarfreehaar} gives a complete computation of the bi-free $*$-cumulants involving a  bi-Haar unitary pair.

\begin{corollary}\label{bihaarcum}
Let $(A,\varphi)$ be a non-commutative $*$-probability space and $(u_l, u_r)$ a bi-Haar unitary pair in $A$. Also, let $n\in\N, \chi\in {\{l,r\}}^{2n}$ and $a_1,\ldots,a_{2n}\in A$ such that
\begin{enumerate}[(a)]
\item for all $i=1,\ldots,2n$ we have
\[   
a_i\in 
     \begin{cases}
       \{u_l, {u_l}^*\}, &\text{if }  \chi(i) = l\\
       \{u_r, {u_r}^*\}, &\text{if } \chi(i) = r, \\
     \end{cases}
\]
\item the sequence $(a_{\sx(1)},\ldots, a_{\sx (2n)})$ is $*$-alternating.
\end{enumerate}
Then,
\[
\kappa_\chi (a_1,\ldots,a_{2n}) = (-1)^{n-1}\cdot C_{n-1},
\]
where $C_n$ denotes the $n$-th Catalan number. All other bi-free $*$-cumulants of the pair $(u_l, u_r)$ vanish. In particular, the pair $(u_l, u_r)$ is bi-R-diagonal.
\end{corollary}

\subsection{Operator-Valued Bi-Free Independence and R-cyclic Pairs of Matrices}\label{sctn:bbbb}

In the spirit of \cite{CNS2015-1} and \cite{S2016-2}, we will present the basic definitions regarding operator-valued bi-free independence and a number of results concerning R-cyclic pairs of matrices. The results that are cited will be used in Section \ref{sctn:distributions}
 to discuss an equivalent characterization of the condition of  bi-R-diagonality, which will be formulated in terms of the bi-freeness of certain matrix pairs from scalar matrices with amalgamation over the diagonal scalar matrices (see Theorem \ref{together}).

\begin{definition}
Let $B$ be a unital algebra.
\begin{enumerate}[(i)]
\item A \emph{$B$-$B$-bimodule with specified $B$-vector state} is a triple $(\mathcal{X},\overset{\circ}{\mathcal{X}},p)$ where $\mathcal{X}$ is a direct sum of $B$-$B$-bimodules
\[
\mathcal{X} = B\oplus\overset{\circ}{\mathcal{X}},
\]
and $p:\mathcal{X}\rightarrow B$ is the linear map given by
\[
p(b\oplus\eta)=b,
\]
for all $b\in B$ and $\eta\in\overset{\circ}{\mathcal{X}}$. On $\mathcal{L}(\mathcal{X})$, the space of linear maps on $\mathcal{X}$, we  define the \textit{expectation of $\mathcal{L}(\mathcal{X})$ onto $B$} which is    the linear map given by 
\[
E_{\mathcal{L}(\mathcal{X})}(T) = p\left(T\left(1_B\oplus 0\right)\right),
\]
for all $T\in\mathcal{L}(\mathcal{X})$.
\item A \emph{$B$-$B$-non commutative probability space} is a triple $(A,E_A,\varepsilon)$, where $A$ is a unital algebra, $\varepsilon:B\otimes {B}^{\text{op}}\rightarrow A$ is a unital homomorphism such that both maps ${\varepsilon}|_{B\otimes 1_B}$ and ${\varepsilon}|_{1_B\otimes B^{\text{op}}}$ are injective and $E_A:A\rightarrow B$ is a linear map such that
\[
E_A(\varepsilon(b_1\otimes b_2)Z)= b_1 E_A(Z)b_2,
\]
and
\[
E_A(Z\varepsilon(b\otimes 1_B)) = E_A(Z\varepsilon(1_B\otimes b)),
\]
for all $b, b_1,b_2\in B$ and $Z\in A$. The unital subalgebras of $A$ defined as
\[
A_l = \{Z\in A : Z\varepsilon(1_B\otimes b)=\varepsilon(1_B\otimes b)Z \text{ for all }b\in B\},
\]
and
\[
A_r = \{Z\in A : Z\varepsilon(b\otimes 1_B)=\varepsilon(b\otimes 1_B)Z \text{ for all }b\in B\},
\]
are called the \emph{left} and \emph{right algebras} of $A$ respectively.
\item A \emph{pair of $B$-faces} in a $B$-$B$-non commutative probability space $ (A,E_A,\varepsilon)$ consists of a pair $(C,D)$ of unital subalgebras of $A$ such that
\[
\varepsilon(B\otimes 1_B)\subseteq C\subseteq A_l \text{ and } \varepsilon(1_B\otimes {B}^{\text{op}})\subseteq D\subseteq A_r.
\]\item A family ${\{(C_k,D_k)\}}_{k\in K}$ of pairs of $B$-faces in a $B$-$B$-non commutative probability space $(A,E_A,\varepsilon)$ is said to be \emph{bi-free with amalgamation over $B$} if there exist $B$-$B$-bimodules with specified $B$-vector states ${\{(\mathcal{X}_k,\overset{\circ}{\mathcal{X}_k},p_k)\}}_{k\in K}$ and unital homomorphisms $l_k:C_k\rightarrow\mathcal{L}_l(\mathcal{X}_k)$ and $r_k:D_k\rightarrow\mathcal{L}_r(\mathcal{X}_k)$ such that the joint distribution of ${\{(C_k,D_k)\}}_{k\in K}$ with respect to $E_A$ is equal to the joint distribution of the images of
\[
{\{((\lambda_k\circ l_k)(C_k),(\rho_k\circ r_k)(D_k))\}}_{k\in K}
\]
inside $\mathcal{L}(*_{k\in K}\mathcal{X}_k)$ with respect to $E_{\mathcal{L}(*_{k\in K}\mathcal{X}_k)}$, where $\lambda_k$ and $\rho_k$ denote the left and right regular representations onto $\mathcal{X}_k\subseteq *_{k\in K}\mathcal{X}_k$, respectively. 

If $S_k\subseteq A_l$ and $V_k\subseteq A_r$ for all $k\in K$, we will say that the family ${\{(S_k,V_k)
\}}_{k\in K}$ is bi-free with amalgamation over $B$ if the family 
\[
{\{(\alg(\varepsilon(B\otimes 1_B)\cup S_k),\alg(\varepsilon(1_B\otimes {B}^{\text{op}})\cup V_k))\}}_{k\in K},
\]
of pairs of $B$-faces is bi-free with amalgamation over $B$.

\end{enumerate}
\end{definition}

See \cite[Section 3]{CNS2015-1} for a discussion on why $B$-$B$-non-commutative probability spaces are the correct framework to formulate the notions of operator-valued bi-free probability. There, a combinatorial approach was adopted and the bi-multiplicative operator-valued bi-free cumulant maps were defined and  used to characterize operator-valued bi-free independence.

Let $(A,\varphi)$ be a non-commutative $*$-probability space and let $d\in\N$. In the algebra $\mathcal{M}_d(A)$ of all $d\times d$ matrices over $A$, consider the unital subalgebras $\mathcal{M}_d(\C)$ and $\mathcal{D}_d$ consisting of all scalar matrices and all diagonal scalar matrices respectively and let $F_d:\mathcal{M}_d(\C)\rightarrow \mathcal{D}_d$ denote the conditional expectation onto the diagonal. We will  recall from \cite[Section 4]{S2016-2} the process on how to turn $\mathcal{L}(\mathcal{M}_d(A))$, the space of all linear maps on $\mathcal{M}_d(A)$, into a $\mathcal{M}_d(\C)$-$\mathcal{M}_d(\C)$-non-commutative probability space. We will denote by $A=[a_{i,j}]$ a matrix whose ${(i,j)}$-th entry equals $a_{i,j}$. Define the unital, linear map $\varphi_d:\mathcal{M}_d(A)\rightarrow \mathcal{M}_d(\C)$ by
\[
\varphi_d (T)= [\varphi (T_{i,j})],\text{ for all }T=[T_{i,j}]\in\mathcal{M}_d(A).
\]
Also, for $A=[a_{i,j}]\in\mathcal{M}_d(\C),$ let
\[
L_A(T) = \left[\sum_{k=1}^d a_{i,k}T_{k,j}\right]
\text{ and }R_A(T) = \left[\sum_{k=1}^d a_{k,j}T_{i,k}\right],\text{ for all }T=[T_{i,j}]\in\mathcal{M}_d(A).
\]
Then, if $\varepsilon:\mathcal{M}_d(\C)\otimes {\mathcal{M}_d(\C)}^{\text{op}}\rightarrow \mathcal{L}(\mathcal{M}_d(A))$ is defined as 
\[
\varepsilon\left(A\otimes A'\right)=L_A R_{A'},\text{ for }A=[a_{i,j}], A'=[a'_{i,j}]\in\mathcal{M}_d(\C),
\]
and $E_d:\mathcal{L}(\mathcal{M}_d(A))\rightarrow \mathcal{M}_d(\C)$ is given by
\[
E_d(Z) = \varphi_d(Z(\iden_d)),\text{ for } Z\in\mathcal{L}(\mathcal{M}_d(A)),
\]
where $\iden_d$ denotes the $d\times d$ identity matrix, we have that the triple $(\mathcal{L}(\mathcal{M}_d(A)),E_d,\varepsilon)$ is a $\mathcal{M}_d(\C)$-$\mathcal{M}_d(\C)$-non-commutative probability space. We will also need the unital homomorphisms $L:\mathcal{M}_d(A)\rightarrow {\mathcal{L}(\mathcal{M}_d(A))}_l$ and $R:{\mathcal{M}_d(A^{\text{op}})}^{\text{op}}\rightarrow {\mathcal{L}(\mathcal{M}_d(A))}_r$ given by
\[
L(S)(T) = \left[\sum_{k=1}^d S_{i,k}T_{k,j}\right]\text{ and }R(S)(T) = \left[\sum_{k=1}^d S_{k,j}T_{i,k}\right],\text{ for all }S=[S_{i,j}], T=[T_{i,j}]\in\mathcal{M}_d(A).
\]
In the setting of free probability, there is a connection between   R-diagonal operators and  R-cyclic matrices (see \cite[Example 20.5]{nicaspeicher}). In the bi-free setting, R-cyclic pairs of matrices were first defined and studied in \cite{S2016-2}.

\begin{definition}\cite[Definition 4.4]{S2016-2}\label{rcyclicdef}
Let $(A,\varphi)$ be a non-commutative probability space, $d\in\N$, $I,J$  be disjoint index sets and let  ${\{[Z_{k;i,j}]\}}_{k\in I}\cup{\{[Z_{k;i,j}]\}}_{k\in J}\subseteq\mathcal{M}_d(A)$. The pair $(\{[Z_{k;i,j}]\}_{k\in I}, \{[Z_{k;i,j}]\}_{k\in J})$ is called \emph{R-cyclic} if for all $n\in\N, \omega:\{1,\ldots,n\}\rightarrow I\sqcup J$ and $1\leq i_1,\ldots,i_n,j_1,\ldots,j_n\leq d$, by defining $\chi\in {\{l,r\}}^n$ as
\[
\chi(i)= \begin{cases}
       l, &\text{if }  \omega (i)\in I\\
        r, &\text{if } \omega (i)\in J\\   
     \end{cases}  \quad (i=1,\ldots,n),
\]
we have that
\[
{\kappa}_{\chi}(Z_{\omega(1);i_1,j_1},Z_{\omega(2);i_2,j_2}
,\ldots,Z_{\omega(n);i_n,j_n})=0,
\]
whenever at least one of the relations
\[
j_{\sx(1)}=i_{\sx(2)}, j_{\sx(2)}=i_{\sx(3)},\ldots,j_{\sx(n-1)}=i_{\sx(n)},j_{\sx(n)}=i_{\sx(1)},
\]
is \emph{not} satisfied.
\end{definition}

The following result was mentioned (but not proved) in \cite[Section 4]{S2016-2} and we include the proof for the convenience of the reader.

\begin{proposition}\label{rcyclic}
Let $(A,\varphi)$ be a non-commutative $*$-probability space and let $X,Y\in A$. The following are equivalent:
\begin{enumerate}[(i)]
\item the pair $(X,Y)$ is bi-R-diagonal,
\item in  $\mathcal{M}_2(A)$, the pair $(Z,W)$ defined as 
\[
Z={[Z_{i,j}]} = \begin{bmatrix}
0 & X\\ X^* &0
\end{bmatrix} \text{ and }W={[W_{i,j}]} = \begin{bmatrix}
0 & Y\\ Y^* &0
\end{bmatrix},
\]
is R-cyclic.

\end{enumerate}
\end{proposition}

\begin{proof}
Let $n\in\N,\chi\in\{l,r\}^n$ and $a_{i_1,j_1},\ldots,a_{i_n,j_n}\in A$ be such that  $a_{i_m,j_m}=Z_{i_m,j_m}$ if $\chi(m)=l$, while $a_{i_m,j_m}=W_{i_m,j_m}$ if $\chi(m)=r$, for all $1\leq m\leq n$. We may clearly assume that each $a_{i_m,j_m}$ is non-zero, thus $i_m\neq j_m$ for all $1\leq m\leq n$.
The main observation that will make the equivalence of the proposition apparent is that the condition that at least one of the relations 
\[
j_{\sx(1)}=i_{\sx(2)}, j_{\sx(2)}=i_{\sx(3)},\ldots,j_{\sx(n-1)}=i_{\sx(n)},j_{\sx(n)}=i_{\sx(1)},
\]
is not satisfied is equivalent to the statement that the sequence
\[
\left(a_{i_{\sx(1)},j_{\sx(1)}},\ldots, a_{i_{\sx(n)},j_{\sx(n)}}\right),
\]
is either non-$*$-alternating, or is of odd length. Indeed, first suppose that $j_{\sx(m)}\neq i_{\sx(m+1)}$ for some $m\in\{1,\ldots,n-1\}$ and notice that this implies that we must have
\[
i_{\sx(m)}=i_{\sx(m+1)}\text{ and }j_{\sx(m)}=j_{\sx(m+1)}.
\]
But this is equivalent to stating that the elements $a_{i_{\sx(m)},j_{\sx(m)}}$ and $a_{i_{\sx(m+1)},j_{\sx(m+1)}}$ both correspond to either $*$-terms or non-$*$-terms and hence the sequence
\[
\left(a_{i_{\sx(1)},j_{\sx(1)}},\ldots, a_{i_{\sx(n)},j_{\sx(n)}}\right),
\]
is not $*$-alternating. Next, assume that $j_{\sx(n)}\neq i_{\sx(1)}$. As before, we must have that 
\[
i_{\sx(1)}=i_{\sx(n)}\text{ and } j_{\sx(1)}=j_{\sx(n)}.
\]
This is equivalent to stating that the first and last terms of the sequence
\[
\left(a_{i_{\sx(1)},j_{\sx(1)}},\ldots, a_{i_{\sx(n)},j_{\sx(n)}}\right),
\]
both correspond to either $*$-terms or non-$*$-terms, which means that  this sequence either is non-$*$-alternating, or is of odd length.
\end{proof}

The main result we will need for Theorem \ref{together} concerns the following  characterization of R-cyclic pairs.

\begin{theorem}[\cite{S2016-2}, Theorem 4.9]\label{thmrcyclic}
Let $(A,\varphi)$ be a non-commutative probability space, $d\in\N$, $I,J$ be disjoint index sets and  let ${\{[Z_{k;i,j}]\}}_{k\in I}\cup{\{[Z_{k;i,j}]\}}_{k\in J}\subseteq\mathcal{M}_d(A)$. The following are equivalent:
\begin{enumerate}[(i)]
\item the pair $\left({\{[Z_{k;i,j}]\}}_{k\in I},{\{[Z_{k;i,j}]\}}_{k\in J}\right)$ is R-cyclic,
\item  the family $\left({\{L([Z_{k;i,j}])\}}_{k\in I},{\{R([Z_{k;i,j}])\}}_{k\in J}\right)
$
is bi-free from
$
\left(L(\mathcal{M}_d(\C)), R\left({\mathcal{M}_d(\C)}^{\text{op}}\right)\right)
$
with amalgamation over $\mathcal{D}_d$ with respect to $F_d\circ E_d$.

\end{enumerate}
\end{theorem}

\section{Operations Involving Bi-R-Diagonal Pairs}\label{sctn:products}
In this section, we will study the behaviour of bi-R-diagonal pairs of operators under the taking of sums, products and arbitrary powers, where, in most cases, a $*$-bi-free independence condition will be assumed. The proofs obtained will indicate that most of the results that hold for free R-diagonal elements (see \cite{approach} and \cite[Lecture 15]{nicaspeicher}) have corresponding generalizations in the bi-free setting. The last part of this section will be devoted to the study of joint distributions of pairs formed by freely and classically independent R-diagonal operators.
We begin with the following proposition  regarding sums of $*$-bi-free bi-R-diagonal pairs.

\begin{proposition}\label{birsum}
Let $(A,\varphi)$ be a non-commutative $*$-probability space and $X,Y,Z,W\in A$  such that:
\begin{enumerate}[(a)]
\item the pairs $(X,Y)$ and $(Z,W)$ are both bi-R-diagonal,
\item the pairs $(X,Y)$ and $(Z,W)$ are $*$-bi-free.
\end{enumerate}
Then, the pair $(X+Z,Y+W)$ is also bi-R-diagonal.
\end{proposition}

\begin{proof}
Let $n\in\N,\chi\in {\{l,r\}}^n$ and $a_1,\ldots,a_n\in A$ such that 
\[
a_i\in \begin{cases}
       \{X+Z,X^*+Z^*\}, &\text{ if }  \chi (i)=l\vspace{5pt}\\
        \{Y+W,Y^*+W^*\}, &\text{ if }  \chi (i) = r\\   
     \end{cases}\quad (i=1,\ldots,n).
\]
Define $b_1,\ldots,b_n\in\left\{X,X^*,Y,Y^*\right\}$ and $ c_1,\ldots,c_n\in\left\{Z,Z^*,W,W^*\right\}$ such that $a_i=b_i+c_i$ for all $i=1,\ldots,n$ (for instance, if $a_i=X+Z$, then $b_i=X$ and  $c_i=Z$). 
The multi-linearity of the bi-free cumulants combined with the $*$-bi-free independence of the pairs $(X,Y)$ and $(Z,W)$  imply that
\[
{\kappa}_{\chi}(a_1,\ldots,a_n) = {\kappa}_{\chi}(b_1,\ldots,b_n) + {\kappa}_{\chi}(c_1,\ldots,c_n).
\]
The conclusion of the proposition follows from the observation that the sequence
\[
\left(a_{\sx(1)},\ldots,a_{\sx(n)}\right),
\]
is $*$-alternating  if and only if both  the sequences
\[
\left(b_{\sx(1)},\ldots,b_{\sx(n)}\right)\text{ and }\left(c_{\sx(1)},\ldots,c_{\sx(n)}\right),
\]
are also $*$-alternating.
\end{proof}

With the previous proof in mind, it is easy to see that if exactly one of the $*$-bi-free pairs $(X,Y)$ and $(Z,W)$ is bi-R-diagonal, then the pair $(X+Z,Y+W)$ cannot be bi-R-diagonal.

We now proceed to study various cases on products involving bi-R-diagonal pairs. The products of pairs will be considered pointwise, with the condition that the order of the right operators is reversed being necessary for the results concerning the more general cases (see Theorem \ref{prodbir-bifree} below and also Proposition \ref{prod-bi-even}). The proofs of these results will require more delicate arguments when compared to the cases of sums involving bi-R-diagonal pairs and, for this, the formula for bi-free cumulants with products of operators as arguments will play a key role. The next theorem states that the product of a bi-R-diagonal pair of operators by any $*$-bi-free pair is also bi-R-diagonal and  exhibits the fact that bi-R-diagonal pairs exist in abundance.

\begin{theorem}\label{prodbir-bifree}
Let $(A,\varphi)$ be a non-commutative $*$-probability space and let $X,Y,Z,W\in A$ such that:
\begin{enumerate}[(a)]
\item the pair $(X,Y)$ is bi-R-diagonal,
\item the pairs $(X,Y)$ and $(Z,W)$ are $*$-bi-free.
\end{enumerate}
Then, the pair $(XZ,WY)$ is also bi-R-diagonal.
\end{theorem}

\begin{proof}
Let $n\in\N$, $\chi\in\{l,r\}^n$ and $a_1,\ldots,a_n\in A$ be such that 
\[
a_i\in \begin{cases}
       \{XZ,Z^*X^*\}, &\text{if } \chi (i)=l\\
        \{WY,Y^*W^*\}, &\text{if }  \chi (i) = r\\   
     \end{cases}\quad  (i=1,\ldots,n).
\]
Also let $\xh\in\{l,r\}^{2n},\oxh\in\BNC(\xh)$ and $c_1,\ldots,c_{2n}\in A$ be as in  Notation \ref{notation: products}, so that by Theorem \ref{openup} we have
\[
\kappa_\chi(a_1,\ldots,a_n) =  \sum_{\substack{\tau\in \BNC(\xh)\\
\tau\vee\oxh = 1_{\xh}\\}}
        \kappa_{\xh,\tau}(c_1,\ldots,c_{2n})=  \sum_{\substack{\tau\in \BNC(\xh)\\
\tau\vee\oxh = 1_{\xh}\\}}
\prod_{V\in\tau}       \kappa_{{\xh}|_V}\left({(c_1,\ldots,c_{2n})}|_{V}\right).
\]
To start, we make a preparatory observation:  if $i\in\{1,\ldots,n\}$ is such that $a_{\sx (i)}=XZ$, then $c_{\sxh (2i-1)}=X$ and $c_{\sxh (2i)}=Z$ (a similar situation occurs when $a_{\sx (i)}=Z^*X^*$, since this corresponds to a left operator). On the other hand, if $a_{\sx (i)}=WY$, then $c_{\sxh (2i-1)}=Y$ and $c_{\sxh (2i)}=W$ (and a similar situation occurs when $a_{\sx (i)}= Y^*W^*$ since this corresponds to a right operator). Note that in the latter case, the  right operators must appear reversed in the $\xh$-order. In particular, this implies for $1\leq q\leq 2n$ that  $c_{\sxh(q)}\in\{X,Z^*,Y,W^*\}$ if and only if $q$ is  an odd index, while $c_{\sxh(q)}\in\{Z,X^*,W,Y^*\}$ if and only if  $q$ is even.

If a bi-non-crossing partition $\tau\in\BNC(\xh)$ yields a non-zero contribution to the sum of cumulants above, then $\tau$ satisfies $\tau\vee\oxh=1_{\xh}$, therefore the conclusion of Lemma \ref{lemma: no block coupling} holds for $\tau$. 
Secondly, since the pair $(X,Y)$ is bi-R-diagonal and is $*$-bi-free from the pair $(Z,W)$, Proposition \ref{prop: reduction properties} 
implies that for any block $V$ of the contributing partition $\tau$
we have that either $\{c_i:i\in V\}\subseteq \{X,X^*,Y,Y^*\}$ or $\{c_i:i\in V\}\subseteq \{Z,Z^*,W,W^*\}$ and, moreover, if  $\{c_i:i\in V\}\subseteq\{X,X^*,Y,Y^*\}$, then as the bi-free cumulant $\kappa_{\xh|_V}((c_1,\ldots,c_{2n})|_V)$ has to be non-zero and has entries in the bi-R-diagonal pair $(X,Y)$, we deduce that
$V$ must have even cardinality and the sequence  $(c_1,\ldots,c_{2n})|_V$ must be $*$-alternating  when read with the indices in the $\xh|_V$-order. This implies that if $n$ is odd, then $\kappa_\chi (a_1\ldots,a_n)=0$, as then the cardinality of the set 
\[
\{j\in\{1,\ldots,2n\} : c_j\in \{X,X^*,Y,Y^*\} \},
\]
is odd and hence for any contributing partition $\tau\in \BNC(\xh)$ there exists $V\in\tau$ with odd cardinality that contains indices corresponding to elements in $\{X,X^*,Y,Y^*\}$.

Going forward, we  assume that $n$ is even. We must show that the cumulant $\kappa_{\chi}(a_1,\ldots,a_n)$ vanishes if the sequence $(a_{\sx(1)},\ldots,a_{\sx(n)})$ is not $*$-alternating. Suppose the following situation occurs:
\[
\left(a_{\sx(1)},\ldots,a_{\sx(n)}\right) = (\ldots, XZ, XZ,\ldots),
\]
with $a_{\sx(m)} = a_{\sx(m+1)} = XZ$ for some $m\in\{1,\ldots,n-1\}$. This implies the following situation for the $\xh$-order:
\[
\left(c_{\sxh(1)},\ldots,c_{\sxh(2n)}\right) = (\ldots,X,Z,X,Z,\ldots),
\]
with $c_{\sxh(2m-1)} = c_{\sxh(2m+1)} = X$ and $c_{\sxh(2m)} = c_{\sxh(2m+2)} = Z$. For a contributing partition $\tau\in\BNC(\xh)$, let $V\in\tau$ be such that $\sxh(2m+1)\in V$. To start, consider the case when $\sxh(2m+1)={\min}_{{\prec}_{\xh}}V$, so that $\sxh(2m+1)$ is the first index that appears in $V$ in the $\xh$-order. If $q\in \{1,\ldots,2n\} $ is such that $\sxh(q)={\max}_{\prec_{\xh}}V$, 
 then since the sequence $(c_1,\ldots,c_{2n})|_V$ when read in the induced $\xh|_V$-order is of the form $(X,\ldots,c_{\sxh(q)})$ and must be $*$-alternating,  its last term  corresponds to a $*$-term, i.e. $c_{\sxh(q)}\in\{X^*,Y^*\}$. But, as noted above, this implies that $q$ must be an even number, so that the $\xh$-interval $\{\sxh(j):2m+1\leq j\leq q\}$ between the minimum and maximum indices of $V$ (in the $\xh$-order)  is written both as a union of blocks  $\oxh$. The next diagram illustrates this case with the sequence $(c_1,\ldots,c_{2n})$ read in the $\xh$-order.
\begin{equation*}
\begin{tikzpicture}[baseline]
			\draw[thick,dashed] (2,0) -- (12.9,0);
          \filldraw (2.4,-0.3) circle (0.02) node[anchor=north] {};
           \filldraw (2.9,-0.3) circle (0.02) node[anchor=north] {};
            \filldraw (3.4,-0.3) circle (0.02) node[anchor=north] {};
             \filldraw (3.9,0) circle (0.06) node[anchor=north] {$X$};
             \filldraw (4.9,0) circle (0.06) node[anchor=north] {$Z$};
              \filldraw (5.9,0) circle (0.06) node[anchor=north] {$X$};
               \filldraw (6.9,0) circle (0.06) node[anchor=north] {$Z$};
              \filldraw (7.6,-0.3) circle (0.02) node[anchor=north] {};
              \filldraw (8.1,-0.3) circle (0.02) node[anchor=north] {};
              \filldraw (8.6,-0.3) circle (0.02) node[anchor=north] {};
              \filldraw (9.1,-0.3) circle (0.02) node[anchor=north] {};
              \filldraw (9.9,0) circle (0.06) node[anchor=north] {$W^*$};
              \filldraw (10.9,0) circle (0.06) node[anchor=north] {$Y^*$};
              \filldraw (11.4,-0.3) circle (0.02) node[anchor=north] {};
              \filldraw (11.9,-0.3) circle (0.02) node[anchor=north] {};
              \filldraw (12.4,-0.3) circle (0.02) node[anchor=north] {};
 \draw[-latex,line width=1pt] (5.9,-1.3)--(5.9,-0.5);
  \node[draw] at (5.9,-1.6) {${\min}_{{\prec}_{\xh}}V$};
   \draw[-latex,line width=1pt] (10.9,-1.3)--(10.9,-0.5);
  \node[draw] at (10.9,-1.6) {${\max}_{{\prec}_{\xh}}V$};
\draw[thick, black] (5.9,-2.4) -- (5.9,-2);
\draw[thick, black] (10.9,-2.4) -- (10.9,-2);
 \draw[-latex,line width=1pt] (6.8,-2.2)--(5.9,-2.2);
  \draw[-latex,line width=1pt] (9.7,-2.2)--(10.9,-2.2);
  \draw[thick,dashed, line width=1pt] (6.8,-2.2) -- (8,-2.2);
  \draw[thick,dashed, line width=1pt] (8.6,-2.2) -- (9.7,-2.2);
 \node[draw] at (8.3,-2.2) {$V$};
 \draw [thick, black,decorate,decoration={brace,amplitude=10pt,mirror},xshift=0.4pt,yshift=-0.4pt](5.87,-2.5) -- (10.93,-2.5) node[black,midway,yshift=-0.6cm] {\footnotesize union of blocks of $\tau$ and union of blocks of $\oxh$};
\draw[thick, black] (5.9,0) -- (5.9,1) -- (10.9,1) -- (10.9,0);	
	
			\end{tikzpicture}
\end{equation*}
This contradicts the conclusion of Lemma \ref{lemma: no block coupling}, which shows that the case   $\sxh(2m+1)=\min_{\prec_{\xh}}V$ is impossible. As a result, we find that there exists $q<2m+1$ such that the index $\sxh(q)$ is in $V$. Additionally, we may  assume that   $\sxh(q)$ and $\sxh(2m+1)$ are consecutive indices of $V$ in the $\xh$-order. Since the sequence $(c_1,\ldots,c_{2n})|_V$  when read in the $\xh|_V$-order is of the form $(\ldots, c_{\sxh(q)},X,\ldots)$ and must also be $*$-alternating, it is necessarily the case that $c_{\sxh (q)}=X^*$ (note that any operator that appears before any left operator in the $\xh|_V$-order must also be a left operator). This implies
\[
\left(c_{\sxh(1)},\ldots,c_{\sxh(2n)}\right) = (\ldots,Z^*,X^*,\ldots,X,Z,X,Z,\ldots),
\]
from which we again see that $q$ is even. But this again contradicts Lemma \ref{lemma: no block coupling}, since then the $\xh$-interval $\{\sxh(j):q+1\leq j\leq 2m\}$ between the consecutive indices $\sxh(q)$ and $\sxh(2m+1)$ of $V$  is non-empty and is written both as a union of blocks  $\oxh$.
This case is illustrated in the next diagram, with the indices in the $\xh$-order.
\begin{equation*}
\begin{tikzpicture}[baseline]
			\draw[thick,dashed] (2,0) -- (12.9,0);

          \filldraw (2.6,-0.3) circle (0.02) node[anchor=north] {};
           \filldraw (3.1,-0.3) circle (0.02) node[anchor=north] {};
            \filldraw (3.6,-0.3) circle (0.02) node[anchor=north] {};
            \filldraw (4.2,0) circle (0.06) node[anchor=north] {$Z^*$};
             \filldraw (5.4,0) circle (0.06) node[anchor=north] {$X^*$};
             \filldraw (5.85,-0.3) circle (0.02) node[anchor=north] {};
             \filldraw (6.35,-0.3) circle (0.02) node[anchor=north] {};
             \filldraw (6.85,-0.3) circle (0.02) node[anchor=north] {};
              \filldraw (7.2,0) circle (0.06) node[anchor=north] {$X$};
                          \filldraw (8.4,0) circle (0.06) node[anchor=north] {$Z$};
                          \filldraw (9.6,0) circle (0.06) node[anchor=north] {$X$};
                          \filldraw (10.8,0) circle (0.06) node[anchor=north] {$Z$};
                          \filldraw (11.4,-0.3) circle (0.02) node[anchor=north] {};
                          \filldraw (11.9,-0.3) circle (0.02) node[anchor=north] {};
                          \filldraw (12.4,-0.3) circle (0.02) node[anchor=north] {};
\draw [thick, black,decorate,decoration={brace,amplitude=10pt},xshift=0.4pt,yshift=-0.4pt](5.9,0.2) -- (8.8,0.2) node[black,above=28pt, midway,yshift=-0.6cm,align=left] {\footnotesize union of blocks of $\tau$ and \\ \footnotesize union of blocks of $\oxh$};
          
\draw[thick, black] (5.4,0) -- (5.4,1.6) -- (9.6,1.6) -- (9.6,0);	
\draw[thick, black] (9.6,1.6) -- (10.2,1.6);
\draw[thick, black] (4.8,1.6) -- (5.4,1.6);
\draw[thick, black, dashed] (3.7,1.6) -- (4.8,1.6);
\draw[thick, black, dashed] (10.2,1.6) -- (11.3,1.6);
\node[draw] at (5.4,-1.5) {$c_{\sxh(q)}, q\text{ even}$};
\draw[-latex,line width=1pt] (5.4,-1.25)--(5.4,-0.5);
\node[draw] at (9.6,-1.5) {$c_{\sxh(2m+1)}$};
\draw[-latex,line width=1pt] (9.6,-1.25)--(9.6,-0.5);
			\end{tikzpicture}
			\end{equation*}
 As a result, when
\[
\left(a_{\sx(1)},\ldots,a_{\sx(n)}\right) = (\ldots,XZ,XZ,\ldots),
\]  
for any $\tau\in\BNC(\xh)$ we either have that $\kappa_{\xh,\tau}(c_1,\ldots,c_{2n})=0$ or the partition $\tau$ does not satisfy the relation $\tau\vee\oxh=1_{\xh}$. Thus
the bi-free cumulant ${\kappa}_{\chi}(a_1,\ldots,a_n)$ vanishes and  similar arguments show that this is also the case when the sequence $\left(a_{\sx(1)},\ldots,a_{\sx(n)}\right)$ is in any one of the forms
\[
(\ldots,XZ,WY,\ldots)  \text{ or } (\ldots,WY,WY,\ldots).
\]
The remaining cases for when the sequence $\left(a_{\sx(1)},\ldots,a_{\sx(2n)}\right)$ is not $*$-alternating are
\[
\left(\ldots, Y^*W^*, Y^*W^*,\ldots\right),  \left(\ldots,Z^*X^*,Z^*X^*,\ldots\right) \text{ or } \left(\ldots,Z^*X^*,Y^*W^*,\ldots\right),
\]
and here we briefly only sketch the differences in the corresponding arguments.
If 
\[
\left(a_{\sx(1)},\ldots,a_{\sx(n)}) = (\ldots, Y^*W^*, Y^*W^*,\ldots\right),
\]
with $a_{\sx(m)} = a_{\sx(m+1)} = Y^*W^*$ for some $m\in\{1,\ldots,n-1\}$, then in the $\xh$-order we have
\[
\left(c_{\sxh(1)},\ldots,c_{\sxh(2n)}\right) = (\ldots,W^*,Y^*,W^*,Y^*,\ldots),
\]
with $c_{\sxh(2m-1)} = c_{\sxh(2m+1)} = W^*$ and $c_{\sxh(2m)} = c_{\sxh(2m+2)} = Y^*$.  For  a contributing partition $\tau\in\BNC(\xh)$, if $V\in\tau$ is such that $\sxh(2m)\in V$, this time we note that the case  $\sxh(2m)={\max}_{\prec_{\xh}}V$ is impossible, since otherwise if $\sxh(q)$ denotes the minimum index of $V$ in the $\xh$-order, then $c_{\sxh(q)}$ must be an operator in $\{X,Y\}$ due to the $*$-alternating condition, thus $q$ is odd, which contradicts Lemma \ref{lemma: no block coupling}. Then there exists an index in $V$ greater than $\sxh(2m)$ in the $\xh$-order, and in this case using again Lemma \ref{lemma: no block coupling}  we arrive at a contradiction as done prior.
\end{proof}
The main technical difficulty that necessitates the distinction between several different cases in the previous proof is that we cannot only deal with bi-non-crossing partitions whose blocks contain an even number of elements, thus Proposition \ref{evenblocks} does not apply. This is because the pair $(Z,W)$ need not be bi-R-diagonal and hence bi-free cumulants of odd order with entries in $\{Z,Z^*,W,W^*\}$ need not necessarily vanish.  We remark that for two $*$-bi-free pairs $(X,Y)$ and $(Z,W)$ with the first being bi-R-diagonal, it is not in general true that the pair $(XZ,YW)$ will also be bi-R-diagonal, as the following example indicates. We will denote by $\tr$ the normalized trace on any matrix algebra.

\begin{example}
Let $(A,\varphi)$ be a $\mathrm{C}^*$-probability space and $(u_l,u_r)$ be a bi-Haar unitary pair in $A$. Also, consider the pair $(Z,W)$ in $({\mathcal{M}}_2(\mathbb{C}),\tr)$ given as follows:
\[
Z = \begin{bmatrix}
1 & 0\\ 0 &0
\end{bmatrix} \text{ and } W = \begin{bmatrix}
0 & 0\\ 1 &0
\end{bmatrix}.
\]
We may find a larger  $\mathrm{C}^*$-probability space $(B,\psi)$ such that the pairs $(u_l,u_r)$ and $(Z,W)$ are $*$-bi-free in $(B,\psi)$  and such that $\psi$ preserves the joint $*$-distributions of the pairs $(u_l,u_r)$ and $(Z,W)$ with respect to $\varphi$ and $\tr$ respectively (see Remark \ref{remark:propertiesofbi-freepairs}). We claim that for $\chi=\{l,l,r,r\}$ the bi-free cumulant $\kappa_{\chi}(Z^*{u_l}^*,{u_l}Z,W^*{u_r}^*,{u_r}W)$ does not vanish, even though it is not $*$-alternating in the  $\chi$-order, which implies that the pair $(u_l Z,u_r W)$ is not bi-R-diagonal. Indeed, the moment-cumulant formula yields
\[
\kappa_{\chi}(Z^*{u_l}^*,{u_l}Z,W^*{u_r}^*,{u_r}W)= \sum_{\tau\in\BNC(\chi)}\psi_\tau(Z^*{u_l}^*,{u_l}Z,W^*{u_r}^*,{u_r}W)\mu_{\BNC}(\tau,1_{\chi}).
\]
Since the sets $\{u_l, u_l^*\}$ and $\{Z,Z^*\}$ are freely independent and the sets $\{u_r,u_r^*\}$ and $\{W,W^*\}$ are also freely independent, using the characterization of free independence in terms of moments  it is seen that all operators that appear in the cumulant above are centered, i.e. the following holds
\[
\psi(Z^*u_l^*)=\psi(W^*u_r^*)=\psi(u_l Z)=\psi(u_r W)=0.
\]
Hence, to find bi-non-crossing partitions that  yield a non-zero contribution to the sum above, we may only consider bi-non-crossing partitions  all of whose blocks are not singletons. These are the following three bi-non-crossing partitions:
\[
\tau_1=\{\{1,2\},\{3,4\}\},\tau_2=\{\{1,2,3,4\}\}=1_\chi\text{ and }\tau_3=\{\{1,3\},\{2,4\}\}.
\]
For $\tau_1,$ we have that
\begin{align*}
\psi_{\tau_1}(Z^*u_l^*,u_lZ,W^*u_r^*,u_rW)
&=\psi(Z^*u_l^*u_lZ)\cdot \psi(W^*u_r^*u_rW)\\
&=\tr(Z^*Z)\cdot\tr(W^*W)=\frac{1}{4},
\end{align*}
while for $\tau_2$ we obtain
\begin{align*}
\psi_{\tau_2}(Z^*u_l^*,u_lZ,W^*u_r^*,u_rW)
&=\psi(Z^*u_l^*u_lZW^*u_r^*u_rW)\\
&=\tr(Z^*ZW^*W)=\frac{1}{2}.
\end{align*}
For the case of $\tau_3$, it follows that
\[
\psi_{\tau_3}(Z^*u_l^*,u_lZ,W^*u_r^*,u_rW)
=\psi(Z^*u_l^*W^*u_r^*)\cdot \psi(u_lZu_rW).
\]
To show that the product of moments  above is equal to zero we will use the moment condition supplied by \cite[Theorem 7]{C2019}. Since the pairs $(u_l,u_r)$ and $(Z,W)$ are $*$-bi-free, they must satisfy the  vanishing alternating centered $\chi$-interval Eigenshaft (abbreviated as ``vaccine''), which means that whenever $n\in\N, \chi\in\{l,r\}^n,\epsilon\in\{1,2\}^n, b_1,\ldots, b_n\in B$ are such that:
\begin{enumerate}[(a)]
\item $b_i\in\{u_l, u_l^*, Z, Z^*\}$ if $\chi(i)=l$ and $b_i\in\{u_r,u_r^*, W, W^*\}$ if $\chi(i)=r$,
\item $b_i\in \{u_l, u_l^*, u_r, u_r^*\}$ if $\epsilon(i)=1$ and $b_i\in\{Z, Z^*, W, W^*\}$ if $\epsilon(i)=2$,
    \item $\psi(b_{i_1}\cdots b_{i_k})=0$ whenever $\{i_1<\ldots <i_k\}$ is a maximal $\epsilon$-monochromatic $\chi$-interval,
\end{enumerate}
it follows that $\psi(b_1\cdots b_n)=0$. For $\chi\in\{l,r\}^4$ with $\chi^{-1}(\{l\})=\{1,2\}$,  $\epsilon\in\{1,2\}^4$ with $\epsilon^{-1}(1)=\{1,3\}$ and $b_1=u_l, b_2=Z, b_3=u_r, b_4=W$, note that the maximal $\epsilon$-monochromatic $\chi$-intervals are $\{1\}, \{3\}, \{2,4\}$ as shown in the following diagram:
\begin{equation*}
\begin{tikzpicture}[baseline]
			\draw[thick,dashed] (-.5,2.4) -- (-.5,-.25) -- (1.5,-.25) -- (1.5,2.4);
\node[left] at (-.5,2.3) {$\raisebox{.5pt}{\textcircled{\raisebox{-.9pt} {1}}}\quad u_l$};
			\draw[black,fill=black] (-.5,2.3) circle (0.05);
\node[left] at (-.5,1.6) {$\raisebox{.5pt}{\textcircled{\raisebox{-.9pt} {2}}}\quad Z$};
			\draw[black,fill=black] (-.5,1.6) circle (0.05);
\node[right] at (1.5,0.95) {$u_r\quad\raisebox{.5pt}{\textcircled{\raisebox{-.9pt} {1}}}$};
			\draw[black,fill=black] (1.5,0.95) circle (0.05);
\node[right] at (1.5,0.2) {$W\quad\raisebox{.5pt}{\textcircled{\raisebox{-.9pt} {2}}}$};
			\draw[black,fill=black] (1.5,0.2) circle (0.05);
			\draw[thick, red] (-1.6,1.61) -- (-2,1.61) -- (-2, -0.5) -- (3.05,-0.5)--(3.05,0.25)--(2.73,0.25);
   \draw[thick, blue] (-1.65,2.31) -- (-2.05,2.31);
   \draw[thick, blue] (2.7,0.97) -- (3.1,0.97);
	\end{tikzpicture}
 \end{equation*}
Since we have  
\[
\psi(b_1)=\varphi(u_l)=0, \psi(b_3)=\varphi(u_r)=0\text{ and }\psi(b_2\cdot b_4)=\tr (ZW)=0,
\]
it follows that $\psi(b_1 b_2 b_3 b_4)=\psi(u_l Z u_r W)=0$ and therefore
\[
\psi_{\tau_3}(Z^* u_l^*, u_l Z, W^* u_r^*, u_r W)=0.
\]
Since
\[
\mu_{\BNC}(\tau_1,1_{\chi})=\mu_{\BNC}(\tau_3,1_{\chi})=-1 \text{ and }\mu_{\BNC}(\tau_2,1_{\chi})=1,
\]
the bi-free cumulant is evaluated as follows
\[
\pushQED{\qed}
\kappa_{\chi}(Z^*{u_l}^*,{u_l}Z,W^*{u_r}^*,{u_r}W)=\frac{1}{2}-\frac{1}{4}-0=\frac{1}{4}\neq 0.\qedhere
\popQED
\]
\end{example}
However, when the pairs $(X,Y)$ and $(Z,W)$ are \emph{both} bi-R-diagonal and $*$-bi-free, then it is the case that the resulting pair $(XZ,YW)$ is also bi-R-diagonal, as the following proposition shows.

\begin{proposition}\label{prodbir-bir}
Let $(A,\varphi)$ be a non-commutative $*$-probability space and  $X,Y,Z,W\in A$ such that:
\begin{enumerate}[(a)]
\item the pairs $(X,Y)$ and $(Z,W)$ are both bi-R-diagonal,
\item the pairs $(X,Y)$ and $(Z,W)$ are $*$-bi-free.
\end{enumerate}
Then, the pair $(XZ,YW)$ is also bi-R-diagonal.
\end{proposition}

\begin{proof}
Let $n\in\N$, $\chi\in\{l,r\}^n$ and $a_1,\ldots,a_n\in A$ such that 
\[
a_i\in \begin{cases}
       \{XZ,Z^*X^*\}, &\text{if }  \chi (i)=l\\
        \{YW,W^*Y^*\}, &\text{if }  \chi (i) = r\\   
     \end{cases}\quad (i=1,\ldots,n).
\]
Also let $\xh\in\{l,r\}^{2n},\oxh\in\BNC(\xh)$ and $c_1,\ldots,c_{2n}\in A$ be as in  Notation \ref{notation: products}, so that by Theorem \ref{openup} we have
\[
\kappa_\chi(a_1,\ldots,a_n) =  \sum_{\substack{\tau\in \BNC(\xh)\\
\tau\vee\oxh = 1_{\xh}\\}}
        \kappa_{\xh,\tau}(c_1,\ldots,c_{2n})=  \sum_{\substack{\tau\in \BNC(\xh)\\
\tau\vee\oxh = 1_{\xh}\\}}
\prod_{V\in\tau}       \kappa_{{\xh}|_V}\left({(c_1,\ldots,c_{2n})}|_{V}\right).
\]
If the partition $\tau\in\BNC(\xh)$ yields a non-zero contribution to the sum of cumulants above, then as in the proof of Theorem \ref{prodbir-bifree},  we may assume  that for every $V\in\tau$, either $\{c_i:i\in V\}\subseteq \{X,X^*,Y,Y^*\}$
or  $\{c_i:i\in V\}\subseteq \{Z,Z^*,W,W^*\}$ and that if $n$ is odd, then $\kappa_\chi (a_1\ldots,a_n)=0$. In addition, since both pairs $(X,Y)$ and $(Z,W)$ are bi-R-diagonal, by Proposition \ref{prop: reduction properties} we may assume that every block $V$  of  $\tau$ contains   an even number of elements and the sequence $(c_1,\ldots,c_{2n})|_V$ is $*$-alternating  when read with the indices in the $\xh|_V$-order. 
Hence,  Proposition \ref{evenblocks} shows that $\sxh (1) {\sim}_{\tau} \sxh (2n)$ and $\sxh (2i) {\sim}_{\tau}\sxh(2i+1)$ for every $i=1,\ldots,n-1$.

We will now show that 
if the sequence $(a_{\sx(1)},\ldots,a_{\sx(n)})$
is not $*$-alternating, then  the cumulant ${\kappa}_{\chi}(a_1,\ldots,a_n)$ must vanish. Suppose the following situation occurs:
\[
\left(a_{\sx(1)},\ldots,a_{\sx(n)}\right)=(\ldots,YW,YW,\ldots),
\]
with $a_{\sx(m)}=a_{\sx(m+1)}=YW$ for some $m\in\{1,\ldots,n-1\}$. This implies the following situation for the $\xh$-order:
\[
\left(c_{\sxh(1)},\ldots,c_{\sxh(2n)}\right) = (\ldots,W,Y,W,Y,\ldots),
\]
with $c_{\sxh(2m-1)}=c_{\sxh(2m+1)}=W$ and $c_{\sxh(2m)}=c_{\sxh(2m+2)}=Y$. If $\tau\in\BNC(\xh)$ contributes to the sum of cumulants above, then the block $V$ of $\tau$ containing $\sxh(2m)$ must also contain $\sxh(2m+1)$, 
so that $V$ contains indices corresponding to $c_{\sxh(2m)}=Y$ and $c_{\sxh}(2m+1)=W$, which contradicts our assumptions on the partition $\tau$ above, due to the $*$-bi-free independence condition.
Hence, when
\[
\left(a_{\sx(1)},\ldots,a_{\sx(n)}\right) = (\ldots,YW,YW,\ldots),
\]  
we obtain that the bi-free cumulant ${\kappa}_{\chi}(a_1,\ldots,a_n)$ vanishes and  similar arguments show that this is also the case when the sequence $(a_{\sx(1)},\ldots,a_{\sx(n)})$ has either one of the following forms:
\[
(\ldots,XZ,XZ,\ldots), (\ldots,Z^*X^*,Z^*X^*,\ldots) \text{ or } (\ldots,W^*Y^*,W^*Y^*,\ldots).
\]
Next, suppose the following situation occurs:
\[
\left(a_{\sx(1)},\ldots,a_{\sx(n)}\right)=(\ldots\ldots,XZ,YW,\ldots\ldots),
\] 
with $a_{\sx(m)}=XZ$ and $a_{\sx(m+1)}=YW$ for some $m\in\{1,\ldots,n-1\}$. This implies the following situation for the $\xh$-order:
\[
\left(c_{\sxh(1)}\ldots,c_{\sxh(2n)}\right) = (\ldots\ldots,X,Z,W,Y,\ldots\ldots),
\]
with $c_{\sxh(2m-1)}=X$, $c_{\sxh(2m)}=Z$,  $c_{\sxh(2m+1)}=W$ and $c_{\sxh(2m+2)}=Y$. If $\tau$ is a contributing partition, then the block $V\in\tau$ containing $\sxh(2m)$ must also contain $\sxh(2m+1)$, which implies that the sequence $(c_1,\ldots,c_{2n})|_V$ when read with the indices in the $\xh|_V$-order will be of the form
\[
(\ldots\ldots,Z,W,\ldots\ldots),
\]
hence this sequence is clearly  not $*$-alternating, which again contradicts our assumptions on the partition $\tau$. As a result, when
\[
\left(a_{\sx(1)},\ldots,a_{\sx(n)}\right) = (\ldots,XZ,YW,\ldots),
\]  
we obtain that the bi-free cumulant ${\kappa}_{\chi}(a_1,\ldots,a_n)$ vanishes and similar arguments show that this is also the case when the sequence $(a_{\sx(1)},\ldots,a_{\sx(n)})$ is of the form
$(\ldots,Z^*X^*,W^*Y^*,\ldots)$, which completes the proof.
\end{proof}

We will now proceed to investigate whether the condition of bi-R-diagonality is preserved under the
taking of arbitrary powers. We begin with the following result, which generalizes \cite[Theorem 2.2]{larsenpowers} to the bi-free setting (see also \cite[Proposition 15.22]{nicaspeicher}).

\begin{theorem}\label{birpowers}
Let $(A,\varphi)$ be a non-commutative $*$-probability space and let $(X,Y)$ be a bi-R-diagonal pair in $A$. Then, for every $p\geq 1$ the pair $(X^p,Y^p)$ is also bi-R-diagonal.
\end{theorem}

\begin{proof}
Let $n\in\N, p\geq 1$, $\chi\in\{l,r\}^n$ and $a_1,\ldots,a_n\in A$ such that
\[
a_i\in \begin{cases}
       \{X^p,{(X^*)}^p\}, &\text{if } \chi (i)=l\\
        \{Y^p,{(Y^*)}^p\}, &\text{if }  \chi (i) = r\\   
     \end{cases}\quad  (i=1,\ldots,n).
\]
Also let $\xh\in\{l,r\}^{np},\oxh\in\BNC(\xh)$ and $c_1,\ldots,c_{2n}\in A$ be as in  Notation \ref{notation: products}, so that by Theorem \ref{openup} we have
\[
\kappa_\chi(a_1,\ldots,a_n) =  \sum_{\substack{\tau\in \BNC(\xh)\\
\tau\vee\oxh = 1_{\xh}\\}}
        \kappa_{\xh,\tau}(c_1,\ldots,c_{np})=  \sum_{\substack{\tau\in \BNC(\xh)\\
\tau\vee\oxh = 1_{\xh}\\}}
\prod_{V\in\tau}       \kappa_{{\xh}|_V}\left({(c_1,\ldots,c_{np})}|_{V}\right).
\]
We remark that since the pair $(X,Y)$ is bi-R-diagonal, by Proposition \ref{prop: reduction properties} if the partition $\tau\in\BNC(\xh)$ has a  non-zero contribution to the sum of cumulants above, every block $V$ of $\tau$ has even cardinality,  contains indices corresponding to an equal number of $*$-terms and non-$*$-terms and the sequence $(c_1,\ldots,c_{np})|_V$ is $*$-alternating when read with the indices in the induced $\xh|_V$-order. We will first show that if the sequence $(a_{\sx(1)},\ldots,a_{\sx(n)})$ is not $*$-alternating, then ${\kappa}_{\chi}(a_1,\ldots,a_n)=0$. Suppose the following situation occurs:
\[
\left(a_{\sx(1)},\ldots,a_{\sx(n)}\right) = (\ldots,X^p,Y^p,\ldots),
\]
where $a_{\sx(m)}=X^p$ and $a_{\sx(m+1)}=Y^p$ for some $m\in\{1,\ldots,n-1\}$. This implies the following situation for the $\xh$-order:
\[
\left(c_{\sxh(1)},\ldots,c_{\sxh(np)}\right) = \left(\ldots,\underbrace{X,X\ldots,X}_{p-\text{ terms}}, \underbrace{Y,Y,\ldots,Y}_{p-\text{terms}},\ldots\right),
\]
where $c_{\sxh((m-1)p+k)}=X$ and $c_{\sxh(mp+k)}=Y$, for all $k=1,\ldots,p$. Let $\tau\in\BNC(\xh)$ be a contributing partition and $V\in\tau$ such that $\sxh(mp+1)\in V$ (i.e. $V$ contains the index corresponding to the first right operator $Y$ which occurs in the sequence $(c_{\sxh(1)},\ldots,c_{\sxh(np)})$).  Consider the case when $\sxh(mp+1) = {\min}_{{\prec}_{\xh}}V$ and let $q\in\{1,\ldots,np\}$ be such that $\sxh(q)={\max}_{\prec_{\xh}}V$. Since the sequence $(c_1,\ldots,c_{np})|_V$ when read in the $\xh|_V$-order must be $*$-alternating, we see that
 $c_{\sxh(q)}$ must be a $*$-term and, since it appears after a right operator in the $\xh$-order, it must itself be a right operator, hence $c_{\sxh(q)}=Y^*$.
  We now claim  that we must necessarily have that $q=tp$ for some $m+2\leq t\leq n$, i.e. $c_{\sxh(q)}$ must be the last entry in a $p$-tuple of $(Y^*)$'s.

To see this, suppose that $q=tp+k$ with $t\in\{m+2,\ldots,n-1\}$ and $k\in\{1,\ldots,p-1\}$. This situation is depicted in the following diagram with the sequence $(c_1,\ldots,c_{np})$ read in the $\xh$-order.
\begin{equation*}
\begin{tikzpicture}[baseline]
			\draw[thick,dashed] (2,0) -- (12.9,0);
   \filldraw (2.1,-0.3) circle (0.02) node[anchor=north] {};
     \filldraw (2.5,-0.3) circle (0.02) node[anchor=north] {};
       \filldraw (2.9,-0.3) circle (0.02) node[anchor=north] {};
        \filldraw (3.3,-0.3) circle (0.02) node[anchor=north] {};
         \filldraw (3.7,0) circle (0.06) node[anchor=north] {$X$};
          \filldraw (4.1,-0.3) circle (0.02) node[anchor=north] {};
           \filldraw (4.5,-0.3) circle (0.02) node[anchor=north] {};
            \filldraw (4.9,-0.3) circle (0.02) node[anchor=north] {};
             \filldraw (5.3,0) circle (0.06) node[anchor=north] {$X$};
              \filldraw (5.9,0) circle (0.06) node[anchor=north] {$Y$};
              \filldraw (6.1,-0.3) circle (0.02) node[anchor=north] {};
              \filldraw (6.5,-0.3) circle (0.02) node[anchor=north] {};
              \filldraw (6.9,-0.3) circle (0.02) node[anchor=north] {};
              \filldraw (7.3,0) circle (0.06) node[anchor=north] {$Y$};
              \filldraw (7.7,-0.3) circle (0.02) node[anchor=north] {};
              \filldraw (8.1,-0.3) circle (0.02) node[anchor=north] {};
              \filldraw (8.5,-0.3) circle (0.02) node[anchor=north] {};
              \filldraw (8.9,-0.3) circle (0.02) node[anchor=north] {};
              \filldraw (9.3,0) circle (0.06) node[anchor=north] {$Y^*$};
              \filldraw (9.6,-0.3) circle (0.02) node[anchor=north] {};
              \filldraw (9.8,-0.3) circle (0.02) node[anchor=north] {};
              \filldraw (10.1,0) circle (0.06) node[anchor=north] {$Y^*$};
              \filldraw (10.4,-0.3) circle (0.02) node[anchor=north] {};
              \filldraw (10.6,-0.3) circle (0.02) node[anchor=north] {};
              \filldraw (10.9,0) circle (0.06) node[anchor=north] {$Y^*$};
              \filldraw (11.3,-0.3) circle (0.02) node[anchor=north] {};
              \filldraw (11.7,-0.3) circle (0.02) node[anchor=north] {};
              \filldraw (12.1,-0.3) circle (0.02) node[anchor=north] {};
              \filldraw (12.5,-0.3) circle (0.02) node[anchor=north] {};
\draw [thick, black,decorate,decoration={brace,amplitude=10pt,mirror},xshift=0.4pt,yshift=-0.4pt](3.6,-0.4) -- (5.3,-0.4) node[black,midway,yshift=-0.6cm] {\footnotesize $p$-terms};
\draw [thick, black,decorate,decoration={brace,amplitude=10pt,mirror},xshift=0.4pt,yshift=-0.4pt](5.9,-0.4) -- (7.3,-0.4) node[black,midway,yshift=-0.6cm] {\footnotesize $p$-terms};
\draw [thick, black,decorate,decoration={brace,amplitude=10pt,mirror},xshift=0.4pt,yshift=-0.4pt](9.2,-0.4) -- (10.8,-0.4) node[black,midway,yshift=-0.6cm] {\footnotesize $p$-terms};
\draw [thick, black,decorate,decoration={brace,amplitude=10pt},xshift=0.4pt,yshift=-0.4pt](6.0,0.2) -- (10,0.2) node[black,above=32pt, midway,yshift=-0.6cm] {\footnotesize unequal  $*$ and non-$*$-terms};
 \draw[-latex,line width=1pt] (5.9,-2)--(5.9,-1.2);
  \node[draw] at (5.9,-2.3) {${\min}_{{\prec}_{\xh}}V$};
   \draw[-latex,line width=1pt] (10.1,-2)--(10.1,-1.2);
  \node[draw] at (10.1,-2.3) {${\max}_{{\prec}_{\xh}}V$};
\draw[thick, black] (5.9,-2.7) -- (5.9,-3.1);
\draw[thick, black] (10.1,-2.7) -- (10.1,-3.1);
 \draw[-latex,line width=1pt] (6.8,-2.9)--(5.9,-2.9);
  \draw[-latex,line width=1pt] (9.2,-2.9)--(10.1,-2.9);
  \draw[thick,dashed, line width=1pt] (6.8,-2.9) -- (7.7,-2.9);
  \draw[thick,dashed, line width=1pt] (8.3,-2.9) -- (9.2,-2.9);
 \node[draw] at (8,-2.9) {$V$};             
\draw[thick, black] (5.9,0) -- (5.9,1.6) -- (10.1,1.6) -- (10.1,0);	
	
			\end{tikzpicture}
\end{equation*}
Observe that the $\xh$-interval $\{\sxh(j):mp+1\leq j\leq tp+k\}$ between the minimum and maximum indices of $V$ (in the $\xh$-order) needs to be written as a union of blocks of $\tau$ (since $\tau$ is bi-non-crossing) and clearly contains indices corresponding to an unequal number of $*$-terms and non-$*$-terms. Thus there exists a block $V'\in\tau$ with the same property, which contradicts our previous assumptions on the contributing partition $\tau$ and this shows that $q=tp$ for some $m+2\leq t\leq n$. But then the aforementioned $\xh$-interval becomes $\{\sxh(j):mp+1\leq j\leq tp\}$ and so is written both as a union of blocks $\oxh$, which is impossible by Lemma \ref{lemma: no block coupling}. Therefore, we may disregard the case when $\sxh(mp+1) = {\min}_{{\prec}_{\xh}}V$. 

Hence   there  exists $q<mp+1$ such that the index $\sxh(q)$ is in $V$ and we additionally 
assume that   $\sxh(q)$ and $\sxh(mp+1)$  are consecutive indices of $V$ in the $\xh$-order.  Note that since the sequence $(c_1,\ldots,c_{np})|_V$ when read in the $\xh|_V$-order is of the form $(\ldots,c_{\sxh(q)},Y,\ldots)$ and must be $*$-alternating, it must be the case that $c_{\sxh(q)}=X^*$ (observe that every operator that appears before $Y$ in the $\xh|_V$-order must be a left operator, since $c_{\sxh (mp+1)}=Y$ was the first  right operator that appeared in the $\xh$-order). As before, we find that $q=tp$ for some $1\leq t\leq m-1$ (i.e. $c_{\sxh(q)}$ must be the last entry in a $p$-tuple of $(X^*)$'s), since otherwise the $\xh$-interval $\{\sxh(j):q+1\leq j\leq mp\}$ between the consecutive indices $\sxh(q)$ and $\sxh(mp+1)$ contains unequal $*$-terms and non$*$-terms, as shown in the next diagram.
\begin{equation*}
\begin{tikzpicture}[baseline]
			\draw[thick,dashed] (2,0) -- (12.9,0);
   \filldraw (2.1,-0.3) circle (0.02) node[anchor=north] {};
     \filldraw (2.5,-0.3) circle (0.02) node[anchor=north] {};
       \filldraw (2.9,-0.3) circle (0.02) node[anchor=north] {};
        \filldraw (3.3,-0.3) circle (0.02) node[anchor=north] {};
         \filldraw (3.7,0) circle (0.06) node[anchor=north] {$X^*$};
          \filldraw (4.1,-0.3) circle (0.02) node[anchor=north] {};
           \filldraw (4.3,-0.3) circle (0.02) node[anchor=north] {};
            \filldraw (4.6,0) circle (0.06) node[anchor=north] {$X^*$};
            \filldraw (4.9,-0.3) circle (0.02) node[anchor=north] {};
             \filldraw (5.1,-0.3) circle (0.02) node[anchor=north] {};
             \filldraw (5.4,0) circle (0.06) node[anchor=north] {$X^*$};
              \filldraw (5.8,-0.3) circle (0.02) node[anchor=north] {};
               \filldraw (6.2,-0.3) circle (0.02) node[anchor=north] {};
                \filldraw (6.6,-0.3) circle (0.02) node[anchor=north] {};
                 \filldraw (7,-0.3) circle (0.02) node[anchor=north] {};
                 \filldraw (7.4,0) circle (0.06) node[anchor=north] {$X$};
                  \filldraw (7.8,-0.3) circle (0.02) node[anchor=north] {};
                   \filldraw (8.2,-0.3) circle (0.02) node[anchor=north] {};
                    \filldraw (8.6,-0.3) circle (0.02) node[anchor=north] {};
                    \filldraw (9,0) circle (0.06) node[anchor=north] {$X$};
                      \filldraw (9.6,0) circle (0.06) node[anchor=north] {$Y$};
                       \filldraw (10,-0.3) circle (0.02) node[anchor=north] {};
                        \filldraw (10.4,-0.3) circle (0.02) node[anchor=north] {};
                         \filldraw (10.8,-0.3) circle (0.02) node[anchor=north] {};
                          \filldraw (11.2,0) circle (0.06) node[anchor=north] {$Y$};
                          \filldraw (11.6,-0.3) circle (0.02) node[anchor=north] {};
                          \filldraw (12,-0.3) circle (0.02) node[anchor=north] {};
                          \filldraw (12.4,-0.3) circle (0.02) node[anchor=north] {};
                          \filldraw (12.8,-0.3) circle (0.02) node[anchor=north] {};
\draw [thick, black,decorate,decoration={brace,amplitude=10pt,mirror},xshift=0.4pt,yshift=-0.4pt](3.6,-0.4) -- (5.3,-0.4) node[black,midway,yshift=-0.6cm] {\footnotesize $p$-terms};
\draw [thick, black,decorate,decoration={brace,amplitude=10pt,mirror},xshift=0.4pt,yshift=-0.4pt](7.4,-0.4) -- (9,-0.4) node[black,midway,yshift=-0.6cm] {\footnotesize $p$-terms};
\draw [thick, black,decorate,decoration={brace,amplitude=10pt,mirror},xshift=0.4pt,yshift=-0.4pt](9.6,-0.4) -- (11.2,-0.4) node[black,midway,yshift=-0.6cm] {\footnotesize $p$-terms};
\draw [thick, black,decorate,decoration={brace,amplitude=10pt},xshift=0.4pt,yshift=-0.4pt](4.9,0.2) -- (8.9,0.2) node[black,above=32pt, midway,yshift=-0.6cm] {\footnotesize unequal  $*$ and non-$*$-terms};
          
\draw[thick, black] (4.6,0) -- (4.6,1.6) -- (9.6,1.6) -- (9.6,0);	
\draw[thick, black] (9.6,1.6) -- (10.2,1.6);
\draw[thick, black] (4,1.6) -- (4.6,1.6);
\draw[thick, black, dashed] (3,1.6) -- (4.1,1.6);
\draw[thick, black, dashed] (10.2,1.6) -- (11.3,1.6);	
			\end{tikzpicture}
\end{equation*}
However, if $q$ is a multiple of $p$, then the $\xh$-interval between $\sxh(q)$ and $\sxh(mp)$ is written as a union of blocks  of $\oxh$, which contradicts Lemma \ref{lemma: no block coupling}, as shown below.
\begin{equation*}
\begin{tikzpicture}[baseline]
			\draw[thick,dashed] (2,0) -- (12.9,0);
   \filldraw (2.1,-0.3) circle (0.02) node[anchor=north] {};
     \filldraw (2.5,-0.3) circle (0.02) node[anchor=north] {};
       \filldraw (2.9,-0.3) circle (0.02) node[anchor=north] {};
        \filldraw (3.3,-0.3) circle (0.02) node[anchor=north] {};
         \filldraw (3.7,0) circle (0.06) node[anchor=north] {$X^*$};
          \filldraw (4.2,-0.3) circle (0.02) node[anchor=north] {};
           \filldraw (4.6,-0.3) circle (0.02) node[anchor=north] {};
            \filldraw (5,-0.3) circle (0.02) node[anchor=north] {};
             \filldraw (5.4,0) circle (0.06) node[anchor=north] {$X^*$};
              \filldraw (5.8,-0.3) circle (0.02) node[anchor=north] {};
               \filldraw (6.2,-0.3) circle (0.02) node[anchor=north] {};
                \filldraw (6.6,-0.3) circle (0.02) node[anchor=north] {};
                 \filldraw (7,-0.3) circle (0.02) node[anchor=north] {};
                 \filldraw (7.4,0) circle (0.06) node[anchor=north] {$X$};
                  \filldraw (7.8,-0.3) circle (0.02) node[anchor=north] {};
                   \filldraw (8.2,-0.3) circle (0.02) node[anchor=north] {};
                    \filldraw (8.6,-0.3) circle (0.02) node[anchor=north] {};
                    \filldraw (9,0) circle (0.06) node[anchor=north] {$X$};
                      \filldraw (9.6,0) circle (0.06) node[anchor=north] {$Y$};
                       \filldraw (10,-0.3) circle (0.02) node[anchor=north] {};
                        \filldraw (10.4,-0.3) circle (0.02) node[anchor=north] {};
                         \filldraw (10.8,-0.3) circle (0.02) node[anchor=north] {};
                          \filldraw (11.2,0) circle (0.06) node[anchor=north] {$Y$};
                          \filldraw (11.6,-0.3) circle (0.02) node[anchor=north] {};
                          \filldraw (12,-0.3) circle (0.02) node[anchor=north] {};
                          \filldraw (12.4,-0.3) circle (0.02) node[anchor=north] {};
                          \filldraw (12.8,-0.3) circle (0.02) node[anchor=north] {};
\draw [thick, black,decorate,decoration={brace,amplitude=10pt,mirror},xshift=0.4pt,yshift=-0.4pt](3.6,-0.4) -- (5.3,-0.4) node[black,midway,yshift=-0.6cm] {\footnotesize $p$-terms};
\draw [thick, black,decorate,decoration={brace,amplitude=10pt,mirror},xshift=0.4pt,yshift=-0.4pt](7.4,-0.4) -- (9,-0.4) node[black,midway,yshift=-0.6cm] {\footnotesize $p$-terms};
\draw [thick, black,decorate,decoration={brace,amplitude=10pt,mirror},xshift=0.4pt,yshift=-0.4pt](9.6,-0.4) -- (11.2,-0.4) node[black,midway,yshift=-0.6cm] {\footnotesize $p$-terms};
\draw [thick, black,decorate,decoration={brace,amplitude=10pt},xshift=0.4pt,yshift=-0.4pt](5.8,0.2) -- (8.9,0.2) node[black,above=28pt, midway,yshift=-0.6cm,align=left] {\footnotesize union of blocks of $\tau$ and \\ \footnotesize union of blocks of $\oxh$};
          
\draw[thick, black] (5.4,0) -- (5.4,1.6) -- (9.6,1.6) -- (9.6,0);	
\draw[thick, black] (9.6,1.6) -- (10.2,1.6);
\draw[thick, black] (4.8,1.6) -- (5.4,1.6);
\draw[thick, black, dashed] (3.7,1.6) -- (4.8,1.6);
\draw[thick, black, dashed] (10.2,1.6) -- (11.3,1.6);	
			\end{tikzpicture}
			\end{equation*}
This shows that when
\[
\left(a_{\sx(1)},\ldots,a_{\sx(n)}\right) = (\ldots,X^p,Y^p,\ldots),
\]  
then for any $\tau\in\BNC(\xh)$ we either have that $\kappa_{\xh,\tau}(c_1,\ldots,c_{np})=0$ or the the partition $\tau$ fails to satisfy the relation $\tau\vee\oxh=1_{\xh}$.
As a result, the bi-free cumulant ${\kappa}_{\chi}(a_1,\ldots,a_n)$ vanishes and  similar arguments show that this is also the case when the sequence $\left(a_{\sx(1)},\ldots,a_{\sx(n)}\right)$ is in any other non-$*$-alternating form.
 It remains to show that if the cumulant ${\kappa}_{\chi}(a_1,\ldots,a_n)$ is of odd order, then it must vanish.

Assume that $n$ is an odd number. By the first part of the proof, we may suppose that the sequence  $\left(a_{\sx(1)},\ldots,a_{\sx(n)}\right)$ does not contain consecutive elements that both correspond to either $*$-terms or non-$*$-terms. Suppose the following situation occurs:
\[
\left(a_{\sx(1)},\ldots,a_{\sx(n)}\right) = \left({(X^*)}^p,\ldots\ldots,{(Y^*)}^p\right),
\]
where $a_{\sx(1)}={(X^*)}^p$ and $a_{\sx(n)}={(Y^*)}^p$. This implies the following situation for the $\xh$-order:
\[
\left(c_{\sxh(1)},\ldots,c_{\sxh(np)}\right) = \left(\underbrace{X^*,X^*,\ldots,X^*}_{p-\text{terms}},\ldots\ldots,\underbrace{Y^*,Y^*,\ldots,Y^*}_{p-\text{terms}}\right),
\]
where $c_{\sxh(k)}=X^*$ and $c_{\sxh((n-1)p+k)}=Y^*$, for all $k=1,\ldots,p$. Let $\tau\in\BNC(\xh)$ be a contributing partition and $V\in\tau$ such that $\sxh(1)\in V$. Also, let $q\in\{1,\ldots,np\}$ be such that $\sxh(q)={\max}_{\prec_{\xh}}V$. As argued previously, we  necessarily have $q=tp$ for some $t\in\{2,\ldots,n\}$, while Lemma \ref{lemma: no block coupling} shows that this is only possible when $q=np$.  But then the sequence $(c_1,\ldots,c_{np})|_V$ when read in the $\xh|_V$-order would be of the form $(X^*,\ldots,Y^*)$ and thus is either of odd length or not $*$-alternating, hence 
 the bi-free cumulant ${\kappa}_{\chi}(a_1,\ldots,a_n)$ necessarily vanishes. All remaining cases are handled similarly and
this completes the proof.
\end{proof}

We will now discuss why the conclusion of the previous theorem is no longer true if the exponents of the left and right operators comprising a bi-R-diagonal pair are distinct.  Consider a bi-Haar unitary pair $(u_l,u_r)$ in a non-commutative $*$-probability space $(A,\varphi)$. A computation of joint $*$-moments shows that for any positive integer $k\in\N$, the pair $(u_l^k,u_r^k)$ is also a bi-Haar unitary. However, if $k,m\in\N$ are distinct, then $k-m\neq 0$ but
\[
\varphi\left(\left(u_l^k\right)^{-m}\cdot \left(u_r^m\right)^k\right)=1\neq 0,
\]
thus the pair $(u_l^k,u_r^m)$ is not a bi-Haar unitary. A priori this does not exclude the possibility of $(u_l^k,u_r^m)$ being bi-R-diagonal.  To illustrate the situation, we will look at some examples of bi-free cumulants of the pair $(u_l^k,u_r^m)$ for various values of distinct $k,m\in\N$. First, consider the case $k=1,m=2$. To see that the pair $\left(u_l,u_r^2\right)$ is not bi-R-diagonal, for $\chi=\{l,l,r\}$ the bi-free cumulant $\kappa_\chi(u_l^*,u_l^*,u_r^2)$ is non-zero, even though it is both of odd order and its sequence of entries is not $*$-alternating  when read in the $\chi$-order. Indeed, by an application of Theorem \ref{openup} it is straightforward to verify that for $\xh=\{l,l,r,r\}$ and  $\tau=\left\{\{1,3\},\{2,4\}\right\}\in\BNC(\xh)$ we have
\[
\kappa_\chi(u_l^*,u_l^*,u_r^2)=\kappa_{\xh,\tau}(u_l^*,u_l^*,u_r,u_r)=\left[\kappa_{(l,r)}(u_l^*,u_r)\right]^2=\varphi(u_l^*\cdot u_r)^2=1\neq 0.
\]
The example above, where we considered a bi-free cumulant with two left terms in order to match the square of the single right term, suggests that we may always use a specified number of $(u_l^*)^k$-terms in order to match a given number of $(u_r^m)$-terms, which will result in a non-vanishing, non-$*$-alternating bi-free cumulant. As a result, to handle the general case of the pair $(u_l^k,u_r^m)$ with arbitrary distinct exponents $k,m\in\N$, it is natural to try and use $m$-many left terms and $k$-many right terms or,  even more generally, use $c$-many left terms and $d$-many right terms where $ck=dm$. However, it is not too hard to see that this does not work in general. For a concrete example, consider the pair $(u_l^2,u_r^4)$ and  the bi-free cumulant $\kappa_\chi \left((u_l^*)^2,(u_l^*)^2,(u_l^*)^2,(u_l^*)^2,u_r^4,u_r^4\right)$ where $\chi=\{l,l,l,l,r,r\}$ (clearly this corresponds to using $m=4$ left terms and $k=2$ right terms). For $\xh\in\{l,r\}^{16}$ with $\xh^{-1}(\{l\})=\{1,2,3,4,5,6,7,8\}$ and $\oxh\in\BNC(\xh)$ given by
\[
\oxh=\left\{\{1,2\},\{3,4\},\{5,6\},\{7,8\},\{9,10,11,12\},\{13,14,15,16\}\right\},
\]
an application of Theorem \ref{openup} yields
\[
\kappa_\chi\left((u_l^*)^2,(u_l^*)^2, (u_l^*)^2, (u_l^*)^2,u_r^4, u_r^4\right)=\sum_{\substack{\tau\in\BNC(\xh)\\\tau\vee\oxh=1_{\xh}}}\kappa_{\xh,\tau}\left(c_1,c_2,\ldots,c_{16}\right),
\]
where $c_1=\ldots=c_8=u_l^*$ and $c_9=\ldots=c_{16}=u_r$. Note that a bi-non-crossing partition $\tau\in\BNC(\xh)$ with  non-zero contribution to the sum of cumulants above must satisfy that all of its blocks contain exactly two indices, one corresponding to a $u_l^*$-term and the other to a $u_r$-term.  Indeed, if there exists a block $V\in\tau$ that contains at least three elements, then the sequence $\left(c_1,\ldots,c_{16}\right)|_V$ will be of the form
\[
\left(\ldots,u_l^*,u_l^*,\ldots\right)\text{ or }\left(\ldots,u_r,u_r,\ldots\right),
\]
and  since the pair $(u_l,u_r)$ is bi-R-diagonal we will have that
$\kappa_{\xh|_V}\left(\left(c_1,\ldots,c_{16}\right)|_V\right)=0$. Next, observe that the only bi-non-crossing partition whose every block contains exactly two indices, one corresponding to a $u_l^*$-term and one corresponding to a $u_r$-term, is the partition
\[
\rho=\left\{\{1,9\}, \{2,10\},\{3,11\},\{4,12\},\{5,13\},\{6,14\},\{7,15\},\{8,16\}\}\right\}\in\BNC(\xh).
\]
However, this partition does not satisfy that $\rho\vee\oxh=1_{\xh}$, since if
\[
\lambda=\left\{\{1,2,3,4,9,10,11,12\},\{5,6,7,8,13,14,15,16\}\right\},
\]
then it is easy to see that $\lambda\in\BNC(\xh)$ and $\rho,\oxh\leq\lambda$, hence in this case 
\[
\kappa_\chi\left((u_l^*)^2,(u_l^*)^2, (u_l^*)^2, (u_l^*)^2,u_r^4, u_r^4\right)=0.
\]
The next diagram illustrates the partitions $\oxh$ and $\rho$ (with the indices in the $\xh$-order). The two colored sections  below represent the blocks of $\lambda$, with each being a union of blocks of $\oxh$ and $\rho$.
\begin{equation*}
\begin{tikzpicture}[baseline]
			\draw[thick,dashed] (-1,0) -- (14,0);
\filldraw (-0.55,0) circle (0.06) node[anchor=north] {\begin{footnotesize}$1$\end{footnotesize}};
\filldraw (0.387,0) circle (0.06) node[anchor=north] {\begin{footnotesize}$2$\end{footnotesize}};
\filldraw (1.324,0) circle (0.06) node[anchor=north] {\begin{footnotesize}$3$\end{footnotesize}};
\filldraw (2.261,0) circle (0.06) node[anchor=north] {\begin{footnotesize}$4$\end{footnotesize}};
\filldraw (3.198,0) circle (0.06) node[anchor=north] {\begin{footnotesize}$5$\end{footnotesize}};
\filldraw (4.135,0) circle (0.06) node[anchor=north] {\begin{footnotesize}$6$\end{footnotesize}};
\filldraw (5.072,0) circle (0.06) node[anchor=north] {\begin{footnotesize}$7$\end{footnotesize}};
\filldraw (6.009,0) circle (0.06) node[anchor=north] {\begin{footnotesize}$8$\end{footnotesize}};
\filldraw (6.946,0) circle (0.06) node[anchor=north] {\begin{footnotesize}$16$\end{footnotesize}};
\filldraw (7.883,0) circle (0.06) node[anchor=north] {\begin{footnotesize}$15$\end{footnotesize}};
\filldraw (8.82,0) circle (0.06) node[anchor=north] {\begin{footnotesize}$14$\end{footnotesize}};
\filldraw (9.757,0) circle (0.06) node[anchor=north] {\begin{footnotesize}$13$\end{footnotesize}};
\filldraw (10.694,0) circle (0.06) node[anchor=north] {\begin{footnotesize}$12$\end{footnotesize}};
\filldraw (11.631,0) circle (0.06) node[anchor=north] {\begin{footnotesize}$11$\end{footnotesize}};
\filldraw (12.518,0) circle (0.06) node[anchor=north] {\begin{footnotesize}$10$\end{footnotesize}};
\filldraw (13.505,0) circle (0.06) node[anchor=north] {\begin{footnotesize}$9$\end{footnotesize}};

\draw[thick, black] (-0.55,0) -- (-0.55,1.5) -- (0.387,1.5) -- (0.387,0);
\draw[thick, black] (1.324,0) -- (1.324,1.5) -- (2.261,1.5) -- (2.261,0);
\draw[thick, black] (3.198,0) -- (3.198,1.5) -- (4.135,1.5) -- (4.135,0);
\draw[thick, black] (5.072,0) -- (5.072,1.5) -- (6.009,1.5) -- (6.009,0);
\draw[thick, black] (6.946,0) -- (6.946,1.5) -- (9.757,1.5) -- (9.757,0);
\draw[thick, black] (7.883,0) -- (7.883,1.5);
\draw[thick, black] (8.82,0) -- (8.82,1.5);
\draw[thick, black] (10.694,0) -- (10.694,1.5) -- (13.505,1.5) -- (13.505,0);
\draw[thick, black] (11.631,0) -- (11.631,1.5);
\draw[thick, black] (12.518,0) -- (12.518,1.5);

\draw[thick, black] (-0.55,-0.4) -- (-0.55,-2.7) -- (13.505,-2.7) -- (13.505,-0.4);
\draw[thick, black] (0.387,-0.4) -- (0.387,-2.4) -- (12.518,-2.4) -- (12.518,-0.4);
\draw[thick, black] (1.324,-0.4) -- (1.324,-2.1) -- (11.631,-2.1) -- (11.631,-0.4);
\draw[thick, black] (2.261,-0.4) -- (2.261,-1.7) -- (10.694,-1.7) -- (10.694,-0.4);
\draw[thick, black] (3.198,-0.4) -- (3.198,-1.4) -- (9.757,-1.4) -- (9.757,-0.4);
\draw[thick, black] (4.135,-0.4) -- (4.135,-1.1) -- (8.82,-1.1) -- (8.82,-0.4);
\draw[thick, black] (5.072,-0.4) -- (5.072,-0.8) -- (7.883,-0.8) -- (7.883,-0.4);
\draw[thick, black] (6.009,-0.4) -- (6.009,-0.5) -- (6.946,-0.5) -- (6.946,-0.4);

 \draw[red] (0.9,-0.05) ellipse (1.9cm and 0.45cm);
 \draw[red] (12.1,-0.05) ellipse (1.9cm and 0.45cm);
  \draw[blue] (6.5,-0.1) ellipse (3.6cm and 0.55cm);

 \draw[-latex,line width=1pt] (14.3,-1.35)--(13.8,-1.35);
  \node[draw] at (14.7,-1.35) {$\rho$};

   \draw[-latex,line width=1pt] (14.3,0.65)--(13.8,0.65);
  \node[draw] at (14.7,0.65) {$\oxh$};	\end{tikzpicture}
				\end{equation*}
The largest element of the second block of $\oxh$ is connected via $\rho$ to the smallest element of the last block of $\oxh$, so that unions of certain blocks of $\oxh$ equal unions of blocks of $\rho$, which gives rise to the partition $\lambda$ that is greater than both $\oxh$ and $\rho$, yet smaller than $1_{\xh}$. Thus a bi-free cumulant with $m=4$ $(u_l^*)^2$-terms and $k=2$ $(u_r^4)$-terms evaluates to zero, and the same argument shows that this is also the case when we have $c$-many left terms and $d$-many right terms with $c\geq m, d\geq k$ such that $ck=dm$. However,  the corresponding partition $\rho$ will satisfy that $\rho\vee\oxh=1_{\xh}$, provided  that the pair $(c,d)$ is minimal with respect to the relation $ck=dm$. This is the content of  the following lemma. 
\begin{lemma}\label{lemma:bi-R-diagpowerscounterexample}
Let $k,m\in\N$  and define
\[
c=\frac{m}{\gcd (k,m)}\text{ and }d=\frac{k}{\gcd (k,m)},
\]
 where $\gcd(k,m)$ denotes the greatest common divisor of $k$ and $m$. Note that $ck=dm$  and for 
$n=ck+dm$ 
define $\xh\in\{l,r\}^n$ by $\xh(i)=l$ for all $1\leq i\leq ck$ and $\xh(ck+i)=r$ for all $1\leq i\leq dm$. Also let $\oxh,\rho\in\mathcal{P}(n)$ 
be given by 
\[
\oxh=\left\{\{ik+1,\ldots,(i+1)k\}:0\leq i\leq c-1\right\}\cup\left\{\{ck+im+1,\ldots,ck+(i+1)m\}:0\leq i\leq d-1\right\},
\]
and
\begin{align*}
    \rho&=\left\{\{1,ck+1\},\{2,ck+2\},\ldots,\{ck,ck+dm\}\right\}=\left\{\{i,ck+i\}:1\leq i\leq ck\right\}.
\end{align*}
Then $\oxh,\rho\in\BNC(\xh)$ and $\rho\vee\oxh=1_{\xh}$.
\end{lemma}
\begin{proof}
Note that 
\[
\sxh(i)=i\text{ and }\sxh(ck+i)=n-i+1\text{ for all }1\leq i\leq ck,
\]
which implies that $\sxh^{-1}=\sxh$. If $\pi\in \NC(n)$ is given by
\[
\pi=\left\{\{1,ck+dm\},\{2,ck+dm-1\},\ldots,\{ck,ck+1\}\right\}=\left\{\{i,n-i+1\}:1\leq i\leq ck\right\},
\]
then $\rho=\sxh\cdot\pi$ and thus $\rho\in \BNC(\xh)$. Also, since $\sxh\cdot\oxh=\oxh$ and $\oxh\in\NC(n)$, it follows that $\oxh\in\BNC(\xh)$.  To show that $\rho\vee\oxh=1_{\xh}$,  it is enough to prove 
\[
\left(\sxh^{-1}\cdot 1_{\xh}\right)=\left(\sxh^{-1}\cdot\rho\right)\vee\left(\sxh^{-1}\cdot\oxh\right)\iff 1_n=\pi\vee\oxh.
\]
The next diagram illustrates the  partitions $\pi$ and $\oxh$.
\begin{equation*}
\begin{tikzpicture}[baseline]
			\draw[thick,dashed] (-1,0) -- (14,0);
\filldraw (-0.95,0) circle (0.06) node[anchor=north] {\begin{footnotesize}$1$\end{footnotesize}};
\filldraw (-0.65,0) circle (0.06) node[anchor=north] {\begin{footnotesize}$2$\end{footnotesize}};
\filldraw (-0.45,-0.25) circle (0.02) node[anchor=north] {};
\filldraw (-0.25,-0.25) circle (0.02) node[anchor=north] {};
\filldraw (-0.05,-0.25) circle (0.02) node[anchor=north] {};
\filldraw (0.15,0) circle (0.06) node[anchor=north] {\begin{footnotesize}$k$\end{footnotesize}};
\filldraw (0.45,-0.25) circle (0.02) node[anchor=north] {};
\filldraw (0.65,-0.25) circle (0.02) node[anchor=north] {};
\filldraw (0.85,-0.25) circle (0.02) node[anchor=north] {};
\filldraw (1.55,0) circle (0.06) node[anchor=north] {$\substack{(c-1)k+1}$};
\filldraw (2.25,-0.25) circle (0.02) node[anchor=north] {};
\filldraw (2.45,-0.25) circle (0.02) node[anchor=north] {};
\filldraw (2.65,-0.25) circle (0.02) node[anchor=north] {};
\filldraw (2.95,0) circle (0.06) node[anchor=north] {\begin{footnotesize}$ck$\end{footnotesize}};
\filldraw (3.75,0) circle (0.06) node[anchor=north] {\begin{footnotesize}$ck+1$\end{footnotesize}};
\filldraw (4.35,-0.25) circle (0.02) node[anchor=north] {};
\filldraw (4.55,-0.25) circle (0.02) node[anchor=north] {};
\filldraw (4.75,-0.25) circle (0.02) node[anchor=north] {};
\filldraw (5.45,0) circle (0.06) node[anchor=north] {\begin{footnotesize}$ck+m$\end{footnotesize}};
\filldraw (6.15,-0.25) circle (0.02) node[anchor=north] {};
\filldraw (6.35,-0.25) circle (0.02) node[anchor=north] {};
\filldraw (6.55,-0.25) circle (0.02) node[anchor=north] {};
\filldraw (7.60,0) circle (0.06) node[anchor=north] {$\substack{ck+(d-1)m+1}$};
\filldraw (8.65,-0.25) circle (0.02) node[anchor=north] {};
\filldraw (8.85,-0.25) circle (0.02) node[anchor=north] {};
\filldraw (9.05,-0.25) circle (0.02) node[anchor=north] {};
\filldraw (10.05,0) circle (0.06) node[anchor=north] {$\substack{ck+dm-k+1}$};
\filldraw (11.05,-0.25) circle (0.02) node[anchor=north] {};
\filldraw (11.25,-0.25) circle (0.02) node[anchor=north] {};
\filldraw (11.45,-0.25) circle (0.02) node[anchor=north] {};
\filldraw (12.25,0) circle (0.06) node[anchor=north] {$\substack{ck+dm-1}$};
\filldraw (13.55,0) circle (0.06) node[anchor=north] {$\substack{ck+dm}$};

\draw[thick, black] (-0.95,0) -- (-0.95,1.3) -- (0.15,1.3) -- (0.15,0);
\draw[thick, black] (-0.65,0) -- (-0.65,1.3);
\filldraw (-0.45,0.65) circle (0.02) node[anchor=north] {};
\filldraw (-0.25,0.65) circle (0.02) node[anchor=north] {};
\filldraw (-0.05,0.65) circle 
(0.02) node[anchor=north] {};

\draw[thick, black] (1.55,0) -- (1.55,1.3) -- (2.95,1.3) -- (2.95,0);

\draw[thick, black] (3.75,0) -- (3.75,1.3) -- (5.45,1.3) -- (5.45,0);

\filldraw (1.85,0.65) circle (0.02) node[anchor=north] {};
\filldraw (2.25,0.65) circle (0.02) node[anchor=north] {};
\filldraw (2.75,0.65) circle (0.02) node[anchor=north] {};

\filldraw (4.15,0.65) circle (0.02) node[anchor=north] {};
\filldraw (4.65,0.65) circle (0.02) node[anchor=north] {};
\filldraw (5.15,0.65) circle (0.02) node[anchor=north] {};

\draw[thick, black] (7.60,0) -- (7.60,1.3) -- (13.55,1.3) -- (13.55,0);
\draw[thick, black] (10.05,0) -- (10.05,1.3);
\draw[thick, black] (12.25,0) -- (12.25,1.3);

\filldraw (8,0.65) circle (0.02) node[anchor=north] {};
\filldraw (8.60,0.65) circle (0.02) node[anchor=north] {};
\filldraw (9.20,0.65) circle (0.02) node[anchor=north] {};
\filldraw (9.80,0.65) circle (0.02) node[anchor=north] {};

\filldraw (10.55,0.65) circle (0.02) node[anchor=north] {};
\filldraw (11.25,0.65) circle (0.02) node[anchor=north] {};
\filldraw (11.85,0.65) circle (0.02) node[anchor=north] {};

\draw[thick, black] (-0.95,-0.4) -- (-0.95,-2.5) -- (13.55,-2.5) -- (13.55,-0.4);
\draw[thick, black] (-0.65,-0.4) -- (-0.65,-2.1) -- (12.25,-2.1) -- (12.25,-0.4);
\draw[thick, black] (0.15,-0.4) -- (0.15,-1.7) -- (10.05,-1.7) -- (10.05,-0.4);
\draw[thick, black] (2.95,-0.4) -- (2.95,-1.3) -- (3.75,-1.3) -- (3.75,-0.4);

\filldraw (10.55,-1.05) circle (0.02) node[anchor=north] {};
\filldraw (11.25,-1.05) circle (0.02) node[anchor=north] {};
\filldraw (11.85,-1.05) circle (0.02) node[anchor=north] {};

\filldraw (-0.45,-1.05) circle (0.02) node[anchor=north] {};
\filldraw (-0.25,-1.05) circle (0.02) node[anchor=north] {};
\filldraw (-0.05,-1.05) circle 
(0.02) node[anchor=north] {};

\filldraw (0.6,-1.05) circle 
(0.02) node[anchor=north] {};
\filldraw (1.2,-1.05) circle 
(0.02) node[anchor=north] {};
\filldraw (1.7,-1.05) circle 
(0.02) node[anchor=north] {};
\filldraw (2.2,-1.05) circle 
(0.02) node[anchor=north] {};

\filldraw (5.25,-1.05) circle 
(0.02) node[anchor=north] {};
\filldraw (6.25,-1.05) circle 
(0.02) node[anchor=north] {};
\filldraw (7.25,-1.05) circle 
(0.02) node[anchor=north] {};
\filldraw (8.25,-1.05) circle 
(0.02) node[anchor=north] {};

 \draw[-latex,line width=1pt] (14.3,-1.35)--(13.8,-1.35);
  \node[draw] at (14.7,-1.35) {$\pi$};

   \draw[-latex,line width=1pt] (14.3,0.65)--(13.8,0.65);
  \node[draw] at (14.7,0.65) {$\oxh$};

  \draw [thick, black,decorate,decoration={brace,amplitude=10pt},xshift=0.4pt,yshift=-0.4pt](-0.95,1.4) -- (0.15,1.4) node[black,midway,yshift=0.6cm] {\footnotesize $k$-terms};

   \draw [thick, black,decorate,decoration={brace,amplitude=10pt},xshift=0.4pt,yshift=-0.4pt](1.55,1.4) -- (2.95,1.4) node[black,midway,yshift=0.6cm] {\footnotesize $k$-terms};

    \draw [thick, black,decorate,decoration={brace,amplitude=10pt},xshift=0.4pt,yshift=-0.4pt](3.75,1.4) -- (5.45,1.4) node[black,midway,yshift=0.6cm] {\footnotesize $m$-terms};

 \draw [thick, black,decorate,decoration={brace,amplitude=10pt},xshift=0.4pt,yshift=-0.4pt](7.60,1.4) -- (13.55,1.4) node[black,midway,yshift=0.6cm] {\footnotesize $m$-terms};

			\end{tikzpicture}
				\end{equation*}
First of all, elementary divisibility arguments show that all pairs $(x,y)$ of positive integers which are solutions to the  equation $xk=ym$ are of the form
\[
(x,y)=\left(\frac{zm}{\gcd(k,m)},\frac{zk}{\gcd(k,m)}\right)=(zc,zd),
\]
for some positive integer $z\geq 1$. This implies that  if $q\in\N$  is such that $q<c$, then $dm-qk$ is not a multiple of $m$. 

To show that $\pi\vee\oxh=1_n$ it is enough to show that if $\sigma\in\NC(n)$ is such that $\pi,\oxh\leq \sigma$, then for all $1\leq i\leq n-1$ we have that $i\sim_\sigma i+1$, i.e. that the indices $i$ and $i+1$ are in the same block of $\sigma$. Moreover, by looking at the block structure of $\oxh$,  it is enough to show that $i\sim_\sigma i+1$ when either $i=qk$ for some $1\leq q\leq c-1$, or when $i=ck+qm$ for some $1\leq q\leq d-1$.  

Suppose that $i=qk$ for $1\leq q\leq c-1$ and note that $i$ must be connected via $\pi$ to an index greater than $ck$, which may be written as $ck+pm+z$, with $0\leq p\leq d-1$ and $1\leq z\leq m$. We observe that the case $z=1$ is impossible; for otherwise if $qk\sim_\pi ck+pm+1$, then by its definition the partition $\pi$ has the set $\{i,n-i+1\}=\{qk,ck+dm-qk+1\}$ as a block, which implies that $dm-qk=pm$, a contradiction. Hence we may assume that $2\leq z\leq m$. In addition, the set $\{qk+1,ck+pm+z-1\}$ is a also block of $\pi$, while on the other hand both the indices $ck+pm+z-1$ and $ck+pm+z$ belong to the block $\{ck+pm+1,\ldots,ck+(p+1)m\}$ of $\oxh$.  We thus  see that
\[
qk\sim_\pi (ck+pm+z)\sim_{\oxh}(ck+pm+z-1)\sim_\pi qk+1,
\]
therefore, since $\oxh,\pi\leq \sigma$, we obtain $i\sim_\sigma i+1$.
The case when $i=ck+qm$ for $1\leq q\leq d-1$ is handled similarly and this completes the proof. 
\end{proof}
The next proposition shows that when one raises the left and right operators comprising a bi-Haar unitary pair to possibly distinct exponents, then the resulting pair will not be bi-R-diagonal unless the exponents are actually equal.
\begin{proposition}\label{prop:bi-R-diagpowerscounterexample}
Let $(A,\varphi)$ be a non-commutative $*$-probability space and let $(u_l,u_r)$ be a bi-Haar unitary pair in $(A,\varphi)$. For $k,m\in\N$ the following are equivalent:
\begin{enumerate}[(i)]
    \item the pair $\left(u_l^k,u_r^m\right)$ is bi-R-diagonal,
    \item $k=m$.
\end{enumerate}
\end{proposition}
\begin{proof}
The implication $(ii)\implies (i)$ follows by Theorem \ref{birpowers} (or, alternatively, by a direct computation of joint $*$-moments which shows that the pair $(u_l^k,u_r^k)$ is a bi-Haar unitary, hence also bi-R-diagonal). For the converse, assume that $k\neq m$, let $c,d,n,\xh,\oxh,\rho$ be as in the statement of Lemma \ref{lemma:bi-R-diagpowerscounterexample} and define $\chi\in\{l,r\}^{c+d}$ by $\chi^{-1}(l)=\{1,\ldots,c\}$. Also let $a_1,\ldots,a_{c+d}\in A$ be given by
\[
a_1=\ldots=a_c=(u_l^*)^k\text{ and }a_{c+1}=\ldots=a_{c+d}=u_r^m.
\]
Observe that the sequence
\[
\left(a_{\sx(1)},\ldots,a_{\sx(c+d)}\right)=\left(\underbrace{(u_l^*)^k,\ldots,(u_l^*)^k}_{c-\text{terms}}, \underbrace{u_r^m,\ldots,u_r^m}_{d-terms}\right),
\]
is $*$-alternating if and only if $c=d=1$, which happens precisely when $k=m$ since
\[
c=\frac{m}{\gcd(k,m)}\text{ and }d=\frac{k}{\gcd(k,m)}.
\]
We will show that $\kappa_\chi(a_1,\ldots,a_{c+d})=1$ which, as we assumed that $k\neq m$, will imply that $\kappa_\chi(a_1,\ldots,a_{c+d})$ is a non-vanishing, non-$*$-alternating  bi-free cumulant and therefore the pair $(u_l^k,u_r^m)$ is not bi-R-diagonal. By Theorem \ref{openup} we obtain
\[
\kappa_\chi (a_1,\ldots,a_{c+d})= \sum_{\substack{\tau\in \BNC(\xh)\\
\tau\vee\oxh = 1_{\xh}\\}}
        \kappa_{\xh,\tau}(c_1,\ldots,c_n),
\]
where $c_1,\ldots,c_n\in A$ are given by
$c_i=u_l^*$ for $1\leq i\leq ck$ and $c_i=u_r$ for $ck+1\leq i\leq ck+dm$. The only bi-non-crossing partition which may contribute to the sum of cumulants above is the partition all of whose blocks contain exactly two indices, one corresponding to a $u_l^*$-term and one corresponding to a $u_r$-term. This is the partition 
\[
\rho=\left\{\{1,ck+1\},\{2,ck+2\},\ldots,\{ck,ck+dm\}\right\}=\left\{\{i,ck+i\}:1\leq i\leq ck\right\}\in\BNC(\xh).
\]
By Lemma \ref{lemma:bi-R-diagpowerscounterexample} we also have that $\rho\vee\oxh=1_{\xh}$, therefore
\[
\kappa_\chi(a_1,\ldots,a_{c+d})=\kappa_{\xh,\rho}(c_1,\ldots,c_n)=\prod_{V\in\rho}\kappa_{\xh|_V}\left((c_1,\ldots,c_n)|_V\right)=\left[\kappa_{(l,r)}(u_l^*,u_r)\right]^{ck}=\varphi(u_l^*\cdot u_r)^{ck}=1,
\]
as desired.
\end{proof}
The previous proposition also shows that if $X,Y$ are R-diagonal operators in a non-commutative $*$-probability space, then the pair $(X,Y)$ is not necessarily bi-R-diagonal (see also Example \ref{example:freeR-diagnotbi-R-diag} and Proposition \ref{prop:tensorproductofR-diag}).   
With our discussion on powers of bi-R-diagonal pairs complete, we now proceed to show that bi-R-diagonal pairs of operators  yield examples of bi-free pairs  that consist of self-adjoint operators.

\begin{proposition}\label{bifree-s.a.}
Let $(A,\varphi)$ be a non-commutative $*$-probability space and $(X,Y)$ be a bi-R-diagonal pair in $A$. Then, the pairs
\[
(XX^*, Y^* Y) \ \text{and} \  (X^* X, Y Y^*)
\]
are bi-free.
\end{proposition}

\begin{proof}
Let $n\in\N, n\geq 2$, $\chi\in\{l,r\}^n$ and $a_1,\ldots,a_n\in A$ such that
\[
a_i\in \begin{cases}
       \{XX^*,X^*X\}, &\text{if }  \chi (i)=l\\
        \{Y^*Y,YY^*\}, &\text{if }  \chi (i) = r\\   
     \end{cases}\quad (i=1,\ldots,n).
\]
Moreover, suppose that there exist $i\neq j\in\{1,\ldots,n\}$ such that $a_i\in\{XX^*,Y^*Y\}$ and $a_j\in\{X^*X,YY^*\}$. We will show that ${\kappa}_{\chi}(a_1,\ldots,a_n)=0$, which will imply that the pairs $(XX^*,Y^*Y)$ and $(X^*X,YY^*)$ are indeed bi-free.  Let $\xh\in\{l,r\}^{2n},\oxh\in\BNC(\xh)$ and $c_1,\ldots,c_{2n}\in A$ be as in  Notation \ref{notation: products}, so that by Theorem \ref{openup} we have
\[
\kappa_\chi(a_1,\ldots,a_n) =  \sum_{\substack{\tau\in \BNC(\xh)\\
\tau\vee\oxh = 1_{\xh}\\}}
        \kappa_{\xh,\tau}(c_1,\ldots,c_{2n})=  \sum_{\substack{\tau\in \BNC(\xh)\\
\tau\vee\oxh = 1_{\xh}\\}}
\prod_{V\in\tau}       \kappa_{{\xh}|_V}\left({(c_1,\ldots,c_{2n})}|_{V}\right).
\]
Since the pair $(X,Y)$ is bi-R-diagonal, by Proposition \ref{prop: reduction properties}  we have that every block of a partition $\tau\in\BNC(\xh)$ which contributes to the sum of cumulants above has an even number of elements. Our initial assumption implies that there exists $i\in\{1,\ldots,n\}$ such that either
\[
a_{\sx(i)}\in\{XX^*,Y^*Y\}\text{ and } a_{\sx(i+1)}\in\{X^*X,YY^*\},
\]
or 
\[
a_{\sx(i)}\in\{X^*X,YY^*\}\text{ and }a_{\sx(i+1)}\in\{XX^*,Y^*Y\}.
\]
Assume that $a_{\sx(i)}=XX^*$ and $a_{\sx(i+1)}=YY^*$ (with the remaining cases handled similarly). Then, the following situation occurs for the $\xh$-order:
\[
(c_{\sxh(1)},\ldots,c_{\sxh(2n)}) = (\ldots,X,X^*,Y^*,Y,\ldots),
\]
where $c_{\sxh(2i-1)}=X, c_{\sxh(2i)}=X^*, c_{\sxh(2i+1)}=Y^*$ and $c_{\sxh(2i+2)}=Y$. By Proposition \ref{evenblocks}, for $\tau\in\BNC(\xh)$ such that $\tau\vee\oxh = 1_{\xh}$ there exists $V\in\tau$ with $\{\sxh(2i),\sxh(2i+1)\}\subseteq V$. But then, the sequence ${(c_1,\ldots,c_{2n})}|_V$ when read in the induced $\xhv$-order would be of the form
\[
(\ldots\ldots,X^*,Y^*,\ldots\ldots),
\]
with this implying that ${\kappa}_{\xhv}({(c_1,\ldots,c_{2n})}|_V)=0$, since  bi-free cumulants with entries in $\{X,X^*,Y,Y^*\}$  that are non-$*$-alternating in the  corresponding $\chi$-order must vanish. Hence ${\kappa}_{\xh,\tau}(c_1,\ldots,c_{2n})=0$ for all $\tau\in\BNC(\xh)$ with $\tau\vee\oxh=1_{\chi}$, which shows that $\kappa_\chi (a_1,\ldots,a_n)=0$,   as desired.
\end{proof}

We remark that if $(X,Y)$ is a bi-R-diagonal pair in some non-commutative $*$-probability space, then it is not necessarily true that the pairs $(XX^*,YY^*)$ and $(X^*X,Y^*Y)$ are bi-free, as the following example indicates.

\begin{example}
Let $(u_l,u_r)$ be a bi-Haar unitary pair in a  $\mathrm{C}^*$-probability space $(A,\varphi)$ and consider the pair $(Z,W)$ in the space $(\mathcal{M}_2(\C),\tr)$ defined as follows:
\[
Z = \begin{bmatrix}
1 & 0\\ 0 &0
\end{bmatrix} \text{ and } W = \begin{bmatrix}
0 & 0\\ 1 &0
\end{bmatrix}.
\]
We may find a larger  $\mathrm{C}^*$-probability space $(B,\psi)$ such that the pairs $(u_l,u_r)$ and $(Z,W)$ are $*$-bi-free in $(B,\psi)$  and such that $\psi$ preserves the joint $*$-distributions of the pairs $(u_l,u_r)$ and $(Z,W)$ with respect to $\varphi$ and $\tr$ respectively (see Remark \ref{remark:propertiesofbi-freepairs}). By Theorem \ref{prodbir-bifree}, the pair $(u_l Z,Wu_r)$ is bi-R-diagonal in $(B,\psi)$. However, the pairs
\[
(Z^* u_l^* u_l Z, u_r^* W^* Wu_r) = (Z^*Z,u_r^* W^* Wu_r)\text{ and }(u_l ZZ^* u_l^*, Wu_r u_r^* W^*) = (u_l ZZ^* u_l^*, WW^*),
\]
are not bi-free, since the moment-cumulant formula yields
\[
\pushQED{\qed}
\kappa_{\chi}(Z^*Z,WW^*) = \tr(Z^*Z WW^*)-\tr(Z^*Z)\cdot\tr(WW^*) = -\frac{1}{4}\neq 0.\qedhere
\popQED
\]
\end{example}

We close this section by discussing free product and tensor product distibutions of pairs formed by R-diagonal operators. First of all, pairs formed by  freely independent R-diagonal operators need not be bi-R-diagonal, as the next example indicates.
\begin{example}\label{example:freeR-diagnotbi-R-diag}
Let $(A,\varphi_1)$ and $(B,\varphi_2)$ be non-commutative $*$-probability spaces and let $u\in (A,\varphi_1)$ and $v\in (B,\varphi_2)$ be Haar unitaries. Then the pair $(u,v)$ is not bi-R-diagonal in $(A*B,\varphi)$ where $\varphi=\varphi_1*\varphi_2$. Indeed, for $\chi\in\{l,r\}^4$ with $\chi^{-1}(\{l\})=\{1,3\}$, we claim that the bi-free cumulant $\kappa_\chi (u,v,u^*,v^*)$ does not vanish, even though the sequence $(u,v,u^*,v^*)$ when viewed in the induced $\chi$-order is given by $(u,u^*,v^*,v)$ and thus is not $*$-alternating. By the moment-cumulant formula we have
\[
\kappa_\chi(u,v,u^*,v^*)=\sum_{\tau\in \BNC(\chi)}\varphi_\tau (u,v,u^*,v^*)\mu_{\BNC}(\tau,1_\chi).
\]
Note that since  
both $u$ and $v$ are Haar unitaries,  all elements in the set $\{u,v,u^*,v^*\}$ are centered (i.e. $\varphi(u)=\varphi(u^*)=\varphi(v)=\varphi(v^*)=0$), thus in order for a bi-non-crossing partition to yield a non-zero contribution to the sum above all of its blocks must not be singletons.
These are the following three bi-non-crossing partitions:
\[
\tau_1=\{\{1,3\},\{2,4\}\}, \tau_2=\{\{1,2\},\{3,4\}\}\text{ and }\tau_3=\{\{1,2,3,4\}\}=1_\chi.
\]
Since the algebras $\alg (\{u,u^*\})$ and $\alg (\{v,v^*\})$ are freely independent, by the characterization of free independence in terms of moments we see that $\varphi (uv)=\varphi(u^*v^*)=\varphi(uvu^*v^*)=0$ and thus
\[
\varphi_{\tau_2}(u,v,u^*,v^*)=\varphi(uv)\varphi(u^*v^*)=0\text{ and }\varphi_{\tau_3}(u,v,u^*,v^*)=\varphi(uvu^*v^*)=0.
\]
Since $\mu(\tau_1,1_\chi)=-1$ we have
\[
\kappa_\chi(u,v,u^*,v^*)=-\varphi_{\tau_1}(u,v,u^*,v^*)=-\varphi(uu^*)\varphi(vv^*)=-1\neq 0,
\]
which shows that the pair $(u,v)$ is not bi-R-diagonal.\qed
\end{example}

An example as the one above is not surprising, since  free independence fits into the setting of bi-freeness  via considering either only left or only right actions of algebras on reduced free product spaces independently, and is by itself not sufficient to account for a mixture of  left and right actions. In addition, the joint $*$-distributions of bi-R-diagonal pairs are, in a specific sense, characterized by the bi-free independence condition (see Theorem \ref{invdistr}). On the other hand, since as stated in Remark \ref{remark:propertiesofbi-freepairs} the classical independence of algebras implies  bi-freeness of the induced left and right algebras (see also  \cite[Section 3.2]{skoufranisind} for a combinatorial description of the interplay between bi-freeness and classical independence), it is natural for one to expect that pairs formed by R-diagonal operators in tensor product non-commutative probability spaces (or by classically independent R-diagonal operators) will be bi-R-diagonal. This is indeed the case, as indicated in the next proposition.

\begin{proposition}\label{prop:tensorproductofR-diag}
Let $(B,\varphi_1)$ and $(C,\varphi_2)$  be non-commutative $*$-probability spaces and let $b\in B, c\in C$ be R-diagonal. Then the pair $(b\otimes 1_C,1_B\otimes c)$ is bi-R-diagonal in  $(B\otimes C,\varphi)$, where $\varphi=\varphi_1\otimes\varphi_2$.  
\end{proposition}
\begin{proof}
First of all, note that since the $*$-distribution of $b\in (B,\varphi_1)$ coincides with the $*$-distribution of $b\otimes 1_C\in (B\otimes C,\varphi)$, we have that $b\otimes 1_C$ is R-diagonal in $B\otimes C$ and, similarly, $1_B\otimes c$ is also R-diagonal. Since the algebras $B\otimes 1_C$ and $1_B\otimes C$ are  classically independent, by \cite[Proposition 2.16]{voiculescu} we have that the pairs of faces $(B\otimes 1_C,\mathbb{C}1_{B\otimes C})$ and $(\mathbb{C}1_{B\otimes C}, 1_B\otimes C)$  are bi-free in $B\otimes C$ with respect to $\varphi=\varphi_1\otimes\varphi_2$, which then implies that all mixed bi-free cumulants with entries in the set $\{b\otimes 1_C,b^*\otimes 1_C,1_B\otimes c,1_B\otimes c^*\}$ vanish. Specifically, for all $n\in\mathbb{N}$, for all non-constant maps  $\chi\in\{l,r\}^n$ and for all $Z_1,\ldots,Z_n\in B\otimes C$ such that
\[
Z_i\in\begin{cases}
    \{b\otimes 1_C,b^*\otimes 1_C\},\quad\text{ if }\chi(i)=l\\
    \{1_B\otimes c,1_B\otimes c^*\},\quad\text{ if }\chi(i)=r
\end{cases}
(i=1,\ldots,n),
\]
it follows that $\kappa_\chi(Z_1,\ldots,Z_n)=0$. As a result, in order for a bi-free cumulant with entries in the set $\{b\otimes 1_C,b^*\otimes 1_C,1_B\otimes c,1_B\otimes c^*\}$ to be non-vanishing, all of its entries must be either left variables (i.e. all of its entries must lie in the set $\{b\otimes 1_C,b^*\otimes 1_C\}$) or all of its entries must be right variables (i.e. all of its entries must lie in the set $\{1_B\otimes c,1_B\otimes c^*\}$). However, a bi-free cumulant with only left or only right entries reduces to a free cumulant which can be non-vanishing only if the sequence of either the left or the right entries is $*$-alternating, since both $b\otimes 1_C$ and $1_B\otimes c$ are R-diagonal. Specifically, if $n\in\N,\chi=l^n$ and $Z_1,\ldots,Z_n\in\{b\otimes 1_C, b^*\otimes 1_C\}$ then
\[
\kappa_\chi(Z_1,\ldots,Z_n)=\kappa_n(Z_1,\ldots,Z_n),
\]
where on the right-hand side we have a free cumulant.  But since $b\otimes 1_C$ is R-diagonal in $(B\otimes C,\varphi)$, in order for the free cumulant $\kappa_n(Z_1,\ldots,Z_n)$ to be non-zero, we must have that the sequence $(Z_1,\ldots,Z_n)$ is of even  length and  $*$-alternating. Since the corresponding vanishing property also holds for all bi-free cumulants with entries in the set $\{1_B\otimes c,1_B\otimes c^*\}$, we obtain that the pair $(b\otimes 1_C,1_B\otimes c)$ is bi-R-diagonal, as desired.
\end{proof}

\begin{remark}\label{remark: every power is bi-R-diag}
We note that the proposition above gives examples of bi-R-diagonal pairs of operators $(X,Y)$ such that $(X^n,Y^m)$ is also bi-R-diagonal for all $n,m\in\N$ (compare this  with Proposition \ref{prop:bi-R-diagpowerscounterexample}). Indeed, if $a\in (A,\varphi)$ and $b\in (B,\psi)$ are R-diagonal, then by \cite[Theorem 2.2]{larsenpowers}  the operators $a^n$ and $b^m$ are also R-diagonal for all $n,m\in\N$, hence by Proposition \ref{prop:tensorproductofR-diag} the pair 
\[
\left(\left(a\otimes 1_B\right)^n,\left(1_A\otimes b\right)^m\right)=\left(a^n\otimes 1_B, 1_A\otimes b^m\right),
\]
is bi-R-diagonal in $(A\otimes B,\varphi\otimes\psi)$. On the other hand, if $u\in (A,\varphi)$ and $v\in (B,\psi)$ are Haar unitaries, then  the pair $(u\otimes  1_B,1_A\otimes v)$ is bi-R-diagonal but not bi-Haar unitary in $(A\otimes B,\varphi\otimes\psi)$, since for $n\in\N$ we see
\[
(\varphi\otimes\psi)\left((u\otimes 1_B)^n\cdot (1_A\otimes v)^{-n}\right)=\varphi(u^n)\cdot\psi(v^{-n})=0\neq 1.
\]
\end{remark}
The previous proposition allows us to characterize the joint $*$-distributions of bi-R-diagonal pairs where the algebras of left and right operators are classically independent.

\begin{corollary}\label{corr:clasinpbi-R-diagcharacterization}
Let $(A,\varphi)$ be a non-commutative $*$-probability space and let $X,Y\in A$ be such that the algebras $\alg(\{X,X^*\})$ and $\alg(\{Y,Y^*\})$ commute. The following are equivalent:
\begin{enumerate}[(i)]
    \item the pair $(X,Y)$ is bi-R-diagonal and the algebras $\alg(\{X,X^*\})$ and $\alg(\{Y,Y^*\})$ are classically independent, so that $\varphi(Z W)=\varphi(Z)\cdot\varphi(W)$ for all $Z\in\alg(\{X,X^*\})$ and $W\in\alg (\{Y,Y^*\})$,
    \item there exist non-commutative $*$-probability spaces $(B,\psi_1)$ and $(C,\psi_2)$ and R-diagonal operators $b\in B$ and $c\in C$ such that the joint $*$-distribution of the pair $(X,Y)$ coincides with the joint $*$-distribution of the pair $(b\otimes 1_C, 1_B\otimes c)$, with the latter  considered in the tensor product space $(B\otimes C, \psi)$, where $\psi=\psi_1\otimes\psi_2$.
\end{enumerate}
\end{corollary}

\begin{proof}
For the forward implication, let $\psi_1$ and $\psi_2$ be the states given as the restrictions of $\varphi$ on the unital $*$-algebras $B=\alg (\{1_A,X,X^*\})$ and $C=\alg (\{1_A,Y,Y^*\})$ respectively, and note that $X$ and $Y$ are R-diagonal operators on the non-commutative $*$-probability spaces $(B,\psi_1)$ and $(C,\psi_2)$ respectively. If $\psi$ denotes the restriction of $\varphi$ to the unital $*$-algebra $A'$ generated by the union $B\cup C$, then as the algebras $B$ and $C$ are classically independent we must have that $\psi=\psi_1\otimes\psi_2$, and therefore the joint $*$-distributions of the pairs
\[
(X,Y)\in (A',\psi)\text{ and }(X\otimes 1,1\otimes Y)\in (B\otimes C,\psi_1\otimes\psi_2)
\]
must coincide.

For the converse implication,  first note that if the joint $*$-distributions of the pairs $(X,Y)$ and $(b\otimes 1_C, 1_B\otimes c)$ coincide, then  the algebras $\alg (\{X,X^*\})$ and $\alg (\{Y,Y^*\})$ are classically independent. Moreover, since by Proposition \ref{prop:tensorproductofR-diag} the pair $(b\otimes 1_C,1_B\otimes c)$ is bi-R-diagonal, again by the equality of joint $*$-distributions it follows that the pair $(X,Y)$ is also bi-R-diagonal.
\end{proof}
Recall that in the case of a free Haar unitary $u\in (A,\varphi)$, its $*$-distribution is represented by the Haar measure on the unit circle, as can be verified by direct computation of $*$-moments. Indeed,  if $\mu$ denotes the normalized Haar measure on the unit circle $\mathbb{T}$, then for $n,m\in\N$ we see that
\[
\int_\mathbb{T} z^n\overline{z}^md\mu=\int_0^1 e^{2\pi (n-m)it}dt=\begin{cases}
    1,\quad\text{ if }n=m\\
    0,\quad\text{ otherwise},
\end{cases} 
\]
which coincide with the $*$-moments of $u$, so that in the $\mathrm{C}^*$-probability space setting, $\mu$ is the spectral measure of $u$.  We remark that in the case of a bi-Haar unitary pair $(u_l,u_r)$ in $(A,\varphi)$, its joint $*$-distribution is not represented by the product  Haar measure. Even though this can be verified directly via computation of product measure mixed moments and comparing to the corresponding ones for the pair $(u_l,u_r)$, it is enough to note that, in view of Corollary \ref{corr:clasinpbi-R-diagcharacterization},  the joint $*$-distribution of $(u_l,u_r)$ is not a tensor product distribution. This is because even though the algebras $\alg (\{u_l, u_l^*\})$ and $\alg (\{u_r, u_r^*\})$ commute, they are not classically independent, since for all $n\in\N$  we see
\[
\varphi(u_l^n\cdot u_r^{-n})=1\neq 0=\varphi (u_l^n)\cdot\varphi(u_r^{-n}).
\]

\section{Joint $*$-Distributions of Bi-R-Diagonal Pairs}\label{sctn:distributions}
In this section, we will be concerned with proving that the joint $*$-distribution of a bi-R-diagonal pair of operators remains invariant under the multiplication with a $*$-bi-free, bi-Haar unitary pair. We begin by giving the definition of even and $*$-even sets of operators, as well as display how this notion can yield examples of bi-R-diagonal pairs.
\begin{definition}\label{bievendef} Let $(A,\varphi)$ be a non-commutative $*$-probability space, let $\emptyset\neq S\subseteq A$ and $Z,W\in A$.
\begin{enumerate}[(i)]
\item The set $S\subseteq A$ is called \emph{even} if for every $k\in\N$ and $a_1,\ldots,a_{2k+1}\in S$ we have
\[
\varphi(a_1\cdots a_{2k+1})=0,
\]
that is, all joint moments of elements of $S$ of odd order vanish,
\item the set $S\subseteq A$ is called \emph{$*$-even} if the set $S\cup S^*$ is even, so that all joint $*$-moments of elements of $S$ of odd order vanish,
\item the pair $(Z,W)$ is called \emph{even} if the set $\{Z,W\}$ is even and similarly $(Z,W)$ is called \emph{$*$-even} if the set $\{Z,W\}$ is $*$-even.
\end{enumerate}
\end{definition}
The moment-cumulant  formula yields  that the pair $(Z,W)$ is $*$-even if and only if all  bi-free cumulants of odd order with entries in $\{Z,Z^*,W,W^*\}$ vanish. It clearly follows that every bi-R-diagonal pair is $*$-even. In the setting of free probability, it is observed that products of free, self-adjoint, even elements (i.e. self-adjoint elements of non-commutative $*$-probability spaces all whose moments of odd order vanish) result in  R-diagonal elements (\cite[Theorem 15.17]{nicaspeicher}). Generalizing this to the bi-free setting, we will show that products of $*$-even pairs (where the order of the right operators is reversed in the product) yield bi-R-diagonal pairs.  For this, we have the following proposition, the proof of which will be similar to the proofs of Theorem \ref{prodbir-bifree} and Proposition \ref{prodbir-bir}.

\begin{proposition}\label{prod-bi-even}
Let $(A,\varphi)$ be a non-commutative probability space and $X,Y,Z,W\in A$ such that:
\begin{enumerate}[(a)]
\item the pairs $(X,Y)$ and $(Z,W)$ are both $*$-even,
\item the pairs $(X,Y)$ and $(Z,W)$ are $*$-bi-free.
\end{enumerate}
Then, the pair $(XZ,WY)$ is bi-R-diagonal.
\end{proposition}

\begin{proof}
Let $n\in\N$, $\chi\in\{l,r\}^n$ and $a_1,\ldots,a_n\in A$ be such that 
\[
a_i\in \begin{cases}
       \{XZ,Z^*X^*\}, &\text{if } \chi (i)=l\\
        \{WY,Y^*W^*\}, &\text{if } \chi (i) = r\\   
     \end{cases} \quad (i=1,\ldots,n).
\]
Also let $\xh\in\{l,r\}^{2n},\oxh\in\BNC(\xh)$ and $c_1,\ldots,c_{2n}\in A$ be as in  Notation \ref{notation: products}, so that by Theorem \ref{openup} we have
\[
\kappa_\chi(a_1,\ldots,a_n) =  \sum_{\substack{\tau\in \BNC(\xh)\\
\tau\vee\oxh = 1_{\xh}\\}}
        \kappa_{\xh,\tau}(c_1,\ldots,c_{2n})=  \sum_{\substack{\tau\in \BNC(\xh)\\
\tau\vee\oxh = 1_{\xh}\\}}
\prod_{V\in\tau}       \kappa_{{\xh}|_V}\left({(c_1,\ldots,c_{2n})}|_{V}\right).
\]
By Proposition \ref{prop: reduction properties}, we note that if a bi-non-crossing partition $\tau\in\BNC(\xh)$ gives a non-zero contribution to the sum of cumulants above, then   for all $V\in\tau$ either
$\{c_i :i\in V\}\subseteq\{X,X^*,Y,Y^*\}$ or 
$\{c_i :i\in V\}\subseteq\{Z,Z^*,W,W^*\}$, since the pairs $(X,Y)$ and $(Z,W)$ are $*$-bi-free. Also,  every block of $\tau$ must contain an even number of elements, since both pairs $(X,Y)$ and $(Z,W)$ are $*$-even. This also shows that $\kappa_\chi(a_1,\ldots,a_n)=0$ when $n$ is odd, since then every $\tau\in\BNC(\xh)$ contains a block with an odd number of indices corresponding to either elements in the set $\{X,X^*,Y,Y^*\}$ or $\{Z,Z^*,W,W^*\}$.  As a result, for a contributing  partition $\tau\in\BNC(\xh)$ such that $\tau\vee\oxh=1_{\xh}$,  Proposition \ref{evenblocks} implies that $\sxh (1) {\sim}_{\tau} \sxh (2n)$ and $\sxh (2i) {\sim}_{\tau}\sxh(2i+1)$ for every $i=1,\ldots,n-1$.

We will now show that if the sequence $(a_{\sx(1)},\ldots,a_{\sx(n)})$ is  not $*$-alternating, then the cumulant ${\kappa}_{\chi}(a_1,\ldots,a_n)$ must vanish. Suppose the following situation occurs:
\[
\left(a_{\sx(1)},\ldots,a_{\sx(n)}\right)=(\ldots\ldots,Z^*X^*,Y^*W^*,\ldots\ldots),
\]
with $a_{\sx(m)}=Z^*X^*$ and $a_{\sx(m+1)}=Y^*W^*$ for some $m\in\{1,\ldots,n-1\}$. This implies the following situation for the $\xh$-order:
\[
\left(c_{\sxh(1)},\ldots,c_{\sxh(2n)}\right) = (\ldots\ldots,Z^*,X^*,W^*,Y^*,\ldots\ldots),
\]
with
\[
c_{\sxh(2m-1)}=Z^*,c_{\sxh(2m)}=X^*, c_{\sxh(2m+1)}=W^* \text{ and }c_{\sxh(2m+2)}=Y^*.
\]
Now, if $\tau$ is a bi-non-crossing partition contributing to the sum of cumulants above, then the block of $\tau$ containing $\sxh(2m)$ must also contain $\sxh(2m+1)$. But, since
\[
c_{\sxh(2m)}=X^* \text{ and }c_{\sxh(2m+1)}=W^*,
\]
this is impossible, due to the $*$-bi-free independence condition. Hence, when
\[
\left(a_{\sx(1)},\ldots,a_{\sx(n)}\right) = (\ldots,Z^*X^*,Y^*W^*,\ldots),
\]  
we obtain   for all $\tau\in\BNC(\xh)$ that $\kappa_{\xh,\tau}(c_1,\ldots,c_{2n})=0$, thus the bi-free cumulant ${\kappa}_{\chi}(a_1,\ldots,a_n)$ vanishes. All other remaining cases are handled similarly.
\end{proof}
For a non-commutative $*$-probability space $(A,\varphi)$ and $X_1,X_2,Y_1,Y_2\in A$, consider the pair $(Z,W)$ in the tensor product space $(\mathcal{M}_2(A),\varphi\otimes\tr)$ defined by
\[
Z = \begin{bmatrix}
0 & X_1\\ X_2 &0
\end{bmatrix} \text{ and } W = \begin{bmatrix}
0 & Y_1\\ Y_2 &0
\end{bmatrix}.
\]
Since any product with entries in $\{Z,Z^*,W,W^*\}$ containing an odd number of elements results in a matrix with zeroes across the diagonal, it follows that $(Z,W)$ is a $*$-even pair. Such pair is not necessarily bi-R-diagonal, since for instance
\[
{\kappa}_{\chi}(Z,Z)=(\varphi\otimes\tr)(Z\cdot Z) = \frac{1}{2}(\varphi(X_1 X_2) + \varphi(X_2 X_1)),
\]
which need not equal zero. However, the previous proposition implies that the product of two such pairs that are $*$-bi-free will always be bi-R-diagonal. Actually, matrix pairs arising in this manner can be used to characterize the condition of bi-R-diagonality (see Theorem \ref{together}).

We now move on to discuss one of the main results of this section, which concerns the invariance of the joint $*$-distribution of a bi-R-diagonal pair $(X,Y)$ under the multiplication by a $*$-bi-free, bi-Haar unitary pair $(u_l,u_r)$ (see Theorem \ref{invdistr}). The  proof in the free probability setting uses the  moment characterization of freeness  (see \cite[Section 1.9]{approach} and \cite[Theorem 15.10]{nicaspeicher}), which in our current setting of bi-freeness amounts to a much higher degree of complexity. Hence, for the result in the bi-free context we will use a different approach based on bi-free cumulants, thus a new proof will follow for the free probability case as well.

What the proof of the aforementioned invariance result  amounts to is showing that any $*$-alternating, bi-free $*$-cumulant $\kappa_\chi(a_1,\ldots,a_n)$ of even order with entries in the set $\{u_lX, Yu_r\}$ is the same as the bi-free cumulant $\kappa_\chi (b_1,\ldots,b_n)$, where each $b_i$ is obtained by keeping only the terms in $c_i$ which correspond to the bi-R-diagonal pair $(X,Y)$ 
and ``forgetting'' about the multiplicative terms in $\{u_l,u_l^*,u_r,u_r^*\}$. Since the cumulant $\kappa_\chi (a_1,\ldots,a_n)$ has products of operators as arguments, we will make use of Theorem \ref{openup} to isolate the product terms and  express it as sum of cumulants (of higher order). The bi-freeness and the $*$-alternating cumulant assumptions on the pairs $(X,Y)$ and $(u_l,u_r)$, along with the technical requirements of Theorem \ref{openup}, imply that only partitions that satisfy specific properties are able to contribute to this latter sum of cumulants (see properties (A) and (B) in the example and lemma below). For that reason, it is crucial to be able to identify and characterize the contributing partitions in a manner which will make certain cumulant cancellations  apparent (in combination with Lemma \ref{techlemmaBNC}) and will establish the connection with the desired bi-free cumulant $\kappa_\chi (b_1,\ldots,b_n)$.  We begin with an example which illustrates the situation.

\begin{example}
Let $\chi\in\{l,r\}^8$ be given by $\chi^{-1}(\{l\})=\{2,4,6,8\}$.   We are interested in characterizing bi-non-crossing partitions $\tau\in\BNC(\chi)$ that satisfy certain properties. First of all, note that the set $\{1,2,\ldots,8\}$ with the indices in the $\chi$-order is given as
\begin{equation*}
\begin{tikzpicture}[baseline]
			\draw[thick,dashed] (3.1,0) -- (9.1,0);
\filldraw (3.3,0) circle (0.06) node[anchor=south] {$2$};
\filldraw (4.1,0) circle (0.06) node[anchor=south] {$4$};
\filldraw (4.9,0) circle (0.06) node[anchor=south] {$6$};
\filldraw (5.7,0) circle (0.06) node[anchor=south] {$8$};
\filldraw (6.5,0) circle (0.06) node[anchor=south] {$7$};
\filldraw (7.3,0) circle (0.06) node[anchor=south] {$5$};	
\filldraw (8.1,0) circle (0.06) node[anchor=south] {$3$};
\filldraw (8.9,0) circle (0.06) node[anchor=south] {$1$};
			\end{tikzpicture}
			\qquad
				\end{equation*}
Clearly, in order to build a bi-non-crossing partition $\tau\in\BNC(\chi)$, it is enough to build a non-crossing partition on the diagram above when the indices are in the $\chi$-order (i.e. build a partition $\pi\in\NC\left(\left\{2,4,6,8,7,5,3,1\right\}\right)$) and then move the indices in the original order $\{1,2,\ldots,8\}$. Define 
\[
F_1=\{\sx(1),\sx(8)\}=\{2,1\}, G_1=\{\sx(2),\sx(3)\}=\{4,6\},
\]
and
\[
F_2=\{\sx(4),\sx(5)\}=\{8,7\}, G_2=\{\sx(6),\sx(7)\}=\{5,3\}.
\]
We are interested in partitions $\tau\in\BNC(\chi)$ such that
\begin{center}
    (A)\quad $\sx(1)\sim_\tau \sx(8)$ and $\sx(2i)\sim_\tau \sx(2i+1)$, for $i=1,2,3$,
\end{center}
i.e. partitions such that each of $F_1,G_1,F_2,G_2$ is contained in a block of $\tau$. The diagram below displays the sets $F_1,G_1,F_2,G_2$ with the indices in the $\chi$-order and the dashed line segments connect indices that need to be in the same block of a candidate partition $\tau\in\BNC(\chi)$.
\begin{equation*}
\begin{tikzpicture}[baseline]
\draw[thick,dashed] (3.1,0) -- (9.1,0);
\filldraw (3.3,0) circle (0.06) node[anchor=south] {$2$};
\filldraw (4.1,0) circle (0.06) node[anchor=south] {$4$};
\filldraw (4.9,0) circle (0.06) node[anchor=south] {$6$};
\filldraw (5.7,0) circle (0.06) node[anchor=south] {$8$};
\filldraw (6.5,0) circle (0.06) node[anchor=south] {$7$};
\filldraw (7.3,0) circle (0.06) node[anchor=south] {$5$};	
\filldraw (8.1,0) circle (0.06) node[anchor=south] {$3$};
\filldraw (8.9,0) circle (0.06) node[anchor=south] {$1$};
\draw[thick, dashed] (3.3,0) -- (3.3,-1) -- (8.9,-1) -- (8.9,0);
\draw[thick, dashed] (4.1,0) -- (4.1,-0.6) -- (4.9,-0.6) -- (4.9,0);
\draw[thick, dashed] (5.7,0) -- (5.7,-0.6) -- (6.5,-0.6) -- (6.5,0);
\draw[thick, dashed] (7.3,0) -- (7.3,-0.6) -- (8.1,-0.6) -- (8.1,0);
 \draw[blue] (4.5,0.05) ellipse (0.7cm and 0.45cm);
  \draw[red] (6.1,0.05) ellipse (0.7cm and 0.45cm);
  \draw[blue] (7.7,0.05) ellipse (0.7cm and 0.45cm);
   \draw[red] (3.3,0.05) ellipse (0.3cm and 0.45cm);
    \draw[red] (8.9,0.05) ellipse (0.3cm and 0.45cm);
    \node[draw] at (4.5,0.85) {$G_1$};
     \node[draw] at (6.1,0.85) {$F_2$};
      \node[draw] at (7.7,0.85) {$G_2$};
       \node[draw] at (3.3,0.85) {$F_1$};
        \node[draw] at (8.9,0.85) {$F_1$};
			\end{tikzpicture}
			\qquad
				\end{equation*}
We also want our candidate partitions $\tau\in\BNC(\chi)$ to satisfy
\begin{center}
    (B)\quad if $V$ is a block of $\tau$, then either $V\subseteq F_1\cup F_2$ or $V\subseteq G_1\cup G_2$,
\end{center}
i.e. any block of $\tau$ has to contain indices only from $F_1$ and $F_2$ or only from $G_1$ and $G_2$. Together with property (A) this implies that every block of $\tau$ is either a union of only some $F_i's$ or a union of only some $G_i$'s.

As a result, to form candidate bi-non-crossing partitions, we may replace the sub-blocks $F_1, F_2$ by symbols $1$ and $2$, and similarly we may replace the sub-blocks $G_1, G_2$  by symbols $\overline{1}$ and $\overline{2}$ respectively. Note that a non-crossing partition $\rho$ on  $\left\{1,\overline{1},2,\overline{2}\right\}$ all of whose blocks are contained either in $\{1,2\}$ or in $\left\{\overline{1},\overline{2}\right\}$ gives rise to a non-crossing partition on $\{\sx(1),\ldots,\sx(8)\}=\{2,4,6,8,7,5,3,1\}$ whose blocks are given by taking the corresponding unions of the $F_i$'s and the $G_i$'s according to the block structure of $\rho$. For example, the partition $\rho=\left\{\left\{1,2\right\},\left\{\overline{1}\right\},\left\{\overline{2}\right\}\right\}\in\NC\left(\left\{1,\overline{1},2,\overline{2}\right\}\right)$ corresponds to
\begin{equation*}
\begin{tikzpicture}[baseline]
			\draw[thick,dashed] (6.3,0) -- (9.1,0);
\filldraw (6.5,0) circle (0.06) node[anchor=south] {$1$};
\filldraw (7.3,0) circle (0.06) node[anchor=south] {$\overline{1}$};
\filldraw (8.1,0) circle (0.06) node[anchor=south] {$2$};
\filldraw (8.9,0) circle (0.06) node[anchor=south] {$\overline{2}$};
\draw[thick, black] (6.5,0) -- (6.5,-1) -- (8.1,-1) -- (8.1,0);	
\draw[thick, black] (7.3,0) -- (7.3,-0.6);
\draw[thick, black] (8.9,0) -- (8.9,-0.6);
			\end{tikzpicture}
	\qquad\longrightarrow	\qquad
\begin{tikzpicture}[baseline]
			\draw[thick,dashed] (3.1,0) -- (9.1,0);
\filldraw (3.3,0) circle (0.06) node[anchor=south] {$2$};
\filldraw (4.1,0) circle (0.06) node[anchor=south] {$4$};
\filldraw (4.9,0) circle (0.06) node[anchor=south] {$6$};
\filldraw (5.7,0) circle (0.06) node[anchor=south] {$8$};
\filldraw (6.5,0) circle (0.06) node[anchor=south] {$7$};
\filldraw (7.3,0) circle (0.06) node[anchor=south] {$5$};	
\filldraw (8.1,0) circle (0.06) node[anchor=south] {$3$};
\filldraw (8.9,0) circle (0.06) node[anchor=south] {$1$};
\draw[thick, black] (3.3,0) -- (3.3,-1) -- (8.9,-1) -- (8.9,0);
\draw[thick, black] (4.1,0) -- (4.1,-0.6) -- (4.9,-0.6) -- (4.9,0);
\draw[thick, black] (7.3,0) -- (7.3,-0.6) -- (8.1,-0.6) -- (8.1,0);
\draw[thick, black] (5.7,0) -- (5.7,-1);
\draw[thick, black] (6.5,0) -- (6.5,-1);
			\end{tikzpicture}
			\qquad
				\end{equation*}
where the partition on the right-hand side is non-crossing and its blocks are given by $F_1\cup F_2$ (since the indices $1$ and $2$ were in the same block of $\rho$) and the (isolated) blocks $G_1$ and $G_2$ (since the indices $\overline{1}$ and $\overline{2}$ were not connected in $\rho$).  On the other hand, the crossing partition $\left\{\left\{1,2\right\}, \left\{\overline{1},\overline{2}\right\}\right\}$ translates to 
                \begin{equation*}
\begin{tikzpicture}[baseline]
			\draw[thick,dashed] (6.3,0) -- (9.1,0);
\filldraw (6.5,0) circle (0.06) node[anchor=south] {$1$};
\filldraw (7.3,0) circle (0.06) node[anchor=south] {$\overline{1}$};
\filldraw (8.1,0) circle (0.06) node[anchor=south] {$2$};
\filldraw (8.9,0) circle (0.06) node[anchor=south] {$\overline{2}$};
\draw[thick, black] (6.5,0) -- (6.5,-1) -- (8.1,-1) -- (8.1,0);	
\draw[thick, black] (7.3,0) -- (7.3,-0.6) -- (8.9,-0.6) -- (8.9,0);

			\end{tikzpicture}
	\qquad\longrightarrow	\qquad
\begin{tikzpicture}[baseline]
			\draw[thick,dashed] (3.1,0) -- (9.1,0);
\filldraw (3.3,0) circle (0.06) node[anchor=south] {$2$};
\filldraw (4.1,0) circle (0.06) node[anchor=south] {$4$};
\filldraw (4.9,0) circle (0.06) node[anchor=south] {$6$};
\filldraw (5.7,0) circle (0.06) node[anchor=south] {$8$};
\filldraw (6.5,0) circle (0.06) node[anchor=south] {$7$};
\filldraw (7.3,0) circle (0.06) node[anchor=south] {$5$};	
\filldraw (8.1,0) circle (0.06) node[anchor=south] {$3$};
\filldraw (8.9,0) circle (0.06) node[anchor=south] {$1$};
\draw[thick, black] (3.3,0) -- (3.3,-1) -- (8.9,-1) -- (8.9,0);
\draw[thick, black] (4.1,0) -- (4.1,-0.6) -- (8.1,-0.6) -- (8.1,0);
\draw[thick, black] (4.9,0) -- (4.9,-0.6);
\draw[thick, black] (7.3,0) -- (7.3,-0.6);
\draw[thick, black] (5.7,0) -- (5.7,-1);
\draw[thick, black] (6.5,0) -- (6.5,-1);
			\end{tikzpicture}
			\qquad
				\end{equation*}
with the partition on the right-hand side above being also crossing. Therefore, in order to identify all bi-non-crossing partitions $\tau\in\BNC(\chi)$ that satisfy properties (A) and (B) above, it is enough to characterize the non-crossing partitions on $\left\{1,\overline{1},2,\overline{2}\right\}$ all of whose blocks are either in $\{1,2\}$ or in $\left\{\overline{1},\overline{2}\right\}$. But these are simply given by pairs of partitions $(\pi,\sigma)$ where $\pi\in\NC(\{1,2\}), \sigma\in\NC\left(\left\{\overline{1},\overline{2}\right\}\right)$ such that $\pi\cup\sigma\in\NC\left(\left\{1,\overline{1},2,\overline{2}\right\}\right)$, with the last condition being equivalent to $\sigma\leq \Krew_{\NC}(\pi)$, where $\Krew_{\NC}$ denotes Kreweras' complementation map on the lattice of non-crossing partitions. Note that, conversely to the process shown above, any $\tau\in\BNC(\chi)$ which satisfies properties (A) and (B) uniquely determines two non-crossing partitions $\pi,\sigma\in \NC(2)$ such that $\sigma\leq \Krew_{\NC}(\pi)$ by connecting the indices in $\{1,2\}$ and $\left\{\overline{1},\overline{2}\right\}$ according to which of the $F_i$'s and $G_i$'s are connected in the blocks of $\tau$. For instance, the partition $\tau=\left\{\{1,2\}, \{3,4,5,6\}, \{7,8\}\right\}\in\BNC(\chi)$ determines two non-crossing partitions as follows:
                 \begin{equation*}
\begin{tikzpicture}[baseline]
\draw[thick,dashed] (3.1,0) -- (9.1,0);
\filldraw (3.3,0) circle (0.06) node[anchor=south] {$2$};
\filldraw (4.1,0) circle (0.06) node[anchor=south] {$4$};
\filldraw (4.9,0) circle (0.06) node[anchor=south] {$6$};
\filldraw (5.7,0) circle (0.06) node[anchor=south] {$8$};
\filldraw (6.5,0) circle (0.06) node[anchor=south] {$7$};
\filldraw (7.3,0) circle (0.06) node[anchor=south] {$5$};	
\filldraw (8.1,0) circle (0.06) node[anchor=south] {$3$};
\filldraw (8.9,0) circle (0.06) node[anchor=south] {$1$};
\draw[thick, black] (3.3,0) -- (3.3,-1) -- (8.9,-1) -- (8.9,0);
\draw[thick, black] (4.1,0) -- (4.1,-0.6) -- (8.1,-0.6) -- (8.1,0);
\draw[thick, black] (4.9,0) -- (4.9,-0.6);
\draw[thick, black] (7.3,0) -- (7.3,-0.6);
\draw[thick, black] (5.7,0) -- (5.7,-0.3) -- (6.5,-0.3) -- (6.5,0);
 \draw[blue] (4.5,0.05) ellipse (0.7cm and 0.45cm);
  \draw[red] (6.1,0.05) ellipse (0.7cm and 0.45cm);
  \draw[blue] (7.7,0.05) ellipse (0.7cm and 0.45cm);
   \draw[red] (3.3,0.05) ellipse (0.3cm and 0.45cm);
    \draw[red] (8.9,0.05) ellipse (0.3cm and 0.45cm);
    \node[draw] at (4.5,0.85) {$G_1$};
     \node[draw] at (6.1,0.85) {$F_2$};
      \node[draw] at (7.7,0.85) {$G_2$};
       \node[draw] at (3.3,0.85) {$F_1$};
        \node[draw] at (8.9,0.85) {$F_1$};
			\end{tikzpicture}
	\quad\longrightarrow	\quad
\begin{tikzpicture}[baseline]
			\draw[thick,dashed] (6.3,0) -- (9.1,0);
\filldraw (6.5,0) circle (0.06) node[anchor=south] {$1$};
\filldraw (7.3,0) circle (0.06) node[anchor=south] {$\overline{1}$};
\filldraw (8.1,0) circle (0.06) node[anchor=south] {$2$};
\filldraw (8.9,0) circle (0.06) node[anchor=south] {$\overline{2}$};
\draw[thick, black] (7.3,0) -- (7.3,-1) -- (8.9,-1) -- (8.9,0);	
\draw[thick, black] (6.5,0) -- (6.5,-0.6);
\draw[thick, black] (8.1,0) -- (8.1,-0.6);
			\end{tikzpicture}
			\quad\longrightarrow\quad \begin{tikzpicture}[baseline]
			\draw[thick,dashed] (5.1,0) -- (6.3,0);
            \draw[thick,dashed] (7.3,0) -- (8.5,0);
\filldraw (5.3,0) circle (0.06) node[anchor=south] {$1$};
\filldraw (6.1,0) circle (0.06) node[anchor=south] {$2$};
\filldraw (7.5,0) circle (0.06) node[anchor=south] {$1$};
\filldraw (8.3,0) circle (0.06) node[anchor=south] {$2$};
\draw[thick, black] (7.5,0) -- (7.5,-1) -- (8.3,-1) -- (8.3,0);	
\draw[thick, black] (5.3,0) -- (5.3,-0.6);
\draw[thick, black] (6.1,0) -- (6.1,-0.6);
\node[draw] at (5.7,0.85) {$\pi$};
\node[draw] at (7.9,0.85) {$\sigma$};
			\end{tikzpicture}
            \quad
				\end{equation*}
 This establishes a bijection between the sets 
\[
\mathcal{P}=\left\{\tau\in\BNC(\chi):\tau\text{ satisfies properties (A) and (B)}\right\}\text{ and }\mathcal{Q}=\left\{(\pi,\sigma)\in\NC(\{1,2\})^2:\sigma\leq \Krew_{\NC}(\pi)\right\}.
\]
\end{example}
The next lemma explores the previous example in generality.
\begin{lemma}\label{lemma: partition bijection}
Let $m\in\N, \chi\in\{l,r\}^{4m}$, define
\[
F_1=\{\sx(1),\sx(4m)\}, F_i=\{\sx(4i-4),\sx(4i-3)\}\text{ and }G_i=\{\sx(4i-2),\sx(4i-1)\},\text{ for all }i=1,\ldots,m,
\]
and set
\[
F=\bigcup_{i=1}^m F_i\text{ and }G=\bigcup_{i=1}^m G_i.
\]
If $\mathcal{P}$ denotes the subset of $\BNC(\chi)$ consisting of those bi-non-crossing partitions $\tau$ that satisfy the following properties:
\begin{enumerate}[(A)]
    \item for all $1\leq i\leq m$, there exist blocks $V,V'\in\tau$ such that $F_i\subseteq V$ and $G_i\subseteq V'$,
    \item for $V\in\tau$,  either $V\subseteq F$ or $V\subseteq G$,
\end{enumerate}
and if 
\[
\mathcal{Q}=\left\{(\pi,\sigma)\in\NC(m)\times\NC(m): \sigma\leq \Krew_{\NC}(\pi)\right\},
\]
for $(\pi,\sigma)\in\mathcal{Q}$ we define a partition $\tau_{(\pi,\sigma)}$ by
\[
\tau_{(\pi,\sigma)}=\left\{\bigcup_{i\in U}F_i : U\in\pi\right\}\cup\left\{\bigcup_{i\in W}G_i: W\in\sigma\right\}.
\]
Then $\tau_{(\pi,\sigma)}\in\BNC(\chi)$ and the map $(\pi,\sigma)\mapsto\tau_{(\pi,\sigma)}$ is a bijection between the sets $\mathcal{P}$ and $\mathcal{Q}$.
\end{lemma}
\begin{proof}
Given $(\pi,\sigma)\in\mathcal{Q}$, it is clear that $\tau_{(\pi,\sigma)}$ is a partition which satisfies properties (A) and (B) above and, moreover, $\tau_{(\pi,\sigma)}$ is uniquely determined.  Arguing by way of contradiction suppose that $\tau_{(\pi,\sigma)}\notin\BNC(\chi)$, so there exist $V\neq V'\in\tau_{(\pi,\sigma)}$ and $v_1,v_2\in V, w_1,w_2\in V'$ such that
\[
v_1\prec_\chi w_1\prec_\chi v_2\prec_\chi w_2.
\]
Assume that $V=\bigcup_{i\in U}F_i$ and $V'=\bigcup_{i\in W}G_i$ for some $U\in\pi$ and $W\in\sigma$. Then there exist $i_1,i_2\in U$ and $j_1,j_2\in W$ such that $v_1\in F_{i_1}, v_2\in F_{i_2}, w_1\in G_{j_1}, w_2\in G_{j_2}$. The definition of the sets $F_i$ and $G_i$ together with the $\chi$-order relation above imply that $i_1\leq j_1<i_2\leq j_2$. If we consider $\sigma$ as a non-crossing partition on the set $\left\{\overline{1},\ldots,\overline{m}\right\}$, then $\pi\cup\sigma\notin\NC\left(\left\{1,\overline{1},\ldots,m,\overline{m}\right\}\right)$, which contradicts the fact that $\sigma\leq\Krew_{\NC}(\pi)$. We similarly arrive at a contradiction if we assume that both blocks $V,V'$ are either union of some $F_i$'s or some $G_i$'s and this shows that $\tau$ is indeed bi-non-crossing.

Conversely, if $\rho\in\BNC(\chi)$ satisfies properties (A) and (B), then every block of $\rho$ is either a union of some $F_i$'s or a union of some $G_i$'s, so that we may write $\rho=\{V_1,\ldots,V_n,V_1',\ldots,V'_k\}$, where each $V_j$ is a union of a subcollection of $\{F_1,\ldots,F_m\}$  and each $V'_j$ is a union of a subcollection of $\{G_1,\ldots,G_m\}$. Consequently, for each $1\leq j\leq n$ there exists $U_j\subseteq\{1,\ldots,m\}$ such that $V_j=\bigcup_{i\in U_j}F_i$. Since the blocks of the partition $\tau$ are clearly disjoint and their union is equal to the interval $\{1,\ldots,4m\}=F\cup G$, we see that
 the family $\pi=\{U_1,\ldots,U_n\}$ is a  partition on $\{1,\ldots,m\}$. Similarly, for each $1\leq j\leq k$ there exists $W_j\subseteq\{1,\ldots,m\}$ such that $V'_j=\bigcup_{i\in W_j}G_i$ and the family $\sigma=\{W_1,\ldots,W_k\}$ is a partition on $\{1,\ldots,m\}$. By the definitions of the sets $F_i$ and $G_i (i=1,\ldots,m)$ and since $\rho$ is bi-non-crossing, we easily see that both $\pi$ and $\sigma$ are non-crossing partitions such that $\pi\cup\sigma\in\NC\left(\left\{1,\overline{1},\ldots,m,\overline{m}\right\}\right)$, so that $\sigma\leq\Krew_{\NC}(\pi)$. Lastly, it is clear that $\rho=\tau_{(\pi,\sigma)}$ and this completes the proof.   
\end{proof}
We next proceed with a lemma that contains the central combinatorial argument required for  proving one of the main results of this section (see Theorem \ref{invdistr}). At a key point, it makes use of the cancellation property observed in Lemma \ref{techlemmaBNC}.
\begin{lemma}\label{lemmainv1}
Let $(A,\varphi)$ be a non-commutative $*$-probability space and $u_l, u_r, Z,W\in A$ such that:
\begin{enumerate}[(a)]
\item the pair $(u_l, u_r)$ is a bi-Haar unitary,
\item the pair $(Z,W)$ is $*$-even,
\item the pairs $(u_l, u_r)$ and $(Z, W)$ are $*$-bi-free.
\end{enumerate}
Let $m\in\N$, $\chi\in {\{l,r\}}^{2m}$ and $a_1,\ldots, a_{2m}\in A$ with
\[
a_i\in \begin{cases}
       \{u_l Z, Z^* {u_l}^*\}, &\text{if }  \chi (i)=l\\
        \{Wu_r, {u_r}^* W^*\}, &\text{if }  \chi (i) = r\\   
     \end{cases}\quad(i=1,\ldots,2m),
\]
such that the sequence $\left(a_{\sx(1)},\ldots,a_{\sx(2m)}\right)$
is $*$-alternating. Define $b_1,\ldots,b_{2m}\in A$ as follows:
\[
b_i = \begin{cases}
       Z, &\text{if }  a_i=u_l Z\\
        Z^*, &\text{if }  a_i=Z^* {u_l}^*\\
        W, &\text{if } a_i=W u_r\\
        W^*, &\text{if }  a_i={u_r}^* W^*\\
     \end{cases}\quad (i=1,\ldots,2m).
\]
Then, we have 
\[
\kappa_\chi (a_1,\ldots,a_{2m})  = \kappa_\chi (b_1,\ldots,b_{2m}).
\]
\end{lemma}
\begin{proof}
Let $m\in\N,\chi\in {\{l,r\}}^{2m}$ and $a_1,\ldots,a_{2m},b_1,\ldots,b_{2m}$ be given as in the statement of the lemma. Also let $\xh\in\{l,r\}^{4m},\oxh\in\BNC(\xh)$ and $c_1,\ldots,c_{4m}\in A$ be as  Notation \ref{notation: products}, so that by Theorem \ref{openup} we have
\[
\kappa_\chi(a_1,\ldots,a_{2m}) =  \sum_{\substack{\tau\in \BNC(\xh)\\
\tau\vee\oxh = 1_{\xh}\\}}
        \kappa_{\xh,\tau}(c_1,\ldots,c_{4m})=  \sum_{\substack{\tau\in \BNC(\xh)\\
\tau\vee\oxh = 1_{\xh}\\}}
\prod_{V\in\tau}       \kappa_{{\xh}|_V}\left({(c_1,\ldots,c_{4m})}|_{V}\right).\tag{1}
\]
Since the sequence $\left(a_{\sx(1)},\ldots,a_{\sx(2m)}\right)$
was assumed to be $*$-alternating,  both of the sequences
\[
\left(c_{\sxh(1)},c_{\sxh(4)},c_{\sxh(5)},\ldots, c_{\sxh(4m-4)},
c_{\sxh(4m-3)}, c_{\sxh(4m)}\right),
\]
and
\[
\left(c_{\sxh(2)},c_{\sxh(3)},
c_{\sxh(6)},c_{\sxh(7)},\ldots, c_{\sxh(4m-2)},
c_{\sxh(4m-1)}\right),
\]
are also $*$-alternating (observe that for any $i\in\{1,\ldots,2m\}$, the element $a_{\sx(i)}$ corresponds to a $*$-term if and only if both the elements $c_{\sxh(2i-1)}$ and $c_{\sxh(2i)}$ correspond to $*$-terms). Define
\[
F_1=\left\{\sxh(1),\sxh(4m)\right\}, F_i=\left\{\sxh(4i-4),\sxh(4i-3)\right\}\text{ and }G_i=\left\{\sxh(4i-2),\sxh(4i-1)\right\},\text{ for all }i=1,\ldots,m,
\]
and set
\[
F=\bigcup_{i=1}^m F_i\text{ and }G=\bigcup_{i=1}^m G_i.
\]
We claim that either $F$ is equal to the set of all indices $1\leq j\leq 4m$ such that the operator $c_j$ is in the set $\{u_l,u_l^*,u_r,u_r^*\}$ and $G$ is the set of all indices such that $c_j$ is in $\{Z,Z^*,W,W^*\}$, or vice versa. Indeed, begin by assuming that $a_{\sx(1)}=u_lZ$. Since the sequence $(a_{\sx(1)},\ldots,a_{\sx(n)})$ is $*$-alternating, we must have that $a_{\sx(2)}\in\{Z^*u_l^*,u_r^*W^*\}$. If $a_{\sx(2)}=Z^* u_l^*$, then for the $\xh$-order it follows that 
\[
c_{\sxh(1)}=u_l, c_{\sxh(2)}=Z, c_{\sxh(3)}=Z^*\text{ and }c_{\sxh(4)}=u_l^*,
\]
while if $a_{\sx(2)}=u_r^*W^*$, then for the $\xh$-order we have 
\[
c_{\sxh(1)}=u_l, c_{\sxh(2)}=Z, c_{\sxh(3)}=W^*\text{ and }c_{\sxh(4)}=u_r^*,
\]
hence in both cases we see that $\{c_{\sxh(2)},c_{\sxh(3)}\}\subseteq \{Z,Z^*,W,W^*\}$. A straightforward induction argument then shows that for all $i=1,\ldots,m$ and $j\in G_i$, one has $c_j\in\{Z,Z^*,W,W^*\}$. Of course, this also implies that the set $F$  must be equal to the set of all indices that correspond to elements in $\{u_l,u_l^*,u_r,u_r^*\}$. It clearly follows that similar arguments yield an analogous outcome in the case when $a_{\sx(1)}\in\{Z^*u_l^*,Wu_r,u_r^*W^*\}$. As a result, without loss of generality we may assume that
\[
F=\left\{j : c_j\in\left\{u_l,{u_l}^*,u_r,{u_r}^*\right\}\right\}\text{ and }G=\left\{j: c_j\in\left\{Z, Z^*, W, W^*\right\}\right\},
\]
with the remaining case  handled similarly. Observe that
\[
\xh|_G=\chi\text{ and }\left(c_1,\ldots,c_{4m}\right)|_G=(b_1,\ldots,b_{2m}),
\]
where $b_1,\ldots,b_{2m}$ are as in the statement of the lemma. 
Since the pairs $(u_l,u_r)$ and $(Z,W)$ are $*$-bi-free, by Proposition \ref{prop: reduction properties} in order for a bi-non-crossing partition $\tau\in\BNC(\xh)$ to contribute to the sum appearing in $(1)$, we must have  for all $V\in\tau$ that  either
$\{c_i :i\in V\}\subseteq\{u_l,{u_l}^*,u_r,{u_r}^*\}$ or $\{c_i :i\in V\}\subseteq\{Z,Z^*,W,W^*\}$, which shows that if $V\in\tau$, then either $V\subseteq F$ or $V\subseteq G$. In addition, every block of such a partition $\tau$ must contain an even number of elements, since both pairs $(u_l,u_r)$ and $(Z,W)$ are $*$-even.  Furthermore, any bi-non-crossing partition $\tau\in\BNC(\chi)$ which contributes to the sum of cumulants also satisfies that $\tau\vee\oxh=1_{\xh}$ by equation $(1)$. The properties of such a contributing partition $\tau$ are  summarized in the next two points: 
\begin{enumerate}[(A)]
\item $\sxh (1) {\sim}_{\tau} \sxh (4m)$ and $\sxh (2i) {\sim}_{\tau}\sxh(2i+1)$ for every $i=1,\ldots,2m-1$ (this follows by an application of Proposition \ref{evenblocks}),
\item for $V\in\tau$, either $V\subseteq F$ or $V\subseteq G$.
\end{enumerate}
By Lemma \ref{lemma: partition bijection} the partitions which satisfy properties (A) and (B) above are in bijection with the set
\[
\mathcal{Q}=\left\{(\pi,\sigma)\in\NC(m)\times\NC(m):\sigma\leq\Krew_{\NC}(\pi)\right\},
\]
via the map $(\pi,\sigma)\mapsto\tau_{(\pi,\sigma)}$ given by
\[
\tau_{(\pi,\sigma)}=\left\{\bigcup_{i\in U}F_i:U\in\pi\right\}\cup\left\{\bigcup_{i\in W}G_i: W\in\sigma\right\}.
\]
For ease of notation, for $\pi,\sigma\in\NC(m)$  we define
\[
F_U=\bigcup_{i\in U} F_i\text{ and } G_W=\bigcup_{i\in W}G_i,\text{ for all }U\in\pi\text{ and }W\in\sigma,
\]
and we also let
\[
h_\pi=\prod_{U\in\pi}\kappa_{\xh|_{F_U}}\left((c_1,\ldots,c_{4m})|_{F_U}\right)\text{ and }d_\sigma=\prod_{W\in\sigma}\kappa_{\xh|_{G_W}}\left((c_1,\ldots,c_{4m})|_{G_W}\right).
\]
Observe that in the case when $\sigma=1_m$, we obtain
\[
d_{1_m}=\kappa_{\xh|_G}\left((c_1,\ldots,c_{4m})|_G\right)=\kappa_\chi(b_1,\ldots,b_{2m}).
\]
Under the aforementioned identification the sum appearing in (1) becomes:
\[
\sum_{\substack{\tau\in \BNC(\xh)\\
\tau\vee\oxh = 1_{\xh}\\}} \prod_{V\in\tau}       \kappa_{{\xh}|_V}\left({(c_1,\ldots,c_{4m})}|_{V}\right)\ = \sum_{\substack{\tau_{(\pi,\sigma)}\in \BNC(\xh)\\
\pi,\sigma\in\NC(m)\\ \sigma\leq \Krew_{\NC}(\pi)\\}} h_{\pi}\cdot d_{\sigma}   =  \sum_{\pi\in\NC(m)} h_{\pi}\cdot \left( \sum_{\substack{\sigma\in\NC(m)\\
\sigma\leq \Krew_{\NC}(\pi) \\}}d_{\sigma}\right).\tag{2}
\]
For a fixed $\pi\in\NC(m)$, we will compute the value of $h_{\pi}$. Since $F$ is equal to the set of all indices $1\leq j\leq 4m$ such that $c_j$ is in $\{u_l,u_l^*,u_r,u_r^*\}$,
it follows that for all $U\in\pi$, the bi-free cumulant
$\kappa_{\xh|_{F_U}}\left((c_1,\ldots,c_{4m})|_{F_U}\right)$
has entries in the set $\{u_l,u_l^*,u_r,u_r^*\}$ and the sequence 
$\left(c_1,\ldots,c_{4m}\right)|_{F_U}$ is $*$-alternating  when read in the induced $\xh|_{F_U}$-order. Moreover, notice that the cardinality of the set $F_U=\bigcup_{i\in U}F_i$ is equal to two times the cardinality of $U$, since each $F_i$ contains exactly two indices. Thus, by a combination of Corollary \ref{bihaarcum} and Lemma \ref{lemmamobius} we obtain
\[
h_{\pi} = \prod_{U\in\pi} {(-1)}^{|U|-1}\cdot C_{|U|-1} = \mu_{\NC}(0_n,\pi),
\]
therefore equation (2) yields 
\[
\kappa_{\chi}(a_1,\ldots,a_{2m}) = 
\sum_{\pi\in\NC(m)} \mu_{\NC}(0_n,\pi)\cdot \left( \sum_{\substack{\sigma\in\NC(m)\\
\sigma\leq \Krew_{\NC}(\pi) \\}}d_{\sigma}\right)=d_{1_m}=\kappa_\chi(b_1,\ldots,b_{2m}),
\]
where the penultimate equality follows by Lemma \ref{techlemmaBNC}. 
\end{proof}
We are now in a position to state the following theorem (which is the generalization of \cite[Theorem 15.10]{nicaspeicher} to the bi-free setting), regarding the invariance of the joint $*$-distribution of a bi-R-diagonal pair under the multiplication by a $*$-bi-free bi-Haar unitary pair.

\begin{theorem}\label{invdistr}
Let $(A,\varphi)$ be a non-commutative $*$-probability space and  $u_l, u_r, X, Y\in A$  such that:
\begin{enumerate}[(a)]
\item  the pair $(u_l, u_r)$ is a bi-Haar unitary,
\item the pairs $(u_l, u_r)$ and  $(X,Y)$ are $*$-bi-free.
\end{enumerate}
Then, the following are equivalent:
\begin{enumerate}[(i)]
\item the pair $(X, Y)$ is bi-R-diagonal,
\item the joint $*$-distribution of the pair $(X, Y)$ coincides with the joint $*$-distribution of  $(u_l X, Y u_r)$.
\end{enumerate}
\end{theorem}
\begin{proof}
By Theorem \ref{prodbir-bifree} the pair $(u_l X, Y u_r)$ is bi-R-diagonal and, since equality of joint $*$-distributions is equivalent to the equality of bi-free $*$-cumulants, it follows that the pair $(X,Y)$ is also bi-R-diagonal. This yields the implication $(ii)\Rightarrow (i)$. For the converse, we will show the equality of all bi-free $*$-cumulants involving the pairs $(X,Y)$ and $(u_l X,Yu_r)$. Since $(X,Y)$ is bi-R-diagonal, all bi-free cumulants with entries in $\{X,X^*,Y,Y^*\}$ that are either of odd order or that are not $*$-alternating in the $\chi$-order must vanish. The same applies to the pair $(u_l X,Yu_r)$ since it is also bi-R-diagonal by Theorem \ref{prodbir-bifree}. Therefore, it is enough to show that for all even numbers $n\in\N, \chi\in {\{l,r\}}^n$ and $a_1,\ldots,a_n\in A$ with
\[
a_i\in \begin{cases}
       \{u_l X,X^*{u_l}^*\}, &\text{if }\chi (i)=l\\
        \{Yu_r,{u_r}^*Y^*\}, &\text{if }  \chi (i) = r\\   
     \end{cases}\quad (i=1,\ldots,n),
\]
such that the sequence $(a_{\sx(1)},\ldots,a_{\sx(n)})$ is $*$-alternating, by setting
\[
b_i = \begin{cases}
       X, &\text{if }  a_i=u_l X\\
        X^*, &\text{if }  a_i=X^* {u_l}^*\\
        Y, &\text{if }  a_i=Y u_r\\
        Y^*, &\text{if }  a_i={u_r}^* Y^*\\
     \end{cases}\quad (i=1,\ldots,n),
\]
we have that
\[
\kappa_\chi (a_1,\ldots,a_n)  = \kappa_\chi (b_1,\ldots,b_n),
\]
which is exactly what an application of Lemma \ref{lemmainv1}  yields.
\end{proof}

We remark that the conclusion of the previous theorem no longer holds if the order of the multiplication of the right operators is not reversed, as the following example indicates.
\begin{example}
Let $(A,\varphi_1), (B,\varphi_2)$ be two $\mathrm{C}^*$-probability spaces and let $u_l,u_r\in A, v_l, v_r\in B$ such that both pairs $(u_l,u_r)$ and $(v_l,v_r)$ are bi-Haar unitaries. We may find a larger  $\mathrm{C}^*$-probability space $(C,\psi)$ such that the pairs $(u_l,u_r)$ and $(v_l,v_r)$ are $*$-bi-free in $(C,\psi)$  and such that $\psi$ preserves the joint $*$-distributions of the pairs $(u_l,u_r)$ and $(v_l,v_r)$ with respect to $\varphi_1$ and $\varphi_2$ respectively (see Remark \ref{remark:propertiesofbi-freepairs}). Both the pairs $(u_l, u_r)$ and $(v_l, v_r)$ are clearly bi-R-diagonal in $(C,\psi)$. However, the joint $*$-distribution of the pair $(v_l,v_r)$ does not coincide with the joint $*$-distribution of $(u_l v_l,u_r v_r)$.  Indeed, first note that
\[
\kappa_{(l,r)}(v_l,v_r^*) = \psi(v_l\cdot v_r^*)=1.
\]
By an application of Theorem \ref{openup}, we obtain
\[
\kappa_{(l,r)}(u_lv_l,v_r^* u_r^*)=\sum_{\substack{\tau\in \BNC(\xh)\\
\tau\vee\oxh = 1_{\xh}\\}}
        \kappa_{\xh,\tau}(u_l, v_l, v_r^*, u_r^*),
\]
where $\xh^{-1}(l)=\{1,2\}$ and  $\oxh=\{\{1,2\},\{3,4\}\}$. Since each of the four terms appearing as entries in the sum of cumulants above is centred (i.e. $\kappa_l(u_l)=\kappa_l(v_l)=\kappa_r(v_r^*)=\kappa_r(u_r^*)=0$), it is enough to consider bi-non-crossing partitions $\tau\in \BNC (\xh)$ with $\tau\vee\oxh=1_{\xh}$ and all of whose blocks are not singletons. These are the following two bi-non-crossing partitions:
\[
\tau_1=\{\{1,3\},\{2,4\}\}\text{ and }\tau_2=\{\{1,2,3,4\}\}=1_{\xh},
\]
thus we obtain
\[
\kappa_{(l,r)} (u_l v_l,v_r^* u_r^*)=\kappa_{\xh,\tau_1}(u_l, v_l, v_r^*, u_r^*) + \kappa_{\xh,\tau_2}(u_l, v_l, v_r^*, u_r^*)=\kappa_{\xh|_{\{1,3\}}}(u_l,v_r^*)\cdot\kappa_{\xh|_{\{2,4\}}}(v_l, u_r^*)+\kappa_{\xh} 
 (u_l, v_l, v_r^*, u_r^*).
\]
However, since $(u_l,u_r)$ and $(v_l,v_r)$ are $*$-bi-free, all of the mixed bi-free cumulants appearing on the right-hand side of the equation above must be equal to zero, hence $\kappa_\chi (u_l v_l,v_r^* u_r^*)=0$. Since
\[
\kappa_{(l,r)} (v_l, v_r^*)\neq\kappa_{(l,r)} (u_l v_l,v_r^* u_r^*),
\]
the joint $*$-distributions of the pairs $(v_l, v_r)$ and $(u_l v_l, u_r v_r)$ do not coincide.\qed
\end{example}

Gathering the results of this section, one can obtain a theorem similar to \cite[Theorem 1.2]{nicaspeichersh} (and \cite[Theorem 3.1]{dykema} for the operator-valued setting).

\begin{theorem}\label{together}
Let $(A,\varphi)$ be a non-commutative $*$-probability space and $X,Y\in A$. The following are equivalent:
\begin{enumerate}[(i)]
\item the pair $(X,Y)$ is bi-R-diagonal,
\item there exists an enlargement \footnote{An \emph{enlargement} of $(A,\varphi)$ is a non-commutative 
 probability space $(\tilde{A},\tilde{\varphi})$ such that $A\subseteq\tilde{A}$ and ${\tilde{\varphi}}|_A=\varphi$.} $(\tilde{A},\tilde{\varphi})$ of $(A,\varphi)$ and $u_l,u_r\in\tilde{A}$ such that
\begin{enumerate}[(a)]
\item the pair $(u_l,u_r)$ is a bi-Haar unitary,
\item the pairs $(u_l,u_r)$ and $(X,Y)$ are $*$-bi-free,
\item the joint $*$-distribution of the pair $(u_l X,Yu_r)$ coincides with the joint $*$-distribution of $(X,Y)$,
\pushcounter
\end{enumerate}
\item for any enlargement $(\tilde{A},\tilde{\varphi})$ of $(A,\varphi)$ and any $u_l,u_r\in\tilde{A}$ such that
\begin{enumerate}[(d)]
\popcounter
\item the pair $(u_l,u_r)$ is a bi-Haar unitary,
\item[(e)] the pairs $(u_l,u_r)$ and $(X,Y)$ are $*$-bi-free,
\end{enumerate}
one has that the the joint $*$-distribution of the pair $(u_l X,Yu_r)$ coincides with the joint $*$-distribution of $(X,Y)$,
\item[(iv)] consider the unital subalgebras $\mathcal{M}_2(\C)$ and $\mathcal{D}_2$ of $\mathcal{M}_2(A)$ consisting of scalar matrices and diagonal scalar matrices respectively and let the maps
\[
\varepsilon:\mathcal{D}\otimes {\mathcal{D}}^{\text{op}}\rightarrow\mathcal{L}(\mathcal{M}_2(A)), L,R:\mathcal{M}_2(A)\rightarrow \mathcal{L}(\mathcal{M}_2(A)), E_2:\mathcal{L}(\mathcal{M}_2(A))\rightarrow \mathcal{M}_2(\C),
\]
and
\[
F_2:\mathcal{M}_2(\C)\rightarrow\mathcal{D}_2,
\]
be as in section \ref{sctn:bbbb}. Also, in $\mathcal{M}_2(A)$ consider the pair $(Z,W)$ defined as
\[
Z = \begin{bmatrix}
0 & X\\ X^* &0
\end{bmatrix}\text{ and }W = \begin{bmatrix}
0 & Y\\ Y^* &0
\end{bmatrix}.
\]
Then,  the pair $(L(Z),R(W))$ is bi-free from $(L(\mathcal{M}_2(\C)),R({\mathcal{M}_2(\C)}^{\text{op}}))$ with amalgamation over $\mathcal{D}_2$ with respect to $F_2\circ E_2$.
\end{enumerate}
\end{theorem}

\begin{proof}
The equivalence of $(i)$ and $(iii)$, as well as the implication $(ii)\Rightarrow (i)$ both follow from Theorem \ref{invdistr}. Also, the equivalence of $(i)$ and $(iv)$ is a result of Proposition \ref{rcyclic} and Theorem \ref{thmrcyclic}. To see that $(i)$ implies $(ii)$,  consider a non-commutative $*$-probability space $(B,\psi)$ containing a bi-Haar unitary pair $(u_l,u_r)$ and let $(\tilde{A},\tilde{\varphi})$ be an enlargement of both $(A,\varphi)$ and $(B,\psi)$ such that the pairs $(u_l,u_r)$ and $(X,Y)$ are $*$-bi-free in $(\tilde{A},\tilde{\varphi})$. Then again by an application of Theorem \ref{invdistr} the joint $*$-distribution of the pair $(u_l X,Yu_r)$ must coincide with the joint $*$-distribution of $(X,Y)$.
\end{proof}
\section*{Acknowledgements}
The author would like to thank Professor Paul Skoufranis for numerous valuable suggestions and discussions during the development of this manuscript, Professor James A. Mingo for 
 helpful remarks, especially regarding the state of Theorem \ref{birpowers} and Proposition \ref{prop:bi-R-diagpowerscounterexample}, as well as the anonymous referee for  recommendations   that greatly improved the overall presentation of the paper. This research was supported by the Fundamental Research Funds for the Central Universities Grant 3072024CFJ2404.

\bibliographystyle{alpha}
\bibliography{Bibliography}

\end{document}